\documentclass[a4paper]{amsart}
\usepackage{amsthm}
\usepackage{amssymb}
\usepackage{amsmath}

\usepackage{stmaryrd}
\usepackage{amsfonts}
\usepackage{lmodern}
\usepackage{mathtools}
\usepackage{pdfpages}
\usepackage{changes}
\usepackage{exscale,cite,color,amsopn}
\usepackage[colorlinks=true,pdftex,unicode=true,linktocpage,bookmarksopen,hypertexnames=true]{hyperref}
\usepackage{aliascnt}
\usepackage{colortbl}
\usepackage{enumitem}
\usepackage{verbatim}
\usepackage{tikz-cd}
\usepackage[capitalize]{cleveref}
\usepackage{orcidlink}
\usepackage{xcolor}
\usepackage{mathdots}
\usepackage{mathrsfs}

\usetikzlibrary{trees}

\renewcommand{\leq}{\leqslant}
\renewcommand{\geq}{\geqslant}
\DeclareMathOperator{\id}{id}

\DeclareMathOperator{\Triv}{Triv}

\DeclareMathOperator{\Ret}{Ret}

\DeclareMathOperator{\Sym}{Sym}

\DeclareMathOperator{\FSB}{FSB}
\DeclareMathOperator{\FB}{FB}
\DeclareMathOperator{\FSol}{FSol}
\DeclareMathOperator{\Fg}{F}
\DeclareMathOperator{\Fm}{S}
\DeclareMathOperator{\Inj}{Inj}
\DeclareMathOperator{\Inv}{Inv}

\DeclareMathOperator{\eva}{ev}

\newcommand{\N}{\mathbb{N}}
\newcommand{\Z}{\mathbb{Z}}
\newcommand{\F}{\mathbb{F}}
\newcommand{\Q}{\mathbb{Q}}

\newcommand{\Aut}{\operatorname{Aut}}
\newcommand{\Hom}{\operatorname{Hom}}
\newcommand{\Hol}{\operatorname{Hol}}
\newcommand{\End}{\mathrm{End}}
\newcommand{\fab}{\operatorname{F}^{\operatorname{ab}}}

\makeatletter
\numberwithin{equation}{section}
\numberwithin{figure}{section}
\numberwithin{table}{section}
\newtheorem{thm}{Theorem}[section]
\newtheorem*{thm*}{Theorem}
\newtheorem{lem}[thm]{Lemma}
\newtheorem{cor}[thm]{Corollary}
\newtheorem{pro}[thm]{Proposition}

\theoremstyle{definition}
\newtheorem{defn}[thm]{Definition}

\newtheorem{question}[thm]{Question}

\newtheorem{rem}[thm]{Remark}
\newtheorem{exa}[thm]{Example}

\newtheorem*{convention*}{Convention}

\makeatother

\title{Free Skew Braces and Free Solutions of the Yang--Baxter Equation}

\author[\tiny Jespers]{Eric Jespers\orcidlink{0000-0002-2695-7949}}
\address{Department of Mathematics and Data Science\\
Vrije Universiteit Brussel\\
Pleinlaan 2\\
1050 Brussels\\
Belgium}
\email{eric.jespers@vub.be}

\author[\tiny Letourmy]{Thomas Letourmy\orcidlink{0000-0003-4769-1294}}
\address{Faculté des Sciences\\
Université libre de Bruxelles\\
Campus de la Plaine - CP 216\\
Boulevard du Triomphe, ACC.2\\
1050 Brussels\\
Belgium}
\email{thomas.letourmy@ulb.be}

\author[\tiny Properzi]{Silvia Properzi\orcidlink{0000-0002-8228-9156}}
\address{Department of Mathematics and Data Science\\
Vrije Universiteit Brussel\\
Pleinlaan 2\\
1050 Brussels\\
Belgium}
\email{silvia.properzi@vub.be}

\author[\tiny Trombetti]{Marco Trombetti\orcidlink{0000-0003-4532-3690}}
\address{
Dipartimento di Matematica e Applicazioni "Renato Caccioppoli" \\
Università degli Studi di Napoli Federico II\\
Via Cinthia 26\\
80126 Naples\\
Italy}

\email{marco.trombetti@unina.it}

\author[\tiny Van Antwerpen]{Arne Van Antwerpen\orcidlink{0000-0001-7619-6298}}
\address{Department of Mathematics and Statistics\\
Maynooth University\\
Maynooth\\
Ireland}
\email{Arne.VanAntwerpen@mu.ie}

\subjclass{Primary 16T25; Secondary 08B20, 20N99}

\keywords{Yang--Baxter Equation, solution, skew brace, free object}

\date{\today}

\begin{document}

\begin{abstract}
We offer a workable construction of the free right nilpotent skew braces of arbitrary class which allows us to prove (among many other things) that this free object has free additive/multiplicative groups, and that it must also be residually finite and Hopfian. We introduce the class of right nilpotent solutions, which correspond to right nilpotent skew braces. As a consequence of our construction, the free solutions in this class have a solvable Word Problem, and every law holding for finite solutions of the previous type also holds for {\it every} solution of the same type.

In the remainder of the paper, we present further explicit realizations of free objects and explore their consequences. Among these are free two-sided skew braces of abelian type (with an abelian multiplicative group) and free centrally nilpotent skew braces of class $2$.
\end{abstract}

\maketitle

\section{Introduction}

Free objects are among the most basic and important tools in algebra. They can be thought of as ‘‘generic’’ algebraic structures (think of groups) satisfying certain ‘‘test’’ properties (think of nilpotency of bounded class), so, for example, among other things, they may serve as a testing ground for verifying conjectures. The first free objects ever studied are the free groups --- these were introduced by Walther von Dyck \cite{MR1510147} in 1882, although their name and basic properties are due to Jakob Nielsen in 1924. Soon people realized that free groups were actually a very particular case of certain universal constructions, and in 1948 the theory of free objects started thanks to Saunders MacLane (see \cite{MR49192}, where the term ‘‘free’’ has even been given a very curious political context). 

The study of free objects has shown for example that in certain contexts if a ‘‘law’’ holds in every finite object, then it also holds in {\it every} object. This is precisely what happens for free groups, which have been proved to be residually finite. We cannot stress enough the importance of being residually finite in group theory. As a consequence of this property, one has that free groups of finite rank are Hopfian. Moreover, residual finiteness can usually be exploited to solve the Word Problem and other Dehn's Problems in several classes of groups, or at least to provide an algorithmic approach to them (see for example \cite{MR635200}). Furthermore, it provides ground for certain topological approaches, and it is something even relevant in Geometric Group Theory: one of the long-standing questions in this area asks if hyperbolic groups are residually finite or not.

\smallskip

Of course, the amount of information a free object encodes ultimately depends on the way it is constructed. There are general results in universal algebra asserting that free objects always exist, but their mere existence provides little information about their structure beyond the usual universal property. For example, the property of free groups to be residually finite does not follow at once from the construction but it exploits the universal property and certain explicit constructions of finite groups. Moreover, although free groups are very well understood thanks to their clear construction (using which one can for instance prove that subgroups of free groups are themselves free --- the Nielsen--Schreier Theorem), the same cannot be said about many other types of relevant  structures, especially when these structures cannot be really described in terms of neat {\it algebraic} structures. 

\medskip

The main motivation for the present work is indeed to tackle one of these non-easy frameworks. We study free bijective non-degenerate set-theoretic solutions (solutions, for short) to the Yang--Baxter Equation (YBE, for short). This is a consistency equation (independently obtained by the physicists Yang \cite{Yang} and Baxter \cite{Baxter}) in the field of quantum statistical mechanics which has many relevant interpretations in the realm of mathematical physics, and which plays a key role in the foundation of quantum groups. A major problem in mathematical physics is the classification of all (set-theoretic) solutions to the YBE, an objective that, at present, remains far beyond reach. Having a nice description of a ‘‘free solution’’ would thus be a good starting point in solving this major problem. Not only because every other solution would be a homomorphic image of it, but also because it would  provide a rigorous framework to reduce 
certain types of questions on solutions to finite ones.

\bigskip

\noindent{\bf Conjecture ($\star$)}\quad
Finite solutions of the YBE (in a given category) determine the law for all other  solutions (in that category).

\bigskip

It is now established that all solutions to the YBE can be studied by means of the relatively new algebraic structures, called {\it skew braces} (in a nutshell, these are two group structures on the same underlying set linked by a skew distributivity relation).
In fact, by a result of Bachiller \cite{Bach}, 
there is a method to construct all the solutions to the YBE from skew braces.

More details on skew braces and the YBE are provided in~Section \ref{preliminaries}, but for the sake of the introduction, we just notice that every solution of the YBE gives a skew brace (sometimes called the {\it structure} skew brace), and, conversely, that every skew brace is the support of a solution. Clearly, this correspondence is not bijective, and an even subtler issue may arise. Indeed, there is a natural morphism from a solution to its structure skew brace, and this morphism need not be injective (solutions for which it is injective are called \emph{injective}). For this reason, one often considers the \emph{injectivisation} of a solution, obtained by identifying precisely those elements that have the same image in the structure skew brace. In this sense, injective solutions are ``minimal'' among the solutions giving rise to the same structure skew brace.

\smallskip

As noted in \cite{ChiMil} by Chirvasitu and Militaru, it follows from classical results in universal algebra and adjunctions in categories that free solutions and free skew braces exist. In the spirit of that paper, the following result can be viewed as a consequence of Freyd's theorem on adjunctions and motivates to further study the structure of these free objects.

\bigskip

\noindent{\bf Theorem A}\;(see \cref{freesolutionisfreeskew})\\ 
{\it The structure skew brace of a free solution is free.} 
\bigskip

Determining whether a skew brace gives rise to a free solution is much more subtle. Indeed, under the natural correspondence described above, there is no converse to Theorem~A (see \cref{freeskewbracenotfreesolution}). Another possibility would be to find a copy of the free solution inside the free skew brace. This would certainly be possible if the free solution were injective. Unfortunately, determining whether the free solution is injective seems to be a difficult problem (see \cref{questionquestion}). The difficulty lies in the fact that the passage from solutions to skew braces may lose information. This phenomenon already occurs in the natural category of solutions considered in Theorem~B, where the corresponding free objects are non-injective (see \cref{solnotinjective}). For these reasons, it is natural to consider the injectivisation of the free solution, which turns out to be the free injective solution~(see \cref{injectivisation}). 

\bigskip

Constructions of the free skew braces have been given in \cite{Orza} and \cite{Pompili2022}.
However, these descriptions do not seem to reveal the
structural properties of the resulting skew braces in a transparent way.
Is a free skew brace residually finite? Which kind of groups are the additive and multiplicative groups of a free skew brace? Are they free groups, or at least torsion-free? Does an analogue of the Nielsen--Schreier Theorem for groups hold for free skew braces? (In other words, are sub-skew braces free?)\ldots At present, a satisfactory description of free skew braces remains elusive. By The\-o\-rem~A, together with the established correspondence between solutions and skew braces, even obtaining a satisfactory description of free solutions appears to be an extremely difficult task. In particular, Conjecture ($\star$) remains open.

The difficulty of this task can already be appreciated by comparing
one-gen\- erated groups and rings with one-generated skew braces. While the former are easily dealt with,  the latter are as difficult to deal with as arbitrary skew braces. The obstruction here is that the generators may produce elements that are independent (in the additive and/or the multiplicative group)
to the original generators, and therefore cannot be expected to admit any prescribed form; for example, this makes it very difficult to understand if any two elements in a one-generated skew brace additively/multiplicatively commute or not.

The relevance of the problem of providing a workable description of free skew braces and free solutions is also underscored by the notion of a presentation for skew braces introduced in \cite{TrombettiFP}, and which is entirely based on the concept of free skew braces. Even without an operative description of these free objects, skew brace presentations have made it possible to solve some relevant problems concerning skew braces and the YBE (see \cite{MR4718182},\cite{Wordproblem}), but there are other problems which, to be solved, would require a complete understanding of the additive/multiplicative group of a free skew brace.

\medskip

In light of the relevance of this problem and its intrinsic difficulties, our approach shifted from the study of general free objects to that of free objects in specific, hopefully more manageable, categories.

One of the most relevant types of solutions to the YBE is the so-called {\it multipermutation} solutions, which are  solutions that can be retracted into the trivial solution over a singleton after finitely many natural identification steps. These solutions received a lot of attention in recent years (see for example \cite{MR4627847}) and can be characterized precisely as those solutions whose structure skew brace has a nilpotent additive group and satisfies a nilpotency concept for skew braces called ‘‘right nilpotency’’ (this essentially means that the skew brace is an iterated extension of groups). Clearly, the nilpotency condition on the additive group makes things a little more well-behaved, but here we are concerned with the broader category of solutions corresponding to the class of right nilpotent skew braces. Our second main result links free solutions of right-class $n$ (denoted by $\mathcal{RNS}_n$) with the corresponding free skew braces. 

\bigskip

\noindent{\bf Theorem B}\;(see Theorem \ref{StrGrpFSBn})\\ 
{\it The structure skew brace of a free solution in the category $\mathcal{RNS}_n$ is free right nilpotent of class $n$.}
\bigskip

This is achieved through a general, though rather technical, formula translating right nilpotency of the structure skew brace into identities on the originating solution (see Lemma~\ref{lemmmino-class-n}). Motivated by this connection, our
third main result gives an explicit description of free right nilpotent skew braces of arbitrary class. This is the most substantial part of the paper, and its proof occupies a significant portion of the manuscript.

\bigskip

\noindent{\bf Theorem C}\;(see Section \ref{sec: right nilp}, and in particular Theorem \ref{FreeNilpn})\\ 
{\it A \emph{workable} construction of free right nilpotent skew braces.}

\bigskip

In the previous statement the adjective ‘‘workable’’ is fundamental. In fact, as we have already noticed, the value of such constructions depends only on the consequences they have. We are now going to highlight some of the main consequences of Theorem C.

\begin{itemize}
    \item[(1)] Our construction shows that the additive and multiplicative groups of such a free object must be both isomorphic to a free group of infinite rank. In turn, this implies that the corresponding free object of {\it abelian type} (that is, with an abelian additive group) does not have isomorphic additive and multiplicative groups (see~Re\-mark~\ref{remarkrn22}). We also note that this description allows us to prove that in right nilpotency class~$n$,  the last non-zero term of the right nilpotent series is a free group (see~\cref{quotient-by-last-term}), and we also have a description of a free basis. But more than that, it implies that elements in the free skew brace of right nilpotency class at most $n$ have canonical forms. Moreover, we identify a second free basis of the additive group (see \cref{freebasisy}), which is particularly well-suited to construct quotient skew braces. 

    \smallskip

\item[(2)] In general, free right nilpotent skew braces are not co-Hopfian (see~Proposition \ref{notcohop})

\smallskip

\item[(3)] Free right nilpotent skew braces of class $n$ are residually finite and Hopfian (see \cref{classnresfinhopfia}).

\smallskip

\item[(4)] We can identify the regular subgroup corresponding to free right nilpotent skew braces of class $2$ on one element (see Proposition \ref{descriptionregular}). Recall that skew braces can also be identified with regular subgroups of the automorphism group of their additive group.

The reader should not be misled by the words ‘‘on one element’’. In fact, such structures may be as complicated as those with more than one element. For example, a construction of the free right nilpotent skew brace {\it of abelian type} on one element has been given in \cite{BallesterEstebanKurdachenkoPerez2025}, but from that none of the above-mentioned consequences follow directly.
    
\end{itemize}

\smallskip

In turn, Point (3) together with Theorem B and The\-o\-rem~\ref{injectivisation} have an obvious consequence which we wish to highlight more as follows.

\bigskip

\noindent{\bf Theorem D}\\
{\it The free injective solution in the category $\mathcal{RNS}_n$ satisfies~\emph{Conjecture~$(\star)$.}}

\bigskip

We even have more than that. In fact, combining Points (1) and (3), we have that the obvious Word Problem for free right nilpotent skew braces of class $2$ is solvable.

\bigskip

\noindent{\bf Theorem E}\\
{\it The free injective solution in the category $\mathcal{RNS}_n$ has a solvable Word Pro\-blem.}

\bigskip

We should mention that there have been other attempts to describe very particular free objects in certain skew brace categories: \cite{MR4950409} deals with free {\it left}-nilpotent skew braces of class $2$ of abelian type (see also \cite{1genln2Kurd}), \cite{LetFreeComm} with free skew braces with an abelian multiplicative group, and \cite{KN1} with free skew braces of abelian type with an abelian multiplicative group. 

To provide a broader perspective on free constructions in the categories of
skew braces and solutions, we include several supplementary sections devoted to more conclusive investigations of other relevant categories of skew braces. We briefly summarize the content of these sections as follows (for a more comprehensive account the reader should check the corresponding sections):

\begin{itemize}
    \item Section \ref{twosided}: Free two-sided skew braces of abelian type on a finite set. These are residually finite (see Corollary \ref{cortwosidedbrace}). In this context, we should also note the relevant Proposition \ref{two sided brace abelian generated is abelian }, which shows that the multiplicative group of a one-generator two-sided skew brace is abelian, so in this case, the above-mentioned difficulties in understanding one-generator objects do not arise.

\smallskip
    
    \item Section \ref{commutative}: Free skew braces with an abelian multiplicative group. These are residually finite (see Proposition \ref{resfincomm}), their multiplicative groups are free abelian (see Proposition \ref{freeabeliangroup}), and their additive groups are torsion-free (see Proposition \ref{additivegrouptorsionfree}).

\smallskip
    
    \item Section \ref{radicalring}: Free skew braces whose additive and multiplicative groups are abelian (these are known to be equivalent to commutative radical rings). These are residually finite (see Theorem \ref{resfiniradicalring}). Their multiplicative groups are free abelian (see Corollary~\ref{corradicalring}), they have sub-objects that are not free (see~Proposition~\ref{nielsen}). Last but not least, these objects are extensions of free abelian groups (see Theorem \ref{previousproof}).

\smallskip

\item Section \ref{centrallynilponte}: Free centrally nilpotent skew braces of class $2$ (of abelian type). Recall that central nilpotency is the strongest nilpotency concept for skew braces so far introduced, so for example it implies right nilpotency. It turns out that these free objects are residually finite (see Theorem \ref{resfinfinall}), and that they have isomorphic additive and multiplicative groups (see Theorem~\ref{risultatofinale}).
\end{itemize}

\bigskip

\section{Preliminaries}\label{preliminaries}

In this section we first provide some preliminaries on solutions to the Yang--Baxter equation and skew braces (see \cite{cedo2026groups} for a more detailed account on the subject). In doing this, we introduce a new characterization of skew braces of right nilpotent class at most $n$ which is completely based on the lambda-action (see~Pro\-po\-sition~\ref{lambda action for rclass n}). In~Section~\ref{secsolutions}, we apply this description to introduce a novel class of solutions to the Yang--Baxter equation generalizing a well-known class of solutions, namely ‘‘multipermutation solutions’’ (see \cite{cedo2026groups}). Note also that in the last part of this section, we introduce ‘‘relative’’ free skew braces, proving a result that will allow us to obtain that certain ‘‘relative’’ free skew braces are Hopfian (see \cref{hopfian}).

\medskip

A \emph{set-theoretic solution to the Yang--Baxter Equation} (YBE, for short) consists of a non-empty set~$X$
and a bijection $r: X^2 \to X^2$ that satisfies on $X^3$ the equation
\begin{equation}\label{YBE}
  (r \times \id_X) (\id_X \times r)(r \times \id_X) = (\id_X \times r)(r \times \id_X) (\id_X \times r). \tag{$\dagger$}
\end{equation}
Writing such a solution $(X,r)$ as 
$r(x,y) = (\lambda_x(y), \rho_y(x))$,
it is said to be  \emph{non-degenerate} 
if the mappings $\lambda_x, \rho_x$ are bijective for all $x \in X$.
\begin{rem}
\label{nondeg implies bij}
    In \cite{DiagonalsJedPil} it is proven that maps satisfying \cref{YBE} and which are non-degenerate are necessarily bijective.
\end{rem}

\begin{convention*}
    From now on, we use the phrase \emph{solutions to the Yang--Baxter Equation}, or simply \emph{solutions}, to refer to non-degenerate set-theoretic solutions to the Yang--Baxter Equation.
\end{convention*}

A \emph{morphism of solutions} $f:(X,r)\to (Y,s)$
is a map $f:X\to Y$ such that $(f\times f)r=s(f\times f)$.
Morphisms of solutions of the Yang--Baxter Equation are closed under composition 
and include identity maps, making these objects into a category.

\medskip

Solutions of the Yang--Baxter Equation are closely related to skew braces, 
which provide an algebraic framework for their study.
Recall that a ({\it left}) \emph{skew brace} is a triple
$(B,+,\circ)$, where $(B,+)$ and $(B,\circ)$ are groups, satisfying the {\it left skew distributivity}:
\begin{equation}\label{E:LeftSkewDistr}
a\circ (b+c)=a\circ b-a+a\circ c,
\end{equation}
for all $a,b,c\in B$; here, $-a$ denotes the inverse of $a$ in $(B,+)$. We say that $B$ is \emph{two-sided} if it also satisfies the {\it right skew distributivity}:
\begin{equation}\label{E:RightSkewDistr}
(b+c)\circ a= b\circ a-a+ c\circ a,
\end{equation} for every $a,b,c\in B$.
The group $(B,+)$ is called the \emph{additive group of $B$}, while $(B,\circ)$ is called the \emph{multiplicative group of~$B$} and $\bar{a}$ denotes the inverse of $a$ in $(B,\circ)$.
A \emph{morphism of skew braces} $f:(A,+,\circ)\to (B,+,\circ)$
is a map $f:A\to Y=B$ 
that is both a morphism of the additive and the multiplicative groups.
Morphisms of skew braces are closed under composition 
and include identity maps, making these objects into a category.

We say that the skew brace is \emph{of type~$\chi$}
if the additive group satisfies some property~$\chi$, so
for example, skew braces with an abelian additive group are called \emph{skew braces of abelian type}. If~$(B,\circ)$ is abelian, then we say that $B$ is {\it commutative}.

We use the notations $[B,B]_+$ and $[B,B]_\circ$ to denote the commutator subgroup of the groups $(B,+)$ and~$(B,\circ)$, respectively. More generally, the subscripts $+$ and $\circ$ indicate which of the two group structures is being considered when using group-theoretic notation. Any group $(G,\cdot)$ yields naturally two skew braces:~$(G,\cdot,\cdot)$, called the \emph{trivial skew brace on $G$} (sometimes simply denoted $\Triv (G)$) and $(G,\cdot^{\mathrm{op}},\cdot)$, called the \emph{almost trivial skew brace on $G$}. 

Every skew brace $B$ defines an action 
by automorphisms of $(B,\circ)$ on $(B,+)$
called the \emph{$\lambda$-action} and 
defined by
\[
\lambda_a(b)=-a+a\circ b,
\]
for all $a,b\in B$. Thus, the {\it $\lambda$-map} $a \mapsto \lambda_a$ defines a homomorphism from $(B,\circ)$ to~$\Aut(B,+)$, so that $B$ embeds as a regular subgroup of the holomorph $\Hol(B,+)=(B,+)\rtimes \Aut(B,+)$. (Recall that a subgroup of $\Hol(B,+)$ is called \emph{regular} if it acts regularly, i.e. freely and transitively, on $(B,+)$.) It turns out that every skew brace arises in this way.

\begin{thm}[see {\cite[Theorem 4.3]{Skew}}]
\label{skew braces and regular subgroups}
Let $(B,+)$ be a group. Then there is a one-to-one correspondence between 
skew brace structures $(B,+,\circ)$ and regular subgroups of the holomorph 
$\mathrm{Hol}(B,+)$. Under this correspondence, the multiplicative group~$(B,\circ)$ is identified with a regular subgroup of $\mathrm{Hol}(B,+)$ via 
$a \mapsto (a,\lambda_a)$.
\end{thm}

Let $X$ be a subset of the skew brace $(B,+,\circ)$. One says  that $X$ is a {\it sub-skew brace} of $B$ if $X$ is both an additive and a multiplicative subgroup of $B$. If $X$ is an additive subgroup (resp. a normal additive subgroup) which is also {\it $\lambda$-invariant} (that is, $\lambda_b(X)=X$ for every~\hbox{$b\in B$),} then $X$ is said to be a {\it left-ideal} (resp. {\it strong left-ideal}). (Note that every left-ideal is a sub-skew brace.) Finally, $X$ is an {\it ideal} if it is a strong left-ideal which is also a multiplicatively normal subgroup. An ideal $I$ of $B$ allows us to consider the quotient skew brace $B/I$ with the induced operations.

\medskip

Skew braces give rise to set-theoretic solutions to the Yang--Baxter Equation in the following way.
If $B$ is a skew brace, then $(B,r_B)$ is a solution,
where
\[
r_B(a,b)=(\lambda_a(b),\overline{\lambda_a(b)}\circ a\circ b),
\]
for every $a,b \in B$.  Moreover, given a set-theoretic solution to the Yang--Baxter Equation $(X,r)$, one can define a skew brace  as follows. First, one defines the \emph{structure group of $(X,r)$} as 
\[
(G(X,r),\circ)=\langle X\mid x\circ y =u\circ v\text{ where }r(x,y)=(u,v)\rangle.
\]
Next, one notices that the position $x+y=x\circ \lambda_x^{-1}(y)$ defines a second group structure on $G(X,r)$ making $(G(X,r),+,\circ)$ a skew brace. This is the {\it structure skew brace of $(X,r)$}.
In this framework, it is convenient to recall that  for all $y\in X$ one has a bijection
\[
\sigma_y:X\to X \qquad \sigma_y:x\mapsto\lambda_y\left(\rho_{\lambda_x^{-1}(y)}(x)\right),
\]
which allows to present the group $(G(X,r),+)$ 
in the following way
\[
(G(X,r),+)=\langle X\mid x+y=y+\sigma_y(x) \rangle.
\]

The following universal property of the structure skew brace associated with a set-theoretic solution of the Yang--Baxter Equation was established in \cite[Theorem~9]{LuYanZhu} and later reformulated in the language of skew braces in \cite[Theorem~4.5]{SmokVen}. Here we denote by $$g_X:X\to G(X,r)$$ the canonical map $x\mapsto x$.

\begin{pro}[Universal property of the structure skew brace] \label{univ prop structure sb}
Let $(Y,s)$ be a solution to the Yang--Baxter Equation. For every skew brace $B$ 
    and every morphism of solutions $f:(Y,s)\to (B,r_B)$,
    there exists a unique homomorphism of skew braces
    \hbox{$G(f):G(Y,s)\to B$} such that $G(f)\circ g_Y=f$.
\end{pro}

Note that the map $g_X$ does not need to be injective.
Solutions for which it is injective are called \emph{injective solutions}. 
Moreover, for every solution $(X,r)$ the image $g_X(X)$
is an injective subsolution of $(G(X,r),r_{G(X,r)})$,
called the \emph{injectivisation} of $(X,r)$ and denoted by 
$\Inj(X,r)$. We will also denote by $\Inj_X$ the map $g_X$ with codomain restricted to its image.

We now recall the notions of nilpotency we shall deal with throughout the paper. In order to define right nilpotency, we also need  to recall that the $\lambda$-map of a skew brace $(B,+,\circ)$ naturally leads to a third operation (called the {\it star operation}):
\[
a\ast b=\lambda_a(b)-b
\]
for all $a,b\in B$. In some sense, this third operation measures the distance between the addition and the multiplication. Note that $(B,+,\ast)$ is a radical ring if and only if $B$ is two-sided of abelian type (see \cite{Brace}). If $X$ and $Y$ are subsets of a skew brace $B$, then we denote by $X\ast Y$ the subgroup of $(B,+)$ generated by~\hbox{$\{x\ast y\mid x\in X, y\in Y\}$.} With this convention, we recursively define the following subsets of $B$ (which turn out to be ideals): $B^{(n+1)}=B^{(n)}*B$ for every $n\geq 1$ and $B^{(1)}=B$.
Then $B$ is \emph{right nilpotent of class $n$} if 
$B^{(n+1)}=\{0\}$ and $B^{(n)}\neq\{0\}$. The category of right nilpotent skew braces of right nilpotent class at most $n$ is denoted by $\mathcal{RN}_n$.  The descending series $(B^{(n)})_{n\geq 1}$
measures how far the skew brace is from being trivial.
In particular, small nilpotency class imposes strong restrictions on the interaction between the two operations.
For example, $B$ is right nilpotent of class $1$ if and only if it is a trivial skew brace.

We show that the right nilpotency of a skew brace~$B$ can be characterised in terms of the 
``nested" $\lambda$-automorphisms of $(B,+)$
\[
\lambda^{\alpha_n}_{
  \lambda^{\alpha_{n-1}}_{
    \ddots_{
      \lambda^{\alpha_1}_{v_1}(v_2)
    }\iddots
  }(v_n)
}(-)
\]
where $n\geq1$ is an integer, $(\alpha_1,\dots ,\alpha_n)\in \{\pm1\}^{\times n}$ and $(v_1,\dots,v_n)\in B^{\times n}$.
Note that by $B^{\times n}$ we denote the cartesian product of $n$ copies of $B$. We use this notation to avoid confusion with products in the skew brace $B$.

To work efficiently, we introduce notations to represent these automorphisms.

    For a set $E$ and integers $k,i\geq2$, the concatenation of two elements $x= (x_1,\dots,x_k)\in E^{\times k}$ and $y=(y_1,\dots,y_i)\in E^{\times i}$ is denoted by $$x\smallfrown y= (x_1,\dots,x_k,y_1,\dots,y_i).$$ 
    If $k$ or $i$ equal $1$, then we usually omit the parentheses in the corresponding $1$-tuple, so for example we write $x\smallfrown (y_1,\dots,y_i)$ instead of $(x)\smallfrown (y_1,\dots,y_i)$.
    We respectively denote by $\mathfrak{l}x$ and $\mathfrak{r}x$  the left and right truncations $(x_2,\dots,x_k)$ and $(x_1,\dots,x_{k-1})$. If $(E,e)$ is a pointed set, we make the convention that $E^{\times0}=\{e\}$ and $\mathfrak{r}z=\mathfrak{l}z=e$ for all $z\in E$. In what follows, the set $\{\pm 1\}$ is considered as pointed in $1$, while the underlying set of a skew brace is pointed in $0$.

\begin{defn}
    \label{def:nested}
    Let $B$ be a skew brace. For $k\geq 1$ an integer, $\alpha=(\alpha_1,\dots,\alpha_k)\in\{\pm 1\}^{\times k}$ and $v=(v_1,\dots,v_k)\in B^{\times k}$ we define the map $\lambda_v^\alpha\colon B\to B$ recursively by putting
    \[\lambda_v^\alpha=\lambda^{\alpha_k}_{\lambda_{\mathfrak{r}v}^{\mathfrak{r}\alpha}(v_k)}.\]
    Note that in the case $k=1$, this is the usual action $\lambda^{\alpha_1}_{v_1}$ of the skew brace $B$.
\end{defn}

\begin{rem}\label{remarkconcatenation}
    Let $B$ be a skew brace, $k,i\geq 1$ be integers, $\alpha\in\{\pm 1\}^{\times k}$, $\beta\in \{\pm 1\}^{\times i}$, $v=(v_1,\dots,v_k)\in B^{\times k}$ and $u\in B^{\times i}$. 
    We have the following formula:
    \[\lambda_{u\smallfrown v}^{\beta\smallfrown\alpha}=\lambda^{\alpha}_{(\lambda^{\beta}_{u}(v_1))\smallfrown\mathfrak{l}v}.\]
\end{rem}

\begin{defn}
    Let $B$ be a skew brace. Let $K_0(B)=\{0\}$ and, for all integers $i\geq 1$, let $K_i(B)$ be the set of the elements $b\in B$ such that $\lambda_{b\smallfrown v}^{\epsilon\smallfrown \alpha}=\lambda_v^\alpha$ for all $\epsilon\in \{\pm 1\}$, $\alpha\in \{\pm 1\}^{\times(i-1)}$ and $v\in B^{\times(i-1)}$.
\end{defn}

\begin{lem}\label{lem:concatinvarience}
Let $B$ be a skew brace. For every integer $k\geq 0$,  $K_k(B)\subseteq K_{k+1}(B)$ for all $k\geq 0$. More precisely, if $i\geq 1$ is an integer, \hbox{$v=(v_1,\dots, v_i)\in B^{\times i}$,} and $\alpha\in \{\pm1\}^{\times i}$, then     \begin{equation}
    \label{eq:concatenationinvarience}
        \lambda_{u\smallfrown v}^{\beta\smallfrown \alpha}=\lambda^\alpha_v
    \end{equation}
for all integers $k\geq 1$, $u\in K_{i+k}(B)\times K_{i+k-1}(B)\times\ldots\times K_{i+1}(B)$ and $\beta\in \{\pm1\}^{\times k}$.
\end{lem}
\begin{proof}
    Let $k\geq 1$ be an integer, $b\in K_k(B)$, $\epsilon\in \{\pm1\}$, $\alpha=(\alpha_1,\dots,\alpha_k)\in \{\pm1\}^{\times k}$ and $v=(v_1,\dots,v_k)\in B^{\times k}$. Then 
    $$\lambda_{b\smallfrown v}^{\epsilon\smallfrown\alpha}=\lambda^{\alpha_k}_{\lambda_{b\smallfrown\mathfrak{r}v}^{\epsilon\smallfrown\mathfrak{r}\alpha}(v_k)}=\lambda^{\alpha_k}_{\lambda_{\mathfrak{r}v}^{\mathfrak{r}\alpha}(v_k)}=\lambda^\alpha_v.$$
    Therefore, $b\in K_{k+1}(B)$, which proves that we have an ascending chain of subsets. 

We now prove formula~\eqref{eq:concatenationinvarience}. Write
$u=(u_1,\dots,u_k)$ and $\beta=(\beta_1,\dots,\beta_k)$. We proceed by
induction on $k$. If $k=1$, then $u_1\in K_{i+1}(B)$, and the previous
argument, applied to the tuple $v$ of length $i$, gives
$$
\lambda_{u_1\smallfrown v}^{\beta_1\smallfrown\alpha}
=
\lambda_v^\alpha.
$$

Assume now that $k>1$. Since $u_1\in K_{i+k}(B)$ and
$(u_2,\dots,u_k)\smallfrown v$ has length $i+k-1$, the previous argument gives
$$
\lambda_{u\smallfrown v}^{\beta\smallfrown\alpha}
=
\lambda_{(u_2,\dots,u_k)\smallfrown v}^{(\beta_2,\dots,\beta_k)\smallfrown\alpha}.
$$
By the induction hypothesis applied to $(u_2,\dots,u_k)$, we get
$$
\lambda_{(u_2,\dots,u_k)\smallfrown v}^{(\beta_2,\dots,\beta_k)\smallfrown\alpha}
=
\lambda_v^\alpha.
$$
Therefore
$
\lambda_{u\smallfrown v}^{\beta\smallfrown\alpha}=\lambda_v^\alpha,
$ as required.
\end{proof}

\begin{pro}
\label{pro:ascendingchainofideals}
    Let $B$ be a skew brace. For every integer $i\geq  0$, the set $K_i(B)$ is a sub-skew brace of $B$, and an ideal of $K_{i+1}(B)$ such that $K_{i+1}(B)\ast B\subseteq K_i(B)$. 
\end{pro}
\begin{proof}
    For $i=0$, it is clear. Let $i\geq 1$. We start by showing that $K_i(B)$ is stable under the lambda action of $K_{i+1}(B)$. Indeed, let $c\in K_{i+1}(B)$, $b\in K_i(B)$, $\epsilon\in \{\pm1\}$, $\alpha\in \{\pm 1\}^{\times(i-1)}$ and $v=(v_1,\dots,v_{i-1})\in B^{\times(i-1)}$. We claim that
$$
\lambda_{\lambda_c(b)\smallfrown v}^{\epsilon\smallfrown\alpha}
=
\lambda^\alpha_v.
$$
To see this, we first use Remark \ref{remarkconcatenation} with $u=(c)$,
$\beta=(1)$ and $(b)\smallfrown v$ in place of~$v$. This gives
$$
\lambda^{(1,\epsilon)\smallfrown \alpha}_{(c,b)\smallfrown v}
=
\lambda^{\epsilon\smallfrown\alpha}_{\lambda_c(b)\smallfrown v}.
$$
On the other hand, since $c\in K_{i+1}(B)$ and $b\in K_i(B)$, we may apply
Lemma~\ref{lem:concatinvarience} to the tuple $(c,b)\in
K_{i+1}(B)\times K_i(B)$ and to $v\in B^{\times(i-1)}$. Therefore
$$
\lambda^{(1,\epsilon)\smallfrown \alpha}_{(c,b)\smallfrown v}
=
\lambda^\alpha_v.
$$
Combining the two equalities, we get
$$
\lambda_{\lambda_c(b)\smallfrown v}^{\epsilon\smallfrown\alpha}
=
\lambda^{(1,\epsilon)\smallfrown \alpha}_{(c,b)\smallfrown v}
=
\lambda^\alpha_v.
$$
Hence $\lambda_c(b)\in K_i(B)$.

Because $b\in K_i(B)$ and since $\lambda_{\bar{b}\smallfrown v}^{\epsilon\smallfrown \alpha}=\lambda_{b\smallfrown v}^{(-\epsilon)\smallfrown \alpha}=\lambda_v^\alpha$ we obtain that also $\bar{b}\in K_i(B)$.
 As $-b=\lambda_b(\bar{b})$ we thus obtain that $-b\in K_i(B)$.
 To show that $K_i(B)$ is a skew brace, it is left to prove that it is closed under addition. 
To do so, assume also $a\in K_i(B)$. By the recursive definition of the maps $\lambda^\alpha_v$,
$$
\lambda_{(a+b)\smallfrown v}^{1\smallfrown\alpha}
=
\lambda^{\alpha}_{\lambda_{a+b}(v_1)\smallfrown\mathfrak{l}v}
=
\lambda^{\alpha}_{\lambda_a\left(\lambda_{\lambda_a^{-1}(b)}(v_1)\right)\smallfrown\mathfrak{l}v}.
$$
Now observe that
$$
\lambda_{(a,b)}^{(-1,1)}
=
\lambda^1_{\lambda_a^{-1}(b)}
=
\lambda_{\lambda_a^{-1}(b)}.
$$
Thus
$$
\lambda^{\alpha}_{\lambda_a\left(\lambda_{\lambda_a^{-1}(b)}(v_1)\right)\smallfrown\mathfrak{l}v}
=
\lambda^{1\smallfrown \alpha}_{a\smallfrown \lambda_{(a,b)}^{(-1,1)}(v_1)\smallfrown \mathfrak{l}v}.
$$ Put
$
c=\lambda_{(a,b)}^{(-1,1)}(v_1)=\lambda_{\lambda_a^{-1}(b)}(v_1).
$
Since $a\in K_i(B)$ and $c\smallfrown\mathfrak{l}v$ has length $i-1$, by the
definition of $K_i(B)$ we have
$$
\lambda^{1\smallfrown \alpha}_{a\smallfrown c\smallfrown \mathfrak{l}v}
=
\lambda^\alpha_{c\smallfrown \mathfrak{l}v}.
$$
On the other hand, by the recursive definition of the maps $\lambda^\alpha_v$,
we have
$$
\lambda_{(a,b)\smallfrown v}^{(-1,1,\alpha)}
=
\lambda^\alpha_{\lambda_{(a,b)}^{(-1,1)}(v_1)\smallfrown\mathfrak{l}v}
=
\lambda^\alpha_{c\smallfrown\mathfrak{l}v}.
$$
Therefore
$$
\lambda^{1\smallfrown \alpha}_{a\smallfrown c\smallfrown \mathfrak{l}v}
=
\lambda_{(a,b)\smallfrown v}^{(-1,1,\alpha)}.
$$ Since $a,b\in K_i(B)$ and $K_i(B)\subseteq K_{i+1}(B)$, we may apply
Lemma~\ref{lem:concatinvarience} to the tuple $(a,b)\in K_{i+1}(B)\times K_i(B)$.
Hence
$$
\lambda_{(a,b)\smallfrown v}^{(-1,1,\alpha)}
=
\lambda_v^\alpha.
$$
Therefore
$\lambda_{(a+b)\smallfrown v}^{1\smallfrown\alpha}=
\lambda_v^\alpha$.
Similarly, one obtains
$\lambda_{(a+b)\smallfrown v}^{(-1)\smallfrown\alpha}=
\lambda_v^\alpha.$
Hence indeed $a+b\in K_i(B)$.

Then, to show that $K_{i+1}(B)\ast B\subseteq K_i(B)$, 
it is enough to prove that the additive generators  $c\ast x$ of $K_{i+1}(B)\ast B$ are contained in $K_i(B)$, with 
$x\in B$ and $c\in K_{i+1}(B)$. Put
$$
d=\lambda_c(x),\qquad p=\lambda_d^{-1}(-x).
$$
Then $\lambda_d(p)=-x$, and therefore
$$
d\circ p=d+\lambda_d(p)=d-x.
$$
Since $c*x=\lambda_c(x)-x$ and $d=\lambda_c(x)$, we get
$$
c*x=d-x=d\circ p.
$$
Let now $\alpha\in\{\pm1\}^{\times(i-1)}$ and
$v=(v_1,\dots,v_{i-1})\in B^{\times(i-1)}$. Put also
$
z=\lambda_p(v_1).
$
We claim that
$$
\lambda_{(c\ast x)\smallfrown v}^{1\smallfrown\alpha}=\lambda_v^\alpha.
$$
Because $c\ast x=d\circ p$, we get
\[
\lambda_{(c\ast x)\smallfrown v}^{1\smallfrown\alpha}
=
\lambda^\alpha_{\lambda_{c\ast x}(v_1)\smallfrown\mathfrak{l}v}
=
\lambda^\alpha_{\lambda_d(\lambda_p(v_1))\smallfrown\mathfrak{l}v}
=
\lambda^\alpha_{\lambda_d(z)\smallfrown\mathfrak{l}v}.
\]
Now, by Remark \ref{remarkconcatenation} applied to $u=(c,x)$, $\beta=(1,1)$ and
$z\smallfrown\mathfrak{l}v$ in place of $v$, we have
$$
\lambda^{(1,1)\smallfrown\alpha}_{c\smallfrown(x,z)\smallfrown\mathfrak{l}v}
=
\lambda^\alpha_{\lambda^{(1,1)}_{(c,x)}(z)\smallfrown\mathfrak{l}v}.
$$
Since
$
\lambda^{(1,1)}_{(c,x)}=\lambda_{\lambda_c(x)}=\lambda_d,
$
this gives
$$
\lambda^{(1,1)\smallfrown\alpha}_{c\smallfrown(x,z)\smallfrown\mathfrak{l}v}
=
\lambda^\alpha_{\lambda_d(z)\smallfrown\mathfrak{l}v}.
$$
Moreover, since $c\in K_{i+1}(B)$ and the tuple
$(x,z)\smallfrown\mathfrak{l}v$ has length $i$, the definition of
$K_{i+1}(B)$ gives
$$
\lambda^{(1,1)\smallfrown\alpha}_{c\smallfrown(x,z)\smallfrown\mathfrak{l}v}
=
\lambda^{1\smallfrown\alpha}_{(x,z)\smallfrown\mathfrak{l}v}.
$$
Furthermore, since
$$
x\circ \lambda_x^{-1}(-x)=x-x=0,
$$
we have
$
\lambda_x\lambda_{\lambda_x^{-1}(-x)}=\id,
$
and hence
$
\lambda_x=\lambda^{-1}_{\lambda_x^{-1}(-x)}.
$
Therefore
\[
\lambda^\alpha_{\lambda_x(z)\smallfrown\mathfrak{l}v}
=
\lambda^\alpha_{\lambda^{-1}_{\lambda_x^{-1}(-x)}(z)\smallfrown\mathfrak{l}v}.
\]
We now compare $\lambda_x^{-1}(-x)$ with $p=\lambda_d^{-1}(-x)$. By the recursive
definition of the maps $\lambda^\alpha_v$, we have
\[
\lambda^\alpha_{\lambda^{-1}_{\lambda_x^{-1}(-x)}(z)\smallfrown\mathfrak{l}v}
=
\lambda^{(-1,-1)\smallfrown\alpha}_{(x,-x,z)\smallfrown\mathfrak{l}v}.
\]
On the other hand, again using $c\in K_{i+1}(B)$ and applying
Lemma~\ref{lem:concatinvarience}, we get
\[
\lambda^{(-1,-1)\smallfrown\alpha}_{(x,-x,z)\smallfrown\mathfrak{l}v}
=
\lambda^{1\smallfrown(-1,-1)\smallfrown\alpha}_{c\smallfrown(x,-x,z)\smallfrown\mathfrak{l}v}.
\]
By the recursive definition, the last term is
\[
\lambda^{1\smallfrown(-1,-1)\smallfrown\alpha}_{c\smallfrown(x,-x,z)\smallfrown\mathfrak{l}v}
=
\lambda^\alpha_{\lambda^{-1}_{\lambda_d^{-1}(-x)}(z)\smallfrown\mathfrak{l}v}
=
\lambda^\alpha_{\lambda_p^{-1}(z)\smallfrown\mathfrak{l}v}.
\]
Since $z=\lambda_p(v_1)$, we have $\lambda_p^{-1}(z)=v_1$. Thus
\[
\lambda_{(c\ast x)\smallfrown v}^{1\smallfrown\alpha}
=
\lambda^\alpha_{v_1\smallfrown\mathfrak{l}v}
=
\lambda_v^\alpha.
\]
This proves the claim.

The case of $-1$ is analogous. Indeed, using again
$c\ast x=d\circ p$ and $p=\lambda_d^{-1}(-x)$, one obtains
\[
\lambda_{(c\ast x)\smallfrown v}^{(-1)\smallfrown\alpha}
=
\lambda^\alpha_{\lambda^{-1}_{c*x}(v_1)\smallfrown\mathfrak{l}v}
=
\lambda^\alpha_{\lambda^{-1}_{p}\lambda^{-1}_{d}(v_1)\smallfrown\mathfrak{l}v}.
\]
Repeating the previous argument with inverse lambda maps gives
\[
\lambda^\alpha_{\lambda^{-1}_{p}\lambda^{-1}_{d}(v_1)\smallfrown\mathfrak{l}v}
=
\lambda^\alpha_{v_1\smallfrown\mathfrak{l}v}
=
\lambda_v^\alpha.
\]
Hence $c\ast x\in K_i(B)$, and therefore
$
K_{i+1}(B)\ast B\subseteq K_i(B).
$

That $K_i(B)$ is an ideal of $K_{i+1}(B)$ now follows from
$K_{i+1}(B)*K_i(B)\subseteq K_i(B)$ and
$
K_i(B)*K_{i+1}(B)\subseteq K_{i+1}(B)*K_{i+1}(B)\subseteq K_i(B).$
\end{proof}

 We are now in a position to prove a characterisation of right nilpotent skew braces in terms of the $\lambda$-action

\begin{pro}\label{lambda action for rclass n}
Let $n\geq 1$ be an integer. A skew brace $B$ is right nilpotent of class at most $n$ if and only if
$\lambda_v^\alpha =\lambda_{\mathfrak{l}v}^{\mathfrak{l}\alpha}$
for all $\alpha\in \{\pm1\}^{\times n}$ and $v\in B^{\times n}$. In other words, $B$ is right nilpotent of class at most $n$ if and only if $K_n(B)=B$.
\end{pro}
\begin{proof}
    For $n=1$, it is clear that both statements are equivalent to the skew brace being trivial.
    Suppose by induction that the result is true for a positive  integer $n$.
    Assume that $B$ is right nilpotent of class at most $n+1$. Let $\alpha=(\alpha_1,\dots,\alpha_{n+1})\in \{\pm 1\}^{\times (n+1)}$ and $(v_1,\dots,v_{n+1})\in B^{\times (n+1)}$.

By the induction hypothesis applied to the quotient $B/B^{(n+1)}$, we have
that
$$
\lambda_{\overline{\mathfrak rv}}^{\mathfrak r\alpha}
=
\lambda_{\overline{\mathfrak l\mathfrak rv}}^{\mathfrak l\mathfrak r\alpha}
$$
in $B/B^{(n+1)}$, where bars denote the images in the quotient. Therefore the
two elements
$$
x=\lambda_{\mathfrak rv}^{\mathfrak r\alpha}(v_{n+1})
\qquad\text{and}\qquad
y=\lambda_{\mathfrak l\mathfrak rv}^{\mathfrak l\mathfrak r\alpha}(v_{n+1})
$$
have the same image in $B/B^{(n+1)}$. Hence there exists
$z\in B^{(n+1)}$ such that
$
x=z\circ y.
$
Since $B$ is right nilpotent of class at most $n+1$, 
$B^{(n+1)}\subseteq\ker(\lambda)$, so $\lambda_z=\id$
and
$
\lambda_x=\lambda_{z\circ y}=\lambda_z\lambda_y=\lambda_y.
$
Consequently also $\lambda_x^{-1}=\lambda_y^{-1}$. Hence, for both possible
values of $\alpha_{n+1}\in\{\pm1\}$, we have
$
\lambda_x^{\alpha_{n+1}}=\lambda_y^{\alpha_{n+1}}.
$
Using the recursive definition of~$\lambda_v^\alpha$, we obtain
\begin{align*}
\lambda_v^\alpha
=
\lambda^{\alpha_{n+1}}_{\lambda_{\mathfrak rv}^{\mathfrak r\alpha}(v_{n+1})}
=
\lambda_x^{\alpha_{n+1}}
=
\lambda_y^{\alpha_{n+1}}
=
\lambda^{\alpha_{n+1}}_{\lambda_{\mathfrak l\mathfrak rv}^{\mathfrak l\mathfrak r\alpha}(v_{n+1})}.
\end{align*}
Finally, since
$$
\mathfrak l\mathfrak rv=\mathfrak r\mathfrak lv
\qquad\text{and}\qquad
\mathfrak l\mathfrak r\alpha=\mathfrak r\mathfrak l\alpha,
$$
the last term is precisely
$
\lambda_{\mathfrak lv}^{\mathfrak l\alpha}
$ by the recursive definition. 
Thus $\lambda_v^\alpha=\lambda_{\mathfrak lv}^{\mathfrak l\alpha}$.

 Conversely, assume that $K_{n+1}(B)=B$.  Since $B^{(1)}=B=K_{n+1}(B)$, we claim by induction on $j$ that
$$
B^{(j)}\subseteq K_{n+2-j}(B)
$$
for every $1\leq j\leq n+2$. The case $j=1$ is clear. If
$B^{(j)}\subseteq K_{n+2-j}(B)$, then, by Proposition~\ref{pro:ascendingchainofideals},
$$
B^{(j+1)}
=
B^{(j)}*B
\subseteq
K_{n+2-j}(B)*B
\subseteq
K_{n+1-j}(B).
$$
Thus the claim follows. Taking $j=n+2$, we obtain
$$
B^{(n+2)}\subseteq K_0(B)=\{0\}.
$$
Hence $B$ is right nilpotent of class at most $n+1$.
\end{proof}

\begin{cor}\label{lambda action for rclass 2}
A skew brace $B$ is right nilpotent of class at most $2$ if and only if~\hbox{$\lambda_b=\lambda_{\lambda_a(b)}$} for every $a,b\in B$.
In particular, in such a skew brace,  $\lambda_{a+b}=\lambda_a\lambda_b$ for every $a,b\in B$.
\end{cor}
\begin{proof}
The first part is an immediate consequence of Proposition~\ref{lambda action for rclass n}.
The second part is obtained as follows:
$\lambda_{a+b} =\lambda_{a\circ\lambda_{a}^{-1}(b)}= \lambda_a \lambda_{\lambda_{\overline{a}}(b)}=\lambda_a\lambda_b$.
\end{proof}

\begin{rem}
Note that nested conditions of the kind of Definition \ref{def:nested} have already been
used in the study of multipermutation solutions; see, for instance,
\cite[Theorem~7.10]{gateva2012multipermutation} for square-free involutive solutions, and \cite{DiagonalsJedPil} for related nested conditions involving
diagonal maps.

{\rm In the context of Corollary \ref{lambda action for rclass 2}, it is worth mentioning that the authors of \cite{MR4634954} studied the class of $2$-reductive solutions $(X,r)$. These solutions satisfy the stronger conditions
\[
\lambda_{\lambda_x(y)}=\lambda_y
\qquad\text{and}\qquad
\rho_{\rho_x(y)}=\rho_y,
\]
and are therefore substantially more restrictive than the setting considered here. Indeed, it is shown in \cite{MR4634954} that these identities force the group
\[
\langle \lambda_x,\rho_x \mid x\in X\rangle
\]
to be abelian. In contrast, this need not be the case for skew braces of right nilpotency class~$2$; examples of order~$24$ already provide counterexamples (e.g.\ \texttt{SmallSkewbrace(24,813)}).
}
\end{rem}

For our purposes, we should also recall a nilpotency concept for skew braces that is stronger than right nilpotency and that reduces to the usual nilpotency for groups in trivial skew braces. In order to do this, we preliminarily need to introduce certain relevant sub-structures that often pop-up in studying a skew brace. Let 
$(B,+,\circ)$ be a skew brace. First, it is well-known that the kernel $\ker(\lambda)$ of the lambda-function is a trivial sub-skew brace (besides obviously being a multiplicatively normal subgroup) that may also not be a left-ideal. On the other hand, the intersection of~$\ker(\lambda)$ with the additive centre $Z(B,+)$ is an ideal of $B$, called the \emph{socle} of $B$, and denoted by $\operatorname{Soc}(B)$. The socle is very relevant in the context of right nilpotency. In fact, having recursively defined the socle series of $B$ by putting $\operatorname{Soc}_0(B)=\{0\}$ and $\operatorname{Soc}_{i+1}(B)/\operatorname{Soc}_i(B)=\operatorname{Soc}(B/\operatorname{Soc}_i(B))$, one has that $B$ is right nilpotent of nilpotent type if and only if $B=\operatorname{Soc}_n(B)$ for some non-negative integer~$n$ (see \cite{MR4627847}).

 We take this opportunity to recall the relation between right nilpotent skew
braces of nilpotent type and multipermutation solutions. If $(X,r)$ is a
solution, its retraction is the solution $\Ret(X,r)$ on the quotient
$X/{\sim}$, where
$x\sim y$
if and only if
$\lambda_x=\lambda_y$
and
$\rho_x=\rho_y.$
Iterating this construction, we put
$\Ret^0(X,r)=(X,r)$ and
$\Ret^{m+1}(X,r)=\Ret(\Ret^m(X,r)).$
The solution $(X,r)$ is called \emph{multipermutation} if
$\Ret^m(X,r)$
is the trivial solution on a singleton for some $m\geq 0$. With this terminology, it is known that a solution $(X,r)$ is
multipermutation if and only if its structure skew brace $G(X,r)$ is right nilpotent of nilpotent type (see the introduction of \cite{MR4627847}).

The intersection of $\operatorname{Soc}(B)$ with the multiplicative centre $Z(B,\circ)$ is again an ideal of $B$, which is called the {\it centre} of $B$ (or also the {\it annihilator} of $B$) and is denoted by~$Z(B)$ (or by $\operatorname{Ann}(B)$). Similarly to the socle series, one recursively defines the {\it upper central series} of $B$ as follows: $Z_0(B)=\{0\}$ and $Z_{i+1}(B)/Z_i(B)=Z(B/Z_i(B))$. Then $B$ is {\it centrally nilpotent} if $B=Z_n(B)$ for some non-negative integer $n$. Clearly, if $B$ is centrally nilpotent, then it is also right nilpotent of nilpotent type, but the converse does not hold in general. Centrally nilpotent skew braces possess several remarkable properties that make them easier to handle. For a general overview of these features, we refer the reader to \cite{MR4884002}.

\medskip

Let $\mathcal{K}$ be a category with a forgetful functor $U_\mathcal{K}:\mathcal{K}\to \mathrm{Set}$ 
to the category of sets and let $X$ be a set.
A \emph{free object on $X$ in $\mathcal{K}$} is an object $F$ of $\mathcal{K}$
with a morphism of sets $i:X\to U(F)$ such that 
for every object $K$ of $\mathcal{K}$ and every morphism of sets $f:X\to U(K)$
there exists a unique morphism 
$\varphi:F\to Y$ in $\mathcal{K}$ such that $\varphi\circ i=f$.
In the categories of skew braces and set-theoretic solutions considered in this paper,
there is a natural forgetful functor to the category of sets.
We will use the same notation for an object and its underlying set when no confusion arises.

Since a number of different types of free objects will appear throughout the paper, we fix the following convention. If $\star$ is any operation symbol, then $$\operatorname{F}_{\star}(X) \text{ and } \operatorname{S}_\star(X)$$ denote respectively 
the free group and the free monoid on $X$ whose operation is written as $\star$.
The corresponding free abelian versions are written as
$$\operatorname{F}_{*}^{\operatorname{ab}}(X) \text{ and }
\operatorname{S}_{*}^{\operatorname{ab}}(X).$$ The categories we are going to consider are: 
\begin{itemize}
    \item $\mathcal B$, the category of skew braces,
    \item $\mathcal{T}$, the category of two-sided skew braces,
    \item $\mathcal C$, the category of commutative skew braces,
    \item $\mathcal{RN}_n$, the category of right nilpotent skew braces of class at most $n$,
    \item $\mathcal{CN}_n$, the category of centrally nilpotent skew braces of class at most $n$.
\end{itemize}

The free object on $X$  within the category $\chi$ of skew braces (resp. braces) will be denoted by 
 $$\operatorname{FSB}_{\chi, X} \text{ (resp. } \operatorname{FB}_{\chi ,X}).$$ 
The free object generated on $X$ in the category of all skew braces $\mathcal{B}$ (resp. all braces)  will be denoted $\operatorname{FSB}_{ X}$ (resp. $\operatorname{FB}_{X}$).

\medskip

As we have already noted in the introduction, on some occasions, properties of the free objects may give interesting information about {\it all} the objects of that category. For example, the fact that a free group is residually finite means that every {\it law} holding in finite groups must also hold in all groups. Also, as a consequence of residual finiteness, one immediately obtains  that free groups of finite rank are {\it Hopfian} 
(that is, they are not isomorphic to any of their proper images); in turn, this has some relevant consequences in terms of generating sets of a free group. This explains the relevance of the property of being residually finite in the category of groups, and why it should be important to prove similar properties in other categories, such as that of skew braces. In the latter context, we say that a skew brace is {\it residually finite} if the intersection of the ideals with finite quotient is zero. In the next sections, we will frequently prove that a free skew brace object is residually finite, and in order to keep the paper as short as possible, we prove here a couple of general statements that provide interesting consequences for these free skew brace objects for which the residual finiteness have been obtained. For the sake of the reader, we recall that a skew brace is {\it Hopfian} if it is not isomorphic to any of its proper quotients, and is {\it finitely generated} if it contains a finite subset for which the only sub-skew brace containing it is the whole skew brace; note that the sub-skew brace generated by a set $X$ (that is, the smallest sub-skew brace containing $X$ with respect to inclusion) will be denoted by $\langle X\rangle$.

\begin{thm}\label{hopfian}
Let $B$ be a finitely generated skew brace which is also residually finite. Then $B$ is Hopfian.
\end{thm}
\begin{proof}
Suppose not. Then there exists an epimorphism $\varphi:B\rightarrow B$ whose kernel contains a non-zero element $b$, say. Since $B$ is residually finite, there is an ideal $I$ such that $B/I$ is finite and $b+I\neq I$. Let $\psi:B\rightarrow B/I$ be the natural  epimorphism. Since $B$ is finitely generated, the number of epimorphisms of $B$ onto $B/I$ is finite. But 
$\{\psi \varphi^n \mid n \text{ a positive integer}\}$ is  an infinite family of morphisms $B\rightarrow B/I$. 
 Indeed, let $b_0=b$ and $b_{n+1}\in B $ be a pre-image of $b_n\in B$ by $\varphi$ for all non-negative integers $n$, then  $\psi\varphi^n(b_n)=b+I\neq I=\psi\varphi^k(b_n)$ for all $k>n\geq 0$. This contradiction completes the proof.
\end{proof}

\begin{cor}\label{corfree}
Let $X$ be a finite set, $\mathfrak X$ a property pertaining to skew braces, and~$B:=\FSB_{\mathfrak X, X}$.
Suppose $B$ is residually finite. If $Y$ is any subset of $B$ such that $B=\langle Y\rangle$ and $|Y|\leq |X|$, then $B$ is free on $Y$.
\end{cor}
\begin{proof}
By freeness of $B$ and the hypothesis $|Y|\leq|X|$, there exists an epimorphism from $B$ to $B$ mapping $X$ onto $Y$. This must be injective by Theorem \ref{hopfian}.
\end{proof}

\section{Free solutions and their structure skew braces}\label{secsolutions}

In this section we study free objects in suitable categories of solutions of the Yang--Baxter equation and compare them with the corresponding free skew braces. We first recall the universal-algebraic framework that guarantees the
existence of free solutions in equational classes (see~Sub\-section \ref{subsection1}). We then discuss
injectivisation and describe free injective solutions inside structure skew braces (see Subsection \ref{subsection2}). Finally, we introduce the classes of right nilpotent solutions and relate them to left multipermutation solutions and to right nilpotent skew braces (see Subsection \ref{subsection3}).

\subsection{Equational classes of solutions and free objects}\label{subsection1}

To establish the existence of the corresponding free objects, one needs to interpret these solutions from the perspective of universal algebra. For the reader's convenience, we now recall some basic definitions and results in this context (for a more comprehensive account, see \cite{Univ}).

For a non-empty set $A$ and a non-negative integer $n$,
an \emph{$n$-ary operation on $A$} is a function
$f:\underbrace{A\times \ldots\times A}_{n \text{ times}}\to A$.
The integer $n$ is called the \emph{arity of $f$}.
Note that an operation with arity $0$ is simply a constant $a\in A$.
A \emph{signature} $(\mathcal{F}, \textrm{ar})$ is a set of operation symbols
$\mathcal{F}$ together with a map 
$\textrm{ar}: \mathcal{F}\to \N$.
Fix a signature $\mathcal{F}$, then
an \emph{$\mathcal{F}$-algebra} is a pair
$\left(A,\{f^A\mid f\in \mathcal{F} \}\right)$, where $A$ is a non-empty set
and $f^A$ is an $\textrm{ar}(f)$-ary operation on $A$.

Let $A$ and $B$ be two $\mathcal{F}$-algebras.
A \emph{homomorphism from $A$ to $B$} is 
a map of sets $\alpha:A\to B$
such that
\[
\alpha \left(f^A(a_1,\dots, a_n)\right)=f^B\left(\alpha(a_1),\dots, \alpha(a_n)\right),
\]
for all $n$-ary operations
$f\in \mathcal{F}$ and for
every $a_1,\dots, a_n\in A$.
Moreover, we say that $A$ is a \emph{subalgebra} of $B$ 
if $A\subseteq B$ and
the inclusion map is a homomorphism,
i.e. $f^B(a_1,\dots , a_n)\in A$
for every $a_1,\dots, a_n\in A$ and every $f\in \mathcal{F}$
of arity $n$.
We can also define the \emph{direct product $A\times B$ of $A$ and $B$}
as the algebra on the set $A\times B$
such that 
\[
f^{A\times B}\left((a_1,b_1),\dots, (a_n, b_n)\right)
=\left(f^A(a_1,\dots, a_n),f^B(b_1,\dots, b_n)\right),
\]
for every $a_1,\dots, a_n\in A$, $b_1,\dots, b_n\in B$ and $f\in \mathcal{F}$ of arity $n$.
More generally, given 
$(A_i)_{i\in I}$ a family of $\mathcal{F}$-algebras,
the \emph{direct product}
$A=\prod_{i\in I}A_i$
is the algebra on the set 
$\prod_{i\in I}A_i$
 with 
\[
f^A(a_1,\dots, a_n)(i)
=f^{A_i}(a_1(i),\dots, a_n(i))
\]
for every $i\in I$, $a_1,\dots, a_n\in \prod_{i\in I}A_i$, and $f\in \mathcal{F}$ of arity $n$.

\begin{defn}
    A non-empty class of $\mathcal{F}$-algebras is called a \emph{variety} 
    if it is closed under subalgebras, homomorphic images and direct products.
\end{defn}

A class of $\mathcal{F}$-algebras can often be described either by closure properties (subalgebras, homomorphic images, and direct products) or by identities. These two points of view are related by Birkhoff’s theorem, which will be used later.

Given an algebra $A$ we can consider more functions besides the ones in the signature 
that are compatible with its structure.
This leads us to the concept of terms.

Given a set of \emph{variables} $V$ and a signature $\mathcal{F}$, then the set $T(V)$ of
\emph{terms of type~$\mathcal{F}$ over $V$}
is the smallest set such that
\begin{enumerate}
    \item $V\subseteq T(V)$;
    \item $f\in T(V)$ for every $f\in \mathcal{F}$ with $\textrm{ar}(f)=0$;
    \item if $t_1,\dots, t_k\in T(V)$, then also the ``string" $f(t_1,\dots,t_k)\in T(V)$, 
    for every
    $f\in \mathcal{F}$ with $\textrm{ar}(f)=k$.
\end{enumerate}
For $t\in T(V)$ we write $t(v_1,\dots, v_n)$ 
to indicate that the elements of $V$ occurring in~$t$ are among $v_1,\dots, v_n$.
Moreover, given a term
$t(v_1,\dots, v_n)$ 
and an $\mathcal{F}$-algebra,
we define 
\[
t^A:\underbrace{A\times \ldots\times A}_{n \text{ times}}\to A
\]
as follows:
\begin{enumerate}
    \item if $t\in V$, say $t=v_i\in V$,
    then
    \[
    t^A(a_1,\dots, a_n)=a_i,
    \]
    for every $a_1,\dots, a_n\in A$;
    \item otherwise, $t=f(t_1(v_1,\dots,v_n),\dots, t_k(v_1,\dots,v_n))$
    for some $f\in \mathcal{F}$ of arity $k$ and $t_1,\dots, t_k\in T(V)$.
    Then 
    \[
    t^A(a_1,\dots, a_n)=
    f^A\left(t^A_1(a_1,\dots,a_n),\dots, t^A_k(a_1,\dots,a_n)\right).
    \]
\end{enumerate}

\begin{defn}
   Given a signature $\mathcal{F}$ and a set $V$, an \emph{identity of type $\mathcal{F}$ over $V$} is an expression of the form
   \(
   p\approx q,
   \)
   where $p,q\in T(V)$.

   If we denote by $\textrm{Id}(V)$ 
   the set of all the identities of type $\mathcal{F}$ over $V$, then we say that an algebra $A$ of type $\mathcal{F}$
   \emph{satisfies an identity} 
   $p(v_1,\dots,v_n)\approx q(v_1,\dots,v_n)$ 
   if for every choice of $a_1,\dots, a_n\in A$ we have
   $p^A(a_1,\dots,a_n)=q^A(a_1,\dots,a_n)$.
\end{defn}

In universal algebra, free algebras are defined as
quotients of the term algebra by a congruence 
(see \cite[Chapter II, §10]{Univ}).
Theorem 10.8 in \cite{Univ}
establishes that this construction satisfies the so called universal mapping property, 
i.e. it agrees with the categorical notion of a free object.

\begin{thm}[see {\cite[Theorem 10.12]{Univ}}]
    Let $\mathcal{V}$ be a variety of $\mathcal{F}$-algebras. Then for every set $X$ there exists a free algebra in $\mathcal{V}$ on $X$. 
\end{thm}

Many of the most popular classes of algebras are defined by
identities.
\begin{defn}
    Let $\mathcal{F}$  be a signature and $\Sigma$ a set of identities of type $\mathcal{F}$.
    Then we denote by $M_{\mathcal F}(\Sigma)=M(\Sigma)$ the class of $\mathcal F$-algebras satisfying $\Sigma$.

    We say that a class $K$ of $\mathcal F$-algebras is an \emph{equational class} if there is a set of identities $\Sigma$
    such that $K=M(\Sigma)$.
\end{defn}

We now recall Birkhoff’s theorem, which appears as Theorem 11.9 in \cite{Univ}.
\begin{thm}[Birkhoff]\label{birkoff}
    $\mathcal{V}$ is an equational class if and only if it is a variety.
\end{thm}

Now, we show that the class of solutions to the Yang--Baxter Equation is an equational class,
and thus a variety according to Birkhoff's theorem (\cref{birkoff}). 
Although one could also verify the closure properties in the definition of a variety,
we instead work with identities,
as this provides a more concrete and tractable description of the structure.

Note that this approach was taken in \cite[Corollary 5.1-5.2]{ChiMil} as well, which also showed the existence of a free solution of the Yang-Baxter equation on an arbitrary set. We provide a detailed proof for the convenience of the reader and as a basis to use in a later proof.

\begin{thm}
\label{solutions are variety}
    The solutions to the Yang--Baxter Equation form an equational class, and so they are a variety.
\end{thm}
\begin{proof}
Recall from \cref{nondeg implies bij} that a solution 
can be described by a map
$r(x,y) = (\lambda_x(y), \rho_y(x))$, satisfying \cref{YBE}
where for each $x,y \in X$ the maps
$\lambda_x, \rho_y : X \to X$ are bijections.

For convenience, we introduce binary operations
\[
\lambda(x,y) = \lambda_x(y), \qquad \rho(x,y) = \rho_y(x),
\]
so that $r(x,y) = (\lambda(x,y), \rho(x,y))$.
In this way, we can rewrite the fact that $r$
satisfies \cref{YBE} as
\begin{align*}
\lambda(\lambda(x,y),\,\lambda(\rho(x,y),z)) 
&=\lambda(x,\lambda(y,z)),\\
\rho(\lambda(x,y),\,\lambda(\rho(x,y),z)) 
&=
\lambda(\rho(x,\lambda(y,z)),\,\rho(y,z)),\\
\rho(\rho(x,\lambda(y,z)),\rho(y,z))
&=\rho(\rho(x,y),z), 
\end{align*}
for all $x,y,z\in X$.

To encode the non-degeneracy,
we can introduce maps $\tilde{\lambda},\tilde{\rho}:X\times X\to X$
such that for every $x,y\in X$
\[
\tilde{\lambda}(x, \cdot)=\lambda(x, \cdot)^{-1} \qquad \tilde{\rho}( \cdot, y)=\rho(\cdot,y)^{-1}.
\]
Or equivalently for every $x,y\in X$
\begin{align*}
    \tilde{\lambda}(x,\lambda(x,y)) &= y,\quad
\lambda(x,\tilde{\lambda}(x,y)) = y,\\
\rho(\tilde{\rho}(x,y),y) &= x,\quad
\tilde{\rho}(\rho(x,y),y) = x.
\end{align*}

Therefore, we can consider the  signature 
$\mathcal{F}=\{\lambda, \rho, \hat{\lambda}, \hat{\rho}, \tilde{\lambda}, \tilde{\rho}\}$
and the following sets of identities $\Sigma=\Sigma_{YB}
\cup \Sigma_{nd}$ over $V=\{v_1,v_2,v_3\}$, where
\[
\Sigma_{YB}=\left\{
\begin{aligned}
\lambda(\lambda(v_1,v_2),\,\lambda(\rho(v_1,v_2),v_3)) &\approx \lambda(v_1,\lambda(v_2,v_3)),\\
\rho(\lambda(v_1,v_2),\,\lambda(\rho(v_1,v_2),v_3)) &\approx \lambda(\rho(v_1,v_2),\rho(v_2,v_3)),\\
\rho(\rho(v_1,\lambda(v_2,v_3)),\rho(v_2,v_3))
&\approx \rho(\rho(v_1,v_2),v_3) \end{aligned}
\right\},
\]
\[
\Sigma_{nd}=\left\{
\begin{aligned}
\tilde{\lambda}(v_1,\lambda(v_1,v_2)) &\approx v_2,\quad
\lambda(v_1,\tilde{\lambda}(v_1,v_2)) \approx v_2,\\
\rho(\tilde{\rho}(v_1,v_2),v_2) &\approx v_1,\quad
\tilde{\rho}(\rho(v_1,v_2),v_2) \approx v_1
\end{aligned}
\right\}.
\]
    It is now easy to see that the class of 
    solutions to the Yang--Baxter Equation 
    is an equational class with signature $\mathcal{F}$
    and set of identities over $V$ given by $\Sigma$.
\end{proof}

A direct consequence is that the category of solutions to the Yang--Baxter Equation admits free objects. Before proceeding further, we note that also in case of free objects in categories of solutions, we adhere to certain general conventions: the free solution on $X$ in the category of {\it all} solutions is denoted by 
$$(\operatorname{FSol}_X,\; r_{\operatorname{FSol}}),$$ while the free solution within a subcategory consisting of all objects satisfying property~$\chi$ 
is denoted by
$$(\operatorname{FSol}_{\chi,X},\; r_{\operatorname{FSol_{\chi}}}).$$

\cref{univ prop structure sb}
implies that $G:\mathcal Sol\to\mathcal{B}$ (which associates with a solution $(Y,s)$ the structure skew brace $G(Y,s)$) is a functor and that for every solution $(Y,s)$ and for every skew brace $B$, the composition with $g_Y$ induces a natural bijection
\[
\Hom_{\mathcal{B}}(G(Y,s), B)\cong \Hom_{\mathcal{S}ol}((Y,s), (B,r_B)).
\]
(Note that this also yields  $\Hom_{\mathcal{S}ol}((Y,s), (B,r_B))\cong \Hom_{\mathcal{S}ol}(\Inj(Y,s), (B,r_B))$.)

This means that $G$
is left adjoint to the functor $S:\mathcal{B}\to \mathcal Sol$ 
that maps a skew brace $B$
to the associated solution $(B,r_B)$.
Recall that free objects are characterized by a universal property, or
equivalently by the fact that the corresponding free functor is left adjoint
to the forgetful functor.
So we have two adjunctions as follows.
\begin{center}
\begin{tikzcd}[column sep=huge]
\mathcal{S}et
  \arrow[r, "F_{\mathcal Sol}"{name=A}, bend left=15]
& \mathcal Sol
  \arrow[l, "U_{\mathcal Sol}"{name=B}, bend left=15]
  \arrow[r, "G"{name=C}, bend left=15]
& \mathcal B
  \arrow[l, "S"{name=D}, bend left=15]
\arrow[from=A, to=B, phantom, "\dashv" sloped]
\arrow[from=C, to=D, phantom, "\dashv" sloped]
\end{tikzcd}    
\end{center}
Since adjunctions compose (see for example \cite[Chapter IV, §8, Theorem 1]{MacLane}), 
the composition $G\circ F_{\mathcal Sol}$
is left adjoint to $S\circ U_{\mathcal Sol}$, which is the forgetful functor
$U_{\mathcal E_{\mathrm{Inj}}}$.
This immediately yields the following theorem.

\color{black}
\begin{thm}\label{freesolutionisfreeskew}
    Let $X$ be a non-empty set and let $(\FSol_X, r_{\FSol_X})$ 
    be the free solution 
    on $X$ with inclusion map $i:X\to \FSol_X$.
    Then its structure skew brace 
     $G(\FSol_X,r_{\FSol_X})$ with inclusion map $g_{\FSol_X}\circ i:X\to G(\FSol_X, r_{\FSol_X})$ 
    is isomorphic to the free skew brace $\FSB_{X}$ on $X$.
\end{thm}

Since involutive solutions are always injective (see \cite{StrGr_LebVen}) and a solution is involutive if and only if its
structure skew brace is of abelian type, we also get the following result.
\begin{cor}
    Let $X$ be a non-empty set and let $(\FSol_{\Inv,X}, r_{\FSol_{\Inv,X}})$ 
    be the free involutive solution 
    on $X$ with inclusion map $i:X\to \FSol_{\Inv,X}$.
    Then its structure skew brace 
     $G(\FSol_{\Inv,X},r_{\FSol_{\Inv,X}})$ with inclusion map $g_{\FSol_{\Inv,X}}\circ i:X\to G(\FSol_{\Inv,X}, r_{\FSol_{\Inv,X}})$ 
    is isomorphic to the free skew brace of abelian type $\FB_{X}$ on~$X$. 
\end{cor}

With respect to Theorem \ref{freesolutionisfreeskew}, we note that the map $i: X \to \FSol_X$ is injective, 
since distinct elements $x,y$ of $X$ 
can be considered as a solution $\{x,y\}$ with the flip map.
By the universal property, considering a map $f:X\to \{x,y\}$
that separates $x$ and~$y$
(for example $f(y)=y$ and $f(a)=x$ for every $a\in X\setminus\{y\}$)
we obtain that $i(x)\neq i(y)$.

\smallskip

Moreover, $g_{\FSol_X}\circ i: X \to G(\FSol_X,r)$ is injective too.
In fact, given a skew brace $B$ containing $X$, since $(\FSol_X,r)$ is free on $X$,
there exists a morphism of solutions $f:(\FSol_X,r)\to (B,r_B)$ such that $f\circ i=\mathrm{id}_X$. By the universal property of the structure skew brace, there exists a homomorphism $G(f):G(\FSol_X,r)\to B$ such that $G(f)\circ g_{\FSol_X}=f$. Hence $G(f)\circ g_{\FSol_X}\circ i=\id_X,$
and $g_{\FSol_X}\circ i$ is injective. 

\smallskip

On the other hand, the injectivity of \(g_{\FSol_X}\) appears to be subtler. The difficulty is that the passage from solutions to skew braces may lose information. This phenomenon already occurs in a natural category of solutions considered below, where the corresponding free objects turn out to be non-injective; see Theorem~\ref{solnotinjective}.

\begin{question}\label{questionquestion}
Is $g_{\FSol_X}$ injective? In other words, is $\FSol_X$ an injective solution?
\end{question}

\subsection{Injectivisation and free injective solutions}\label{subsection2}

We now turn to injective solutions. The structure skew brace provides a natural way to pass from an arbitrary solution to an injective one, and this process satisfies the expected universal property. The next result makes this precise:
it is the solution-theoretic analogue of the universal property of the structure skew brace, and its proof relies on that result. It will be used in the proof of the subsequent theorem, which mirrors \cref{freesolutionisfreeskew} in the setting of solutions via injectivisation, and in the more general framework of arbitrary equational classes of solutions.

\begin{lem}\label{inj_universal_property}
Let $(X,r)$ be a solution.
Then for every injective solution $(Y,s)$ and every morphism of solutions $f:(X,r)\to (Y,s)$
there exists a unique morphism of solutions $\bar f:\Inj(X,r)\to (Y,s)$
such that $f=\bar f\circ \Inj_X$.
\end{lem}
\begin{proof}
By \cref{univ prop structure sb}, there exists a homomorphism of skew braces
\[
G(f):G(X,r)\to G(Y,s)
\]
such that $g_Y\circ f=G(f)\circ g_X$.
If $\Inj_X(a)=\Inj_X(b)$, then $g_X(a)=g_X(b)$, and hence
\[
g_Y(f(a))
=
G(f)(g_X(a))
=
G(f)(g_X(b))
=
g_Y(f(b)).
\]
Since $(Y,s)$ is injective, $g_Y$ is injective, and therefore $f(a)=f(b)$.

Thus $f$ is constant on the fibres of $\Inj_X$, so it factors uniquely through $\Inj_X$.
\end{proof}

Let $\mathcal E$ be an equational class of solutions and let
$$\mathcal E_{\Inj} \text{ denote its full subcategory of injective
solutions}.$$
\cref{inj_universal_property} implies that for every
$(X,r)\in\mathcal E$ and every injective solution $(Y,s)\in\mathcal E$,
composition with the injectivisation map $\Inj_X$
induces a natural bijection
\[
\Hom_{\mathcal E}((X,r),(Y,s))
\cong
\Hom_{\mathcal E_{\Inj}}(\Inj(X,r),(Y,s)).
\] Consequently, the injectivisation functor
is left adjoint to the inclusion \hbox{$I:\mathcal E_{\Inj}\hookrightarrow \mathcal E.$}
Recall that free objects are characterized by a universal property, or
equivalently by the fact that the corresponding free functor is left adjoint
to the forgetful functor.
So we have two adjunction as follows.
\begin{center}
\begin{tikzcd}[column sep=huge]
\mathcal{S}et
  \arrow[r, "F_{\mathcal E}"{name=A}, bend left=15]
& \mathcal E
  \arrow[l, "U_{\mathcal E}"{name=B}, bend left=15]
  \arrow[r, "\Inj"{name=C}, bend left=15]
& \mathcal E_{\mathrm{Inj}}
  \arrow[l,hook', "I"{name=D}, bend left=15]
\arrow[from=A, to=B, phantom, "\dashv" sloped]
\arrow[from=C, to=D, phantom, "\dashv" sloped]
\end{tikzcd}    
\end{center}
Since adjunctions compose (see for example \cite[Chapter IV, §8, Theorem 1]{MacLane}), 
the composition $\Inj\circ F_\varepsilon$
is left adjoint to $I\circ U_\mathcal{E}$, which is the forgetful functor
$U_{\mathcal E_{\mathrm{Inj}}}$.
This implies the following result.

\begin{thm}\label{injectivisation}
Let $\mathcal E$ be an equational class of solutions. Let $X$ be a
non-empty set and $i\colon X\hookrightarrow \FSol_{\mathcal E,X}$
be the canonical inclusion.
Then $\Inj(\FSol_{\mathcal E,X})$ is the
free injective solution in 
$\mathcal{E}_{\Inj}$ on 
$\Inj_{\FSol_{\mathcal E,X}}(i(X))$.

Moreover, if $\mathcal E$ contains an injective solution with at least two elements, then $|X|=|\Inj_{\FSol_{\mathcal E,X}}(i(X))|$.
\end{thm}
\begin{proof}
    The first part is already proven by the previous discussion.

    For the second part,
    assume that $\mathcal E$ contains an injective solution with at least
two elements and for simplicity let $\Inj=\Inj_{\FSol_{\mathcal E,X}}$. We claim that $\Inj\circ i\colon X\to \mathrm{Inj}(\FSol_{\mathcal E,X})$
is injective. Let $x,y\in X$ with $x\neq y$. Choose an injective solution
$(Y,s)\in\mathcal E_{\Inj}$ and two distinct elements $u,v\in Y$. Let
$\alpha\colon X\to Y$
be a map such that $\alpha(x)=u$ and $\alpha(y)=v$. By the universal property
of $\FSol_{\mathcal E,X}$, $\alpha$ extends to a morphism $f\colon \FSol_{\mathcal E,X}\to Y.$
Since $Y$ is injective, the morphism $f$ factors through $\Inj$. Hence, if
$\Inj(i(x))=\Inj(i(y))$, then
\[
u=\alpha(x)=f(i(x))=f(i(y))=\alpha(y)=v,
\]
a contradiction. Therefore $\Inj \circ i$ is injective.
\end{proof}

The injectivisation of a solution can be described explicitly inside its
structure skew brace.
In particular, starting from the free solution, we
obtain the following description of the free object in the subcategory of
injective solutions.

\begin{pro}
\label{free injective in structure brace}
    Let $X$ be a non-empty set and $\mathcal{E}$ be an equational class of solutions.
    Then \[
\widetilde{X}_{\mathcal{E}}=\left\{ a+\lambda_b(x)-a\mid x\in X\text{ and } a,b\in G\left(\FSol_{\mathcal{E},X},r_{\FSol_{\mathcal{E},X}}\right) \right\}
\]
is a subsolution of $G\left(\FSol_{\mathcal{E},X},r_{\FSol_{\mathcal{E},X}}\right)$
that is the 
free \textnormal(injective\textnormal) solution in $\mathcal{E}_{\Inj}$. 
\end{pro}
\begin{proof}
Recall first that for any solution $(Y,s)$, the structure skew brace 
$G\left(Y,s\right)$ is additively generated by $g_Y(Y)$.
Moreover,
for every  $x,y\in g_Y(Y)$,
\[
\rho_y(x)
=\overline{\lambda_x(y)}\circ \left(x+\lambda_x(y)\right)
=\overline{\lambda_x(y)}\circ x-\overline{\lambda_x(y)}
=\overline{\lambda_x(y)}+\lambda_{\overline{\lambda_x(y)}}(x)-\overline{\lambda_x(y)},
\]
so recalling from the preliminaries that  $\sigma_y(x)=\lambda_{y}\left(\rho_{\lambda^{-1}_x(y)}(x)\right)
$, we have that
\[
\sigma_y(x)=\lambda_y\left(\overline{y}+\lambda_{\overline{y}}(x)-\overline{y}\right)
=-y+x-y.
\]
Since $g_Y(Y)=\Inj(Y,s)$ is a subsolution  
of 
$(G(Y,s),r_{G(Y,s)})$, it is stable under all $\lambda$-maps
of $G(Y,s)$. Therefore $\lambda_b(x)\in g_Y(Y)$
for every $b\in G(Y,s)$ and every $x\in g_Y(Y)$.

Moreover we claim that 
$g_Y(Y)$ is also stable under additive
conjugation with elements of $G(Y,s)$. To show this, first note that  $-y+x-y=\sigma_y(x)$ and $y+x-y=\sigma_y^{-1}(x)\in g_Y(Y)$.
Thus $g_Y(Y)$ is stable under additive conjugation by elements of $g_Y(Y)$.
Since $G(Y,s)$ is additively generated by $g_Y(Y)$, every element
$a\in G(Y,s)$ can be written as a finite sum of elements of $g_Y(Y)$
and their additive inverses. 
Repeated use of  the above conjugation
stability yields $a+z-a\in g_Y(Y)$
for all $a\in G(Y,s)$ and all $z\in g_Y(Y)$.

Therefore, given a non-empty subset $T$ of $g_Y(Y)$, 
\[
\left\{ a+\lambda_b(t)-a\mid t\in T,\; a,b\in G\left(Y,s\right) \right\}\subseteq g_Y(Y)=\Inj(Y,s)
\]
and it is a subsolution.
In particular,
\[
\widetilde{X}_{\mathcal{E}}=\left\{ a+\lambda_b(x)-a\mid x\in X,\; a,b\in G\left(\FSol_{\mathcal{E},X},r_{\FSol_{\mathcal{E},X}}\right) \right\}
\]
is a subsolution of $\Inj\left(\FSol_{\mathcal{E},X},r_{\FSol_{\mathcal{E},X}}\right)$.

Furthermore,
denoting by $i\colon X\hookrightarrow \FSol_{\mathcal E,X}$
the canonical inclusion,
by \cref{injectivisation}, 
there is a unique morphism of solutions
$\varphi:\Inj\left(\FSol_{\mathcal{E},X},r_{\FSol_{\mathcal{E},X}}\right)\to \widetilde{X}_{\mathcal{E}}$
such that $\varphi\circ \Inj\circ i=\Inj\circ i$.
Therefore the unique map of solutions $$\psi:\Inj\left(\FSol_{\mathcal{E},X},r_{\FSol_{\mathcal{E},X}}\right)\to \Inj\left(\FSol_{\mathcal{E},X},r_{\FSol_{\mathcal{E},X}}\right)$$
such that 
\begin{center}
\begin{tikzcd}
     & \Inj\left(\FSol_{\mathcal{E},X},r_{\FSol_{\mathcal{E},X}}\right)\arrow[d,"\varphi" description]\arrow[dd, dashed, bend left=40, "\exists!\psi"]\\
     X \arrow[r,"\Inj\circ i" description]\arrow[ru,"\Inj\circ i" description]\arrow[rd,"\Inj\circ i" description]& \widetilde{X}_{\mathcal{E}}\arrow[d, phantom, "\mathrel{\rotatebox{-90}{$\subseteq$}}"]\\
      & \Inj\left(\FSol_{\mathcal{E},X},r_{\FSol_{\mathcal{E},X}}\right) .
\end{tikzcd}    
\end{center}
has to be the identity, and thus $\widetilde{X}_{\mathcal{E}}=\Inj\left(\FSol_{\mathcal{E},X},r_{\FSol_{\mathcal{E},X}}\right)$.
\end{proof}

Since involutive solutions are always injective (see \cite{StrGr_LebVen}), we can apply \cref{injectivisation} to $\mathcal{E}=\Inv$, the class of involutive solutions, and get the following result. 
\begin{cor}
\label{free injective in G}
    Let $X$ be a non-empty set.
    Then \[
\widetilde{X}_{\Inv}=\left\{\lambda_b(x)\mid x\in X\text{ and } b\in G\left(\FSol_{\Inv,X},r_{\FSol_{\Inv,X}}\right) \right\}
\]
is a subsolution of $G\left(\FSol_{\Inv,X},r_{\FSol_{\Inv,X}}\right)$
that is the 
free involutive solution on~$X$.
\end{cor}

\subsection{Right nilpotent solutions and left multipermutation solutions}\label{subsection3}

We now turn to the class of solutions that reflects right nilpotency of skew braces. More precisely, our next aim is to introduce a family of solutions corresponding to skew braces of right nilpotency class at most $n$ (see
Lemmas~\ref{eqsolskewbrace}, \ref{lemmmino-class-n}, and
Theorem~\ref{StrGrpFSBn}). This requires some preliminary notions, including the monoid $W(X)$, the left retract, and the left multipermutation level. These will also allow us to compare right nilpotent solutions with left multipermutation solutions.

\medskip

Let $(X,r)$ be a solution, and consider the subgroup $P(X)$ of $\Sym(X)$ generated by the elements $\lambda_x$ and $\rho_x$ for  $x\in X$. Let $W(X)$ be the submonoid of $X^X$ generated by $P(X)$ and for all $z\in X$, the elements $D_z\colon X\to X$ defined by $D_z(x)=\rho_x(z)$ for all $x\in X$. We denote the canonical action of the monoid $W(X)$ on $X$ by $m\cdot x$ for all $x\in X$ and $m\in W(X)$. We define an equivalence relation $\sim_l$ on $X$ as follows:
$$x\sim_ly \quad \text{ if and only if } \quad \lambda_{m\cdot x}=\lambda_{m\cdot y} \text{ for all } m\in W(X).$$

\begin{rem}\label{rem:small}
Clearly, $x\sim_ly$ imply $m\cdot x\sim_l m\cdot y$ for all $m\in W(X)$. 
\end{rem}

\begin{lem} \label{inducedretract}
Let $(X,r)$ be a solution of the YBE. Then $(X/\sim_l,\tilde r)$, where \[\tilde{r}([x],[y])=([\lambda_x(y)],[\rho_x(y)]),\] is a solution as well, called the \emph{left retract} of $(X,r)$, denoted $\operatorname{Ret}_l(X,r)$. 
\end{lem}
\begin{proof}
By \cref{nondeg implies bij}, it is enough to prove that the maps
$\lambda$, $\lambda^{-1}$, $\rho$ and $\rho^{-1}$ induce well-defined maps on $X/\sim_l$. Let \hbox{$x,y,u,v$} be elements of $X$ such that $x\sim_l y$ and $u\sim_lv$. We must show that $\lambda_x(u)\sim_l\lambda_y(v)$, $\lambda^{-1}_x(u)\sim_l\lambda^{-1}_y(v)$, $\rho_x(u)\sim_l\rho_y(v)$
        and $\rho^{-1}_x(u)\sim_l\rho^{-1}_y(v)$.

Since $\lambda_x,\lambda_x^{-1}\in P(X)$, we have by Remark~\ref{rem:small} that 
\[
\lambda_x(u)
\sim_l
\lambda_x(v)
=
\lambda_y(v)\quad \text{and}\quad\lambda_x^{-1}(u)
\sim_l
\lambda_x^{-1}(v)
=
\lambda_y^{-1}(v).
\]

Similarly, since $\rho_x\in P(X)$, one has $\rho_x(u)\sim_l\rho_x(v)$. We also have $D_v(x)\sim_l D_v(y)$ which means, $\rho_x(v)\sim_l\rho_y(v)$. Thus,
\[
\rho_x(u)\sim_l\rho_y(v).
\]
Put $a=\rho_y^{-1}(v)$. Then $\rho_y(a)=v$. Since $x\sim_l y$, by the above 
\begin{equation}
\label{eq:rhoinv}
\rho_x(a)\sim_l\rho_y(a)=v.
\end{equation}
Because \(\rho_x^{-1}\in P(X)\), we may apply \(\rho_x^{-1}\) to~\eqref{eq:rhoinv} and we obtain
\[
a
=
\rho_x^{-1}\bigl(\rho_x(a)\bigr)
\sim_l
\rho_x^{-1}(v).
\]
Since \(a=\rho_y^{-1}(v)\), this means
\[
\rho_x^{-1}(v)\sim_l\rho_y^{-1}(v).
\]

In addition, $u\sim_l v$ and $\rho_x^{-1}\in P(X)$ imply $\rho_x^{-1}(u)\sim_l\rho_x^{-1}(v)$.
Hence, we obtain $\rho_x^{-1}(u)\sim_l\rho_y^{-1}(v)$.
This completes the proof.
\end{proof}

\begin{rem}
{\it Among all equivalence relations $\equiv$ on $X$  satisfying  $$x\equiv y \Longrightarrow \lambda_x=\lambda_y,$$ the relation $\sim_l$ is the largest one for which the formula
\[
\overline r([x]_\equiv,[y]_\equiv)=([\lambda_x(y)]_\equiv,[\rho_y(x)]_\equiv)
\]
is well-defined and is a solution on the quotient $X/{\equiv}$.}
\end{rem}
\begin{proof}
For such an equivalence relation, it is clear that we should have $x\equiv y\implies m\cdot x\equiv m\cdot y$ for all $x,y\in X$ and $m\in W(X)$ since for all $x\in X$ the maps \hbox{$[y]_\equiv\mapsto [\lambda_x(y)]_\equiv$,} $[y]_\equiv\mapsto [\lambda^{-1}_x(y)]_\equiv$, $[y]_\equiv\mapsto [\rho_x(y)]_\equiv$, $[y]_\equiv\mapsto [\rho^{-1}_x(y)]_\equiv$ and\linebreak \hbox{$[y]_\equiv\mapsto [\rho_y(x)]_\equiv$} should be well defined on the quotient. Hence the relation $\equiv$ has to be a refinement of the relation $\sim_l$.
\end{proof}

In the following we write
$\operatorname{Ret}_l^0(X,r)=(X,r)$ and, for every $n\geq 1$, we denote by
$\operatorname{Ret}^n_l(X,r)$ the solution obtained after left retracting $n$ times.

\begin{defn}
    Let $(X,r)$ be a solution. We say that $(X,r)$ is {\it left multipermutation of level $n$} 
    if there exists an integer $n\geq 0$ such that $\operatorname{Ret}_l^n(X,r)$ reduces to a singleton while $|\operatorname{Ret}^{n-1}_l(X,r)|\geq2$.
\end{defn}

Thus a left multipermutation solution is a solution whose iterated left retract eventually becomes the trivial solution on a singleton. This notion is a left-handed analogue  of the multipermutation solutions, mentioned in Section \ref{preliminaries}, 
of the usual notion of a multipermutation solution, where the retraction $\Ret^n(X,r)$  
is defined using the  equivalence relation identifying elements with the same left and right actions.

\medskip

As for right nilpotent skew braces, we show that these left multipermutation solutions can be characterised in terms of the bijections
\[
\lambda^{\alpha_n}_{
  m_n\cdot\left(\lambda^{\alpha_{n-1}}_{
    \ddots_{
      m_2\cdot \left(\lambda^{\alpha_1}_{v_1}(v_2)\right)
    }\iddots
  }(v_n)\right)
}(-)
\]
where $n\geq1$ is an integer, $(\alpha_1,\dots ,\alpha_n)\in \{\pm1\}^{\times n}$, $(v_1,\dots,v_n)\in B^{\times n}$ and \\$(m_2,\dots, m_{n})\in W(X)^{n-1}$. Therefore  we extend the notations introduced in Definition~\ref{def:nested} to represent these maps.  Note that here we use the convention that
    $W(X)^{\times 0}$ consists of the empty tuple, denoted by $\emptyset$.

\begin{defn}
    Let $(X,r)$ be a solution. 
    For $k\geq 1$, let
    $\alpha=(\alpha_1,\dots,\alpha_k)\in\{\pm 1\}^{\times k}$,
    $v=(v_1,\dots,v_k)\in X^{\times k}$, and
    $m=(m_2,\dots,m_k)\in W(X)^{\times(k-1)}$.
    We define the map $\lambda_{m,v}^{\alpha}\colon X\to X$ recursively as
    follows.

    If $k=1$, then $m=\emptyset$ and
    \[
    \lambda_{\emptyset,(v_1)}^{(\alpha_1)}
    =
    \lambda_{v_1}^{\alpha_1}.
    \]
    If $k>1$, then
    \[
    \lambda_{m,v}^{\alpha}
    =
    \lambda^{\alpha_k}_{m_k\cdot
    \lambda_{\mathfrak r m,\mathfrak r v}^{\mathfrak r\alpha}(v_k)}.
    \]
    Note that in the case $m_i=\id$ for all $2\leq i\leq k$, we recover the
    maps defined in~Definition~\ref{def:nested}.
\end{defn}

For notational purposes, we extend the action of $W(X)$ on $X$ to an action on~$X^{\times k}$ for all integer $k\geq 1$ as follows
\[m\cdot (v_1,v_2,\dots,v_k)=(mv_1,v_2,\dots,v_k).\]

\begin{defn}
A solution $(X,r)$ is right nilpotent of class at most $1$ if $\lambda_x=\id$ for all $x\in X$. For $n>1$, 
we say that a solution $(X,r)$ is {\it right nilpotent of class at most $n$}, if  $\lambda_{m,v}^\alpha =\lambda_{\mathfrak{l}m,m_2\cdot\mathfrak{l}v}^{\mathfrak{l}\alpha}$
for all $\alpha\in \{\pm1\}^{\times n}$, $m=(m_2,\dots, m_n)\in W(X)^{\times {n-1}}$ and $v\in X^{\times n}$. 
The category consisting of these solutions, and their morphisms, is denoted $\mathcal{RNS}$. The full subcategory of right nilpotent solutions of class at most $n$ is denoted by $\mathcal{RNS}_{n}$.
\end{defn}

\begin{rem}
{\it A solution $(X,r)$ is right nilpotent of class at most $2$ if and only if 
\begin{equation}\label{eq:class2}
\lambda_{\lambda_x(y)}=\lambda_y
\end{equation}
for all $x,y\in X$.}
\end{rem} 
\begin{proof}
    It is clear that for a right nilpotent solution of class at most $2$, the equation~\eqref{eq:class2}  holds.
    Conversely, assume~\eqref{eq:class2} holds. 
    Notice that then for all $x,y\in X$, we have 
    \[\lambda_y=\lambda_{\lambda_x(\lambda^{-1}_x(y))}=\lambda_{\lambda_x^{-1}(y)}\]
    and 
    \[\lambda_{\rho_x(y)}=\lambda^{-1}_{\lambda_y(x)}\lambda_y\lambda_x=\lambda^{-1}_x\lambda_y\lambda_x.
    \]
    Thus 
    \[
    \lambda_y=\lambda_{\rho_x(\rho_x^{-1}(y))}=\lambda_{x}^{-1}\lambda_{\rho_x^{-1}(y)}\lambda_x,
    \]
    hence,
    \[\lambda_{\rho_x^{-1}(y)}=\lambda_x\lambda_y\lambda_x^{-1}.\]
    We will now show by induction on the length of $m\in W(X)$ that 
    \[
    \lambda_{m\cdot \lambda^{\alpha}_x(y)}=\lambda_{m\cdot y}
    \]
    for all $x,y\in X$ and $\alpha\in \{\pm 1\}$. When $m$ is the identity, this is clear by the above. Suppose $m=\sigma_z m'$ with $\sigma\in \{\lambda,\rho,\lambda^{-1},\rho^{-1},D\}$ and the statement holds for $m'$. If $\sigma=\lambda^{\pm 1}$ then 
    \[\lambda_{m\cdot \lambda_x(y)}=\lambda_{m'\cdot \lambda_x(y)}=\lambda_{m'\cdot y}=\lambda_{m\cdot y}.\]
    If $\sigma =\rho$, then
    \[
    \lambda_{m\cdot \lambda_x(y)}=\lambda_{z}\lambda_{m'\cdot \lambda_x(y)}\lambda_z^{-1}=\lambda_{z}\lambda_{m'\cdot y}\lambda_z^{-1}=\lambda_{m\cdot y}.
    \]
    A similar argument applies to the case $\sigma=\rho^{-1}$. Finally, if $\sigma =D$, then
    \[
    \lambda_{m\cdot \lambda_x(y)}=\lambda_{\rho_{m'\cdot \lambda_x(y)}(z)}=\lambda_{m'\cdot \lambda_x(y)}^{-1}\lambda_z\lambda_{m'\cdot \lambda_x(y)}=\lambda_{m'\cdot y}^{-1}\lambda_z\lambda_{m'\cdot y}=\lambda_{m\cdot y}.
    \] The statement is proved.
\end{proof}

\begin{pro}\label{propmult}
    Let $n\geq 0$ be an integer. A solution $(X,r)$ is right nilpotent of class at most $n+1$ if and only if $(X,r)$ is left multipermutation of level at most $n$.
\end{pro}
\begin{proof}
    Assume first that $(X,r)$ is left multipermutation of level at most $n$. We proceed by induction on $n$. If $n=0$, then $X$ is a singleton and the statement is trivial. Assume $n>0$ and that the statement is true for $n-1$. Then, $\operatorname{Ret}_l(X,r)$ is left multipermutation of class at most $n-1$, so by induction hypothesis, it is right nilpotent of class at most $n$. Let $\alpha\in \{\pm1\}^{\times (n+1)}$, $m=(m_2,\dots, m_{n+1})\in W(X)^{\times n}$ and $v=(v_1,\dots,v_{n+1})\in X^{\times (n+1)}$. Hence, 
     $$\lambda_{\mathfrak{r}m,\mathfrak{r}v}^{\mathfrak{r}\alpha}(v_{n+1})\sim_l \lambda_{\mathfrak{l}(\mathfrak{r}m),m_2\cdot \mathfrak{l}(\mathfrak{r}v)}^{\mathfrak{l}(\mathfrak{r}\alpha)}(v_{n+1}),$$
    since they coincide in the quotient. 

Put
\[
A=\lambda_{\mathfrak{r}m,\mathfrak{r}v}^{\mathfrak{r}\alpha}(v_{n+1})
\qquad\text{and}\qquad
B=
\lambda_{\mathfrak{l}(\mathfrak{r}m),\,m_2\cdot\mathfrak{l}(\mathfrak{r}v)}
^{\mathfrak{l}(\mathfrak{r}\alpha)}(v_{n+1}),
\] so $A\sim_l B$.
Since \(m_{n+1}\in W(X)\), Remark~\ref{rem:small} gives
\[
m_{n+1}\cdot A\sim_l m_{n+1}\cdot B.
\]
Taking the identity element of \(W(X)\) in the definition of \(\sim_l\), we get
\[
\lambda_{m_{n+1}\cdot A}
=
\lambda_{m_{n+1}\cdot B}.
\]
Therefore also
\[
\lambda_{m_{n+1}\cdot A}^{\alpha_{n+1}}
=
\lambda_{m_{n+1}\cdot B}^{\alpha_{n+1}},
\]
because, if \(\alpha_{n+1}=-1\), equality of the maps implies equality of their
inverses. Thus
\begin{align*}
\lambda^{\alpha_{n+1}}_{m_{n+1}\cdot
\lambda_{\mathfrak{r}m,\mathfrak{r}v}^{\mathfrak{r}\alpha}(v_{n+1})}
&=
\lambda^{\alpha_{n+1}}_{m_{n+1}\cdot
\lambda_{\mathfrak{l}(\mathfrak{r}m),\,m_2\cdot\mathfrak{l}(\mathfrak{r}v)}
^{\mathfrak{l}(\mathfrak{r}\alpha)}(v_{n+1})}  \\
&=
\lambda^{\alpha_{n+1}}_{m_{n+1}\cdot
\lambda_{\mathfrak{r}(\mathfrak{l}m),\,
\mathfrak{r}(m_2\cdot\mathfrak{l}v)}
^{\mathfrak{r}(\mathfrak{l}\alpha)}(v_{n+1})} \\
&=
\lambda_{\mathfrak{l}m,\,m_2\cdot\mathfrak{l}v}^{\mathfrak{l}\alpha}.
\end{align*}
Here we used the identities
\[
\mathfrak{l}(\mathfrak{r}m)=\mathfrak{r}(\mathfrak{l}m),
\qquad
m_2\cdot\mathfrak{l}(\mathfrak{r}v)
=
\mathfrak{r}(m_2\cdot\mathfrak{l}v),
\qquad
\mathfrak{l}(\mathfrak{r}\alpha)=\mathfrak{r}(\mathfrak{l}\alpha),
\]
and the last equality is precisely the recursive definition of the maps $\lambda_{m,v}^{\alpha}$

Since $\alpha$, $m$ and $v$ were arbitrary, this proves that \[ \lambda_{m,v}^{\alpha}=\lambda_{\mathfrak l m,\,m_2\cdot\mathfrak l v}^{\mathfrak l\alpha}\]
for every $\alpha\in\{\pm1\}^{\times(n+1)}$, every
$m\in W(X)^{\times n}$, and every $v\in X^{\times(n+1)}$. Hence $(X,r)$ is right
nilpotent of class at most $n+1$.

For the converse, assume that $(X,r)$ is right nilpotent of class at most $n$. As before, we
proceed by induction on $n$. For $n=1$, it is clear since $\lambda_x=\id$ for
all $x\in X$. Assume that $n>1$.
We claim that the left retract  $\operatorname{Ret}_l(X,r)$ is right nilpotent of class at most $n-1$.
To prove this, let $\beta=(\beta_1,\dots,\beta_{n-1})\in\{\pm1\}^{\times(n-1)}$,
let $M=(M_2,\dots,M_{n-1})\in W(\operatorname{Ret}_l(X,r))^{\times(n-2)}$, and let
$V=(V_1,\dots,V_{n-1})\in (\operatorname{Ret}_l(X,r))^{\times(n-1)}$. Choose representatives
$v_i\in X$ such that $V_i=[v_i]_{\sim_l}$. Also choose elements
$m_i\in W(X)$ inducing $M_i$ on the quotient, for every \hbox{$i=2,\dots,n-1$.} 
This is possible because the quotient map $X\to \operatorname{Ret}_l(X,r)$ is a morphism of
solutions and the monoid $W(-)$ is generated functorially by
$\lambda,\lambda^{-1},\rho,\rho^{-1}$ and the maps $D_z$.
We have to show that the following equality of maps on the left retract  solution $\operatorname{Ret}_l(X,r)$:
\[
\lambda_{M,V}^{\beta}
=
\lambda_{\mathfrak l M,\,M_2\cdot\mathfrak l V}^{\mathfrak l\beta}.
\]
To prove this, we evaluate these maps on an arbitrary class
$[t]_{\sim_l}\in \operatorname{Ret}_l(X,r)$.
Put
\[
A=\lambda_{m,v}^{\beta}(t)
\qquad\text{and}\qquad
B=\lambda_{\mathfrak l m,\,m_2\cdot\mathfrak l v}^{\mathfrak l\beta}(t),
\]
where $m=(m_2,\dots,m_{n-1})$ and $v=(v_1,\dots,v_{n-1})$. We need to show
that $A\sim_l B$.

By definition of $\sim_l$, we need to show that
$$\lambda_{q\cdot A}=\lambda_{q\cdot B},$$ for every $q\in W(X)$. Fix such a
$q$. Since $(X,r)$ is right nilpotent of class at most $n$, we apply the
defining identity to the sign tuple $(\beta_1,\dots,\beta_{n-1},1)$, to the
tuple $(v_1,\dots,v_{n-1},t)\in X^{\times n}$, and to the element
$(m_2,\dots,m_{n-1},q)\in W(X)^{\times(n-1)}$. Hence
\[
\lambda_{(m_2,\dots,m_{n-1},q),(v_1,\dots,v_{n-1},t)}
^{(\beta_1,\dots,\beta_{n-1},1)}
=
\lambda_{(m_3,\dots,m_{n-1},q),\,m_2\cdot(v_2,\dots,v_{n-1},t)}
^{(\beta_2,\dots,\beta_{n-1},1)}.
\]
By the recursive definition of the maps $\lambda_{m,v}^{\alpha}$, the left
hand side is $\lambda_{q\cdot \lambda_{m,v}^{\beta}(t)}=\lambda_{q\cdot A}$, 
whereas the right hand side is $\lambda_{q\cdot
\lambda_{\mathfrak l m,\,m_2\cdot\mathfrak l v}^{\mathfrak l\beta}(t)}
=
\lambda_{q\cdot B}$. Thus, indeed,  $\lambda_{q\cdot A}=\lambda_{q\cdot B}$ for every $q\in W(X)$. Hence
$A\sim_l B$. This proves that
\[
\lambda_{M,V}^{\beta}([t])
=
\lambda_{\mathfrak l M,\,M_2\cdot\mathfrak l V}^{\mathfrak l\beta}([t])
\]
for every $[t]\in \operatorname{Ret}_l(X,r)$. Therefore
$\lambda_{M,V}^{\beta}
=
\lambda_{\mathfrak l M,\,M_2\cdot\mathfrak l V}^{\mathfrak l\beta}$, and
$\operatorname{Ret}_l(X,r)$ is right nilpotent of class at most $n-1$.

By the induction hypothesis, $\operatorname{Ret}_l(X,r)$ is left multipermutation of level at most $n-2$. Hence $(X,r)$ is left multipermutation of level at most
$n-1$.
\end{proof}

We now show that right nilpotency of solutions and right nilpotency of skew braces are strongly linked.

\begin{lem}\label{eqsolskewbrace}
    A skew brace $B$ is right nilpotent of class at most $n$ if and only if its associated solution $(B,r_B)$ is right nilpotent of class at most $n$.
\end{lem}
\begin{proof}
    We proceed by induction. For $n=1$, $B$ is a trivial skew brace if and only if $\lambda_b=\id$ for all $b\in B$, so the statement becomes trivial. Suppose  that $n>1$.
    Assume first that $B$ is right nilpotent of class at most $n$. Let $\alpha\in \{\pm1\}^{\times n}$,
    $m=(m_2,\dots, m_{n})\in W(B)^{\times (n-1)}$
    and $v=(v_1,\dots,v_{n})\in B^{\times n}$.
    By the induction hypothesis, the solution associated to  $B/B^{(n)}$ is nilpotent of class at most $n-1$. 
    The maps in $W(B)$ induce maps on the quotient $B/B^{(n)}$. Hence, applying the induction hypothesis in~$B/B^{(n)}$, the elements
    \[
    m_n\cdot \lambda_{\mathfrak{r}m,\mathfrak{r}v}^{\mathfrak{r}\alpha}(v_n)
    \quad\text{and}\quad
    m_n\cdot
    \lambda_{\mathfrak{l}(\mathfrak{r}m),\,m_2\cdot\mathfrak{l}(\mathfrak{r}v)}
    ^{\mathfrak{l}(\mathfrak{r}\alpha)}(v_n)
    \]
    have the same image in $B/B^{(n)}$.
    Hence, there exists $z\in B^{(n)}\subseteq \ker(\lambda)$ such that 
    \[
    \lambda^{\alpha_{n}}_{m_{n}\cdot \lambda_{\mathfrak{r}m,\mathfrak{r}v}^{\mathfrak{r}\alpha}(v_{n})}
    =
    \lambda^{\alpha_{n}}_{z\circ\left(m_{n}\cdot \lambda_{\mathfrak{l}(\mathfrak{r}m),m_2\cdot \mathfrak{l}(\mathfrak{r}v)}^{\mathfrak{l}(\mathfrak{r}\alpha)}(v_{n})\right)}
    =
    \lambda^{\alpha_{n}}_{m_{n}\cdot \lambda_{\mathfrak{r}(\mathfrak{l}m), \mathfrak{r}(m_2\cdot\mathfrak{l}v)}^{\mathfrak{r}(\mathfrak{l}\alpha)}(v_{n})}
    =
    \lambda_{\mathfrak{l}m,m_2\cdot\mathfrak{l}v}^{\mathfrak{l}\alpha}.
    \]
    Therefore, the solution of $B$ is right nilpotent of class at most $n$.

    Conversely, assume that the associated solution of $B$ is right nilpotent of class at most $n$. We prove that $B$ is right nilpotent of class at most $n$. By Proposition~\ref{lambda action for rclass n}, it is enough to show that
    \[
    \lambda_v^\alpha=\lambda_{\mathfrak l v}^{\mathfrak l\alpha}
    \]
    for every $\alpha\in\{\pm1\}^{\times n}$ and every $v\in B^{\times n}$. Since the associated solution of $B$ is right nilpotent of class at most $n$, we may take $m_i=\id$ for every $i=2,\dots,n$ in the defining identities of right nilpotency for solutions. Then
    \[
    \lambda_{m,v}^{\alpha}
    =
    \lambda_{\mathfrak l m,\,m_2\cdot\mathfrak l v}^{\mathfrak l\alpha}.
    \]
    With this choice of $m$, the left-hand side is precisely $\lambda_v^\alpha$, while the right-hand side is precisely $\lambda_{\mathfrak l v}^{\mathfrak l\alpha}$. Hence $\lambda_v^\alpha=\lambda_{\mathfrak l v}^{\mathfrak l\alpha}$
    or all $\alpha$ and $v$. Therefore, $B$ is right nilpotent of class at most~$n$.
\end{proof}

\begin{lem}\label{lemmaybecanberemoved}
    Let $(X,r)$ be a solution and let $\iota \colon X\to G(X,r)$ be the
    canonical map. If~\hbox{$x,y\in X$} are such that $x\sim_l y$, then, for
    every $m\in W(X)$,
    \[
    \lambda_{\iota(m\cdot x)}=\lambda_{\iota(m\cdot y)}
    \]
    as automorphisms of the additive group of $G(X,r)$. In particular,
    $\lambda_{\iota(x)}=\lambda_{\iota(y)}$.
\end{lem}
\begin{proof}
Let $m\in W(X)$. Since $x\sim_l y$, by definition, we have $\lambda_{m\cdot x}=\lambda_{m\cdot y}$ as permutations of $X$.

Put $G=G(X,r)$. The lambda action of this  structure skew brace extends the
left action of the solution on $\iota (X)$. Hence, for every $z\in X$, we have
\[
\lambda_{\iota(m\cdot x)}(\iota(z))
=
\iota(\lambda_{m\cdot x}(z))
=
\iota(\lambda_{m\cdot y}(z))
=
\lambda_{\iota(m\cdot y)}(\iota(z)).
\]
Thus $\lambda_{\iota(m\cdot x)}$ and $\lambda_{\iota(m\cdot y)}$ coincide on
$\iota(X)$. Since the additive group of $G(X,r)$ is generated by $\iota(X)$,
the two additive automorphisms coincide on all elements of $G(X,r)$. Therefore $\lambda_{\iota(m\cdot x)}=\lambda_{\iota(m\cdot y)}$.
\end{proof}

\begin{lem}\label{lemx}
    Let $(X,r)$ be a left multipermutation solution of class at most $n$, then its structure skew brace is right nilpotent of class at most $n+1$.
\end{lem}
\begin{proof}
We argue by induction on $n$. 
If $n=0$, then the result is clear.
Assume now that \(n>0\), and suppose that the result holds for left
multipermutation solutions of class at most \(n-1\). Put $G=G(X,r)$ and let $p\colon X\to \operatorname{Ret}_l(X,r)=X/{\sim_l}$
be the canonical projection. Because of Lemma~\ref{inducedretract},
the left retract
$\operatorname{Ret}_l(X,r)$ carries a canonical solution,
and \(p\) is a morphism of solutions. Hence, by the universal property of the
structure skew brace, \(p\) induces a surjective skew brace morphism
\[
\Phi\colon G(X,r)\longrightarrow G(\operatorname{Ret}_l(X,r))
\]
such that
\[
\Phi(\iota(x))=\iota_2([x]_{\sim_l})
\]
for every \(x\in X\), where
\[
\iota_2\colon \operatorname{Ret}_l(X,r)\to G(\operatorname{Ret}_l(X,r))
\]
is the canonical map.

Set $I=\ker(\Phi)$. Then \(I\) is an ideal of \(G\). We claim that $I\subseteq \ker(\lambda_G)$. Indeed, at the level of multiplicative groups, the structure group
\((G(X,r),\circ)\) is generated by the elements \(\iota(x)\), with \(x\in X\),
subject to the defining relations coming from the solution. Similarly,
\((G(\operatorname{Ret}_l(X,r)),\circ)\) is generated by the elements
\(\iota_2([x]_{\sim_l})\), with the corresponding relations in the left retract solution.
The morphism
\[
\Phi\colon (G(X,r),\circ)\longrightarrow (G(\operatorname{Ret}_l(X,r)),\circ)
\]
is induced by the map
\[
\iota(x)\longmapsto \iota_2([x]).
\]

Let $N$ be the normal subgroup of $(G,\circ)$ generated by the elements
$\overline{\iota(x)}\circ\iota(y)$, $x\sim_l y$.
Clearly $N\subseteq\ker(\Phi)$. Conversely, since the projection
$p\colon X\to \operatorname{Ret}_l(X,r)$ is a morphism of solutions, the defining relations of
the structure group of $\operatorname{Ret}_l(X,r)$ are precisely the images of the defining
relations of the structure group of $X$, after identifying $\iota(x)$ with
$\iota(y)$ whenever $x\sim_l y$. Thus the map
$\iota_2([x]_{\sim_l})\mapsto \iota(x)N$
is a well-defined group homomorphism
$(G(\operatorname{Ret}_l(X,r)),\circ)\to (G,\circ)/N$
inverse to the homomorphism induced by $\Phi$. Hence
$(G,\circ)/N\simeq (G(\operatorname{Ret}_l(X,r)),\circ),$
and so $\ker(\Phi)=N$.

By \cref{lemmaybecanberemoved}, if \(x\sim_l y\), then $\lambda_{\iota(x)}=\lambda_{\iota(y)}$. Hence
\[
\lambda_{\overline{\iota(x)}\circ\iota(y)}
=
\lambda_{\iota(x)}^{-1}\lambda_{\iota(y)}
=
\id.
\]
Thus each generator of the multiplicative kernel of \(\Phi\) belongs to
\(\ker(\lambda_G)\). Since
\[
\lambda\colon (G,\circ)\to\Aut(G,+)
\]
is a group homomorphism, \(\ker(\lambda_G)\) is a normal subgroup of
\((G,\circ)\). Therefore $I\subseteq \ker(\lambda_G)$, as claimed.

Now, since \((X,r)\) is left multipermutation of class at most \(n\), the
solution $\operatorname{Ret}_l(X,r)$ is left multipermutation of class at most
\(n-1\). By the induction hypothesis, $G(\operatorname{Ret}_l(X,r))$ is right nilpotent of class at most \(n\). Since
\[
G/I\simeq G(\operatorname{Ret}_l(X,r)),
\]
we have $(G/I)^{(n+1)}=0$, and thus  also $G^{(n+1)}\subseteq I$.
Using \(I\subseteq \ker(\lambda_G)\), we obtain
\[
G^{(n+2)}
=
G^{(n+1)}*G
\subseteq
I*G
\subseteq
\ker(\lambda_G)*G
=
0.
\]
Thus \(G\) is right nilpotent of class at most \(n+1\). This completes the
induction.
\end{proof}

\begin{lem}\label{lemmmino-class-n}
Let $(X,r)$ be a solution in the category $\mathcal{RNS}_n$. Then its
structure skew brace $G(X,r)$ is right nilpotent of class at most $n$.
\end{lem}

\begin{proof}
The case $n=1$ is clear.  Assume now that $n>1$. Since $(X,r)$ belongs to~$\mathcal{RNS}_n$, \cref{propmult} gives that $(X,r)$ is left multipermutation of class at most $n-1$. By Lemma~\ref{lemx}, we obtain that $G(X,r)$
is right nilpotent of class at most $(n-1)+1=n$.
This completes the proof.
\end{proof}

\begin{pro}
The category $\mathcal{RNS}_n$ forms an equational class. In particular, it is
a subvariety of the variety of solutions.
\end{pro}
\begin{proof}
We use the same signature $\mathcal F=
\{\lambda,\rho,
\widetilde{\lambda},\widetilde{\rho}\}$
as in the proof of \cref{solutions are variety}, so the operations
$\widetilde{\lambda}$ and $\widetilde{\rho}$ encode the inverse maps $\lambda_x^{-1}$ and $\rho_x^{-1}$.

Let $\mathcal W$ be the smallest set of unary $\mathcal F$-terms in the
variable $z$ such that $z\in\mathcal W$ and such that, whenever
$M(z)\in\mathcal W$ and $t$ is a variable, the terms
\[
\lambda(t,M(z)),\quad
\widetilde{\lambda}(t,M(z)),\quad
\rho(t,M(z)),\quad
\widetilde{\rho}(t,M(z)),\quad
\rho(M(z),t)
\]
also belong to $\mathcal W$. In every solution $(X,r)$, after assigning values
to the auxiliary variables, the interpretations of the terms in~$\mathcal W$
are precisely the elements of the monoid $W(X)$.
Here the last term $\rho(M(z),t)$ represents the map
$D_t$, since $D_t(x)=\rho_x(t)$.

If $M(z)\in\mathcal W$ is a unary term and $T$ is another
term, we write $M(T)$ for the term obtained from $M(z)$ by substituting $T$ for
the distinguished variable $z$.

For every integer $k\geq 1$, every
\[
\alpha=(\alpha_1,\dots,\alpha_k)\in\{\pm1\}^{\times k},
\]
every
\[
v=(v_1,\dots,v_k),
\]
and every tuple of unary terms
\[
M=(M_2,\dots,M_k)\in\mathcal W^{\times(k-1)},
\]
we define recursively a term $\Lambda_{M,v}^{\alpha}(z)$ which represents the value of the map $\lambda_{m,v}^{\alpha}$ at $z$.
Here $v_1,\dots,v_k,z$, together with the auxiliary
variables occurring in the terms $M_i$, are variables of the equational
language.

For $k=1$, we put
\[
\Lambda_{\emptyset,(v_1)}^{(\alpha_1)}(z)
=
\lambda^{\alpha_1}(v_1,z),
\]
where
\[
\lambda^{1}=\lambda,
\qquad
\lambda^{-1}=\widetilde{\lambda}.
\]
For $k>1$, we put
\[
\Lambda_{M,v}^{\alpha}(z)
=
\lambda^{\alpha_k}
\left(
M_k\bigl(
\Lambda_{\mathfrak r M,\mathfrak r v}^{\mathfrak r\alpha}(v_k)
\bigr),
z
\right).
\]
This is an $\mathcal F$-term, and its interpretation in a solution is exactly $\lambda_{m,v}^{\alpha}(z)$.

Now, for $n=1$, the defining condition of $\mathcal{RNS}_1$ is simply $\lambda_x=\id_X$ for every $x\in X$, and this is encoded by the identity $\lambda(v_1,z)\approx z$.

Assume now that $n\geq 2$. By definition, a solution belongs to
$\mathcal{RNS}_n$ if and only if $\lambda_{m,v}^{\alpha}
=
\lambda_{\mathfrak l m,\,m_2\cdot\mathfrak l v}^{\mathfrak l\alpha}$
for every $\alpha\in\{\pm1\}^{\times n}$, $v=(v_1,\dots,v_n)\in X^{\times n}$, and every $m=(m_2,\dots,m_n)\in W(X)^{\times(n-1)}$.
Equivalently, for every tuple of unary terms
\[
M=(M_2,\dots,M_n)\in\mathcal W^{\times(n-1)},
\]
we impose the identity
\[
\Lambda_{M,v}^{\alpha}(z)
\approx
\Lambda_{\mathfrak l M,\,M_2\cdot\mathfrak l v}^{\mathfrak l\alpha}(z),
\]
where
\[
M_2\cdot\mathfrak l v
=
\bigl(M_2(v_2),v_3,\dots,v_n\bigr).
\]

This is an identity of $\mathcal F$-terms in the variables
$v_1,\dots,v_n,z$ and in the auxiliary variables occurring in the terms
$M_2,\dots,M_n$.

Therefore the defining conditions of $\mathcal{RNS}_n$ are identities in the
signature $\mathcal F$.

Hence the class of solutions in $\mathcal{RNS}_n$ is axiomatized by the
identities $\Sigma$ defining solutions, together with the above family of
identities. Thus $\mathcal{RNS}_n$ is an equational class, and hence a
subvariety of the variety of solutions.
\end{proof}

A direct consequence is that the category $\mathcal{RNS}_n$ admits free objects (as already stated, we denote by $\FSol_{\mathcal{RNS}_n, X}$ the free object on $X$ in the category $\mathcal{RNS}_n$). In the following theorem, we show that the structure group of the free solution $\FSol_{\mathcal{RNS}_n, X}$ on a set $X$  is the free skew brace on $X$ in the category of skew braces that are right nilpotent of class at most $n$.

\begin{thm}\label{StrGrpFSBn}
    Let $X$ be a non-empty set and let $(\FSol_{\mathcal{RNS}_n, X}, r_{\FSol_{\mathcal{RNS}_n, X}})$ 
    be the free solution 
    on $X$ in the category $\mathcal{RNS}_n$ with inclusion map $i:X\to \FSol_{\mathcal{RNS}_n, X}$. Then its structure skew brace 
     $G(\FSol_{\mathcal{RNS}_n, X},r_{\FSol_{\mathcal{RNS}_n, X}})$ with ca\-nonical 
     map $g_{\FSol_{\mathcal{RNS}_n, X}}\circ i:X\to G(\FSol_{\mathcal{RNS}_n, X}, r_{\FSol_{\mathcal{RNS}_n, X}})$ 
    is isomorphic to the free right nilpotent skew brace of class at most $n$ on $X$.   
\end{thm}
\begin{proof}
Let $(F,r_F)=(\FSol_{\mathcal{RNS}_n, X}, r_{\FSol_{\mathcal{RNS}_n, X}})$ and let $B$ be a skew brace that is right nilpotent of class at most $n$ with a map $f : X \to B$. By Lemma~\ref{lemmmino-class-n}, the structure skew brace $G(F,r_F)$ is right nilpotent of class at most $n$.
The skew brace $B$ gives rise to a solution $(B,r_B)$, which belongs to $\mathcal{RNS}_n$ by~Lem\-ma~\ref{eqsolskewbrace}.
By the freeness of $F$, there exists a unique morphism of solutions
\[
\varphi : (F,r_F) \to (B,r_B)
\]
such that $\varphi\circ i= f$.
By the universal property of the structure skew brace,
there exists a unique skew brace homomorphism 
$G(\varphi): G(F,r_F) \to B$
such that 
\[
G(\varphi)\circ g_F=\varphi.
\]
So 
\[
G(\varphi)\circ (g_F\circ i)=\varphi\circ i=f.
\]
Moreover, if $\psi : G(F,r_F) \to B$ is another skew brace homomorphism such that $\psi\circ (g_F\circ i)= f$,
then $\psi\circ g_F = \varphi$, by the freeness of $(F,r_F)$.
Hence $\psi\circ g_F = G(\varphi)\circ g_F$ and by the universal property of the structure skew brace, $\psi=G(\varphi)$.

Therefore, $G(F,r_F)$ satisfies the universal property of the free skew brace on $X$ in $\mathcal{RN}_n$, so 
\(
G(F,r_F) \simeq \FSB_{\mathcal{RN}_n,X},
\)
which proves the statement.
\end{proof}

Under the natural correspondence between skew braces and solutions, the following remark shows that freeness is inherited only in the direction as described in \cref{freesolutionisfreeskew} and \cref{StrGrpFSBn}: the free skew brace is too large to serve as the free solution.

\begin{rem} 
\label{freeskewbracenotfreesolution}
Let $X$ be a non-empty set.
Then the associated solutions to
$\FSB_X$ and $\FSB_{\mathcal{RN}_n,X}$ are both not 
the free solutions in the respective category.
Indeed, we can consider
the trivial solution 
$(T,r_T)$ on $T=\{0,1\}$ and 
construct two different morphism
of solutions
$f_0,f_1:\FSB_X\to T$ (or $f_0,f_1:\mathcal{FSB}_{\mathcal{RN}_n,X}\to T$ )
that extend
the map $\alpha:X\to T$, $x\mapsto 0$.
Namely $f_0:b\mapsto 0$ and $f_1:b\mapsto
\begin{cases}
1 & \text{if } b=0,\\
0 & \text{if } b\neq 0
\end{cases}$.\\
One can state this in the following more general context.
Let $\mathcal C$ be a category of solutions containing the trivial solution on $T=\{ 0,1\}$. Let $B$ be a skew brace such that $(B,r_B)$ belongs to~$\mathcal C$, and let $\iota\colon X\to B$ be a map such that 
$0\notin \iota(X)$. Then $(B,r_B)$, together with~$\iota$, 
is not the free solution on $X$ in~$\mathcal C$. 
\end{rem}

Using the explicit description of $\FSB_{\mathcal{RN}_{n},X}$, as a structure skew brace, together with \cref{free injective in structure brace}, we can now derive a concrete description of the free solution in $\mathcal{RN}_{n,\mathrm{Inj}}$
in this case.
Indeed, 
by \cref{StrGrpFSBn}, $\FSB_{\mathcal{RN}_{n},X}$
is isomorphic to $G\left(\FSol_{\mathcal{RN}_{n},X},r_{\FSol_{\mathcal{RN}_{n},X}}\right)$.
Moreover, by \cref{free injective in structure brace}, we obtain the following.

\begin{cor}\label{freerightnilpotentsoldescriptinj}
Let $X$ be a non-empty set. Then,
    $(\widetilde{X}_{\mathcal{RN}_{n}}, \widetilde{r})$,
with $\widetilde{r}$ the restriction of $r_{G\left(FSol_{\mathcal{RN}_{n},X}\right)}$ on 
\[
\widetilde{X}_{\mathcal{RN}_{n}}=\{a+\lambda_b(x_0)-a\mid a,b\in G\left(\FSol_{\mathcal{RN}_{n},X},r_{\FSol_{\mathcal{RN}_{n},X}}\right)\} 
\]
is the free solution in $\mathcal{RN}_{n,\Inj}$ on $X$.
\end{cor}

\medskip

We end this section with a result on free braces that are right nilpotent of class at most $n$.

\begin{pro}\label{propertieskrnn-general-X}
Let $n\geq 1$, let $X$ be a non-empty set, let
$B=\FSB_{\mathcal{RN}_{n},X}$, and let $C$ be
the ideal of $B$ generated by $[B,B]_+$. Denote by $[a]$ the class of an
element $a\in B$ in $B/C$. Then\textnormal:
\begin{enumerate}
\item[\textnormal{(1)}] The skew brace $B/C$ is the free brace on $X$ in the category 
of braces  that are right nilpotent of  class at most
$n$, i.e. we have $B/C=\FB_{\mathcal{RN}_n,X}$.
\item[\textnormal{(2)}] The subset
\[
\widetilde X=\{\, [\lambda_b(x)]\mid x\in X,\ b\in B\,\}
\]
is a subsolution of the associated solution of $B/C$, and it is the free
injective solution on $X$ in the full subcategory of involutive solutions in
$\mathcal{RNS}_{n}$.
\item[\textnormal{(3)}] The structure skew brace $G(\widetilde X,r_{B/C}|_{\widetilde X})$ is the
free brace on $X$ in the category  of  braces that are
right nilpotent of class at most~$n$. Hence
\[
G(\widetilde X,r_{B/C}|_{\widetilde X})\cong B/C.
\]
\end{enumerate}
\end{pro}
\begin{proof}
We first prove (1). Since $C$ contains $[B,B]_+$, the additive group of the skew brace  $B/C$
is abelian. Moreover, since $B$ is right nilpotent of class at most
$n$, also
$B/C$ is right nilpotent of class at most $n$. Thus $B/C$
belongs to the class
of skew braces of abelian type (i.e. braces)  that are right nilpotent of  class at most
$n$.

Let $A$ be a brace that is right nilpotent  of  class at most
$n$, and let $f:X\to A$ be a map. Since
$A\in\mathcal{RN}_{n}$, by the freeness
of $B=\FSB_{\mathcal{RN}_{n},X}$ there exists a unique
skew brace homomorphism
$\widehat f:B\to A$ extending $f$. Since $(A,+)$ is abelian, we have
$\widehat f([B,B]_+)=0$. Hence the ideal generated by $[B,B]_+$ is contained
in $\ker(\widehat f)$, that is, $C\subseteq\ker(\widehat f)$. Therefore
$\widehat f$ factors uniquely through~$B/C$. This proves that $B/C$ is the
free brace on $X$ in the class of braces that are right nilpotent of class at
most $n$.

We now prove (2). Let $\mathcal{IRNS}_{n}$ be the full
subcategory of
$\mathcal{RNS}_{n}$ consisting of involutive solutions. 
The same argument as
in the proof of \cref{StrGrpFSBn}, but applied to the category
$\mathcal{IRNS}_{n}$, the structure
skew brace of the free solution in $\mathcal{IRNS}_{n}$ on
$X$ is the free skew
brace on $X$ in the class of skew braces of abelian type and of right
nilpotency class at most $n$. By (1), this skew brace is
$B/C$.

Now apply \cref{free injective in structure brace} to the category
$\mathcal{IRNS}_{n}$. The free injective solution is the
subsolution of $B/C$
given by
\[
\{\, [a]+ \lambda_{[b]}([x])-[a]\mid x\in X,\ a,b\in B\,\}.
\]
Since the additive group of $B/C$ is abelian, this set is equal to
\[
\{\, \lambda_{[b]}([x])\mid x\in X,\ b\in B\,\}
=
\{\, [\lambda_b(x)]\mid x\in X,\ b\in B\,\}
=
\widetilde X.
\]
Therefore $\widetilde X$ is a subsolution of the associated solution of $B/C$,
and it is the free injective solution on $X$ in
$\mathcal{IRNS}_{n}$.

Finally, (3) follows from (1) and (2). Indeed, by the universal property of the
structure skew brace, $G(\widetilde X,r_{B/C}|_{\widetilde X})$ is the free
brace on $X$ in the category of  right nilpotent braces that are of class  at most $n$. By (1), the same universal
property is satisfied
by $B/C$. Hence
\[
G(\widetilde X,r_{B/C}|_{\widetilde X})\cong B/C.
\]
The statement is proved.
\end{proof}

\section{Free right nilpotent skew brace of class \texorpdfstring{$n$}{n}}\label{sec: right nilp}

In this section we give an explicit construction  of the free object $\FSB_{\mathcal{RN}_n,X}$ on a non-empty set $X$ in the category $\mathcal{RN}_n$ of right nilpotent skew braces of right nilpotent class at most $n$. This explicit construction shows, in particular, that the additive and multiplicative groups of such a free object are free groups, and provides a great deal of information in the case of right nilpotency class at most $2$ (see \cref{rightnilclass2}).

Before outlining the structure of this section we need to introduce some notations. In particular, several  free group and monoid structures defined on the same underlying set will be considered.
To distinguish between these structures, we use the following notation: 
we denote by $\Fg_{\star}(T)$ (resp., $\Fm_\star(T)$) the free group (resp. the free monoid) on $T$ whose operation is written as $\star$. 
Let $Y$ be a set. Then we write $\pm Y$ for the set $Y \cup -Y$, where $-Y$ is a disjoint copy of $Y$, whose elements are elements of $Y$ with a $-$ prefix. To ease formulas later on, we make the convention that for $x,y \in Y$ we set $-(-x)=x$ and $-(x+y)=-y-x$.
Similarly, we denote $Y^{\pm}$ for the set $Y \cup \overline{Y}$, where $\overline{Y}$ is a disjoint copy of $Y$, whose elements are elements of $Y$ with an overline. We make the convention that if $x, y \in Y$, then $\overline{\overline{x}}=x$ and $\overline{(x\circ y)}=\overline{y}\circ\overline{x}$.

The outline of the section will consist of subsections, numbered 4.1-4.5.
\begin{enumerate}
\item[(4.1)] We recursively introduce the sets $A_n$, with $A_1=X$, and we define a homomorphism 
$\Fm_\circ(A_n^{\pm})\rightarrow \End(\Fm_+(\pm A_n))$.
\item[(4.2)] We define a map $\varphi^{(n)}:\Fm_+(\pm A_n) \rightarrow \Fm_{\circ}(A_n^{\pm})$. 
\item[(4.3)]  We show that we  obtain an induced mapping
$\varphi^{(n)}\colon \Fg_{+}(A_n) \rightarrow \Fg_{\circ}(A_n)$
and group homomorphism $\lambda\colon \Fg_{\circ}(A_n) \rightarrow \Aut(\Fg_+(A_n))$.
This requires a delicate proof. As a consequence we obtain that 
 the group $\Fg_{\circ}(A_n)$ acts via $\lambda$ on $\Fg_+(A_n)$.
 \item[(4.4)] Construction of a bijective $1$-cocycle $\pi^{(n)}\colon \Fg_{\circ}(A_n)\rightarrow \Fg_+(A_n)$. In order to do so we prove that $\varphi^{(n)}$ is the inverse map of $\pi^{(n)}$.
 \item[(4.5)] Construction of the free skew brace.  The  bijective $1$-cocycle $\pi^{(n)}$ yields  a skew  brace, denoted  $\FSB_{\mathcal{RN}_n,X}$, with additive group $\Fg_+(A_n)$ and multiplicative group $\Fg_{\circ}(A_n)$. This skew brace is generated by $X$ and is shown to be  right nilpotent of class $n$. We finish by proving it is the free object in the considered category.
\end{enumerate}

\subsection{The sets   \texorpdfstring{$A_n$ }{}  and  the map  \texorpdfstring{
$\lambda^{(n)}: \Fm_\circ(A_n^{\pm})\rightarrow \End(\Fm_+(\pm A_n))$}{}}

$\;\; $

\vspace{6pt}
Let $X$ be a non-empty set. 
We define the sets $X^{(i)}$ and $A_i$ for every positive integer $i$ recursively. 
Set $A_0=\emptyset$, $X^{(1)}=A_1=X$, and suppose we have already defined the sets $X^{(i)}\subseteq A_i$.
We wish to define $X^{(i+1)}$ and $A_{i+1}$.
First, define $F_{i+1}$ as the set of all elements $w$ of the free group $\Fg_{\circ}(A_i)$ whose reduced form ends with an element of $X^{(i)}\cup \overline{X^{(i)}}$ (here the elements with a bar take the role of the inverses),
so in particular $1\notin F_i$. Next, define 
$$X^{(i+1)} \text{ as the set of symbols } x^{(i+1)}_{f,t} \text{ with } f \in F_{i+1} \text{ and } t\in A_i$$ (this set is chosen to be disjoint from $A_i$), and then put~\hbox{$A_{i+1}=A_i \cup X^{(i+1)}$.} This defines $X^{(i)}$ and $A_i$ for every $i$ with the property that $X^{(i)}\subseteq A_i\subseteq A_{i+1}$.

Now, let $n\geq 2$ be an integer, and consider the free monoid $\Fm_+(\pm A_n)$ 
on the set $\pm A_n$ and the free monoid $\Fm_{\circ}(A_n^{\pm})$ on the set $A_{n}^{\pm}$.
For every~\hbox{$c\in A_n$,} we let $k(c)$ be the unique integer such that $c \in X^{(k(c))}$; also, we put $k(\overline{c})=k(-c)=k(c)$, where $\overline c\in \Fm_\circ(A_n^{\pm})$ and $-c\in \Fm_+(\pm A_n)$. 
We define for every~\hbox{$a \in A_{n-1}^{\pm}$} an endomorphism~$\lambda_a^{(n)}$ of the free monoid $\Fm_+(\pm A_n)$ as the extension of the following rule: 
 \begin{align*}
     \lambda_a^{(n)}(b) &= \begin{cases}
     x^{(k(a)+1)}_{a,b} + b, &\text{ if } k(a) \geq k(b),\\[0.2cm]
     x^{(k(b))}_{a\circ f,t}, &\text{ if } k(a) < k(b) \text{ and } b=x^{(k(b))}_{f,t} \text{ with }f \neq \overline{a},\\[0.2cm]
    - x^{(k(b))}_{a,t}, &\text{ if } k(a) < k(b) \text{ and } b=x^{(k(b))}_{f,t} \text{ with }f = \overline{a},
     \end{cases}\\
    \lambda^{(n)}_a(-b)&=-\lambda_a(b),
 \end{align*}
 for every $b\in A_n$,  where $\overline{a}$ denotes the inverse of $a$ in $\Fg_{\circ}(A_{k(a)})$. Note, for example, in the third case, that $b=x^{(k(b))}_{\overline a,t}$ implies $k(b)=k(a)+1$.

 For $a\in \big(X^{(n)}\big)^{\pm}$, we put $\lambda_a^{(n)}= \textup{id}_{\Fm_+(\pm A_n)}$ (this will ultimately allow us to obtain a skew brace of right nilpotency class $n$).

\smallskip

 (For the sake of notation, we will often omit the superscript $(n)$ when it is clear from the context.)

 \smallskip

Since, by definition,  $\Fm_\circ(A_n^{\pm})$ is a free monoid on $A_n\cup \overline{A_n}$, the assignment  we have  defined so far  extends to a monoid  homomorphism:
\[
\lambda^{(n)}\colon a\in \Fm_\circ(A_n^{\pm})\longmapsto \lambda_{a}^{(n)}\in\End(\Fm_+(\pm A_n)).
\] 

\begin{lem}\label{level of lamda images}
Let $1<k\leq n$. If $a\in A_n\cup \overline{A_n}$ and $w\in \Fm_+(\pm (A_n\setminus A_{k-1}))$, then 
\[
\lambda_a^{(n)}(w)\in \Fm_+(\pm (A_n\setminus A_{k-1})).
\]
\end{lem}
\begin{proof}
First, let $b\in A_n\setminus A_{k-1}$.
If $k(a)<k(b)$, then, by definition,
\[
\lambda_a(b),\,\lambda_{\overline{a}}(b),\, \lambda_a(-b),\,\lambda_{\overline{a}}(-b)\in \pm(A_n\setminus A_{k-1})
\]
and
$k(\lambda_a(b))=k(b)$. If $k(a)=n$, then $\lambda_a$ is the identity. If $n>k(a)\geq k(b)$, then \[
\begin{array}{c}
  \lambda_a(b)=x^{(k(a)+1)}_{a,b}+b,\quad  
  \lambda_a(-b)=-b-x^{(k(a)+1)}_{a,b},\quad
  \lambda_{\overline{a}}(b)=x^{(k(a)+1)}_{\overline{a},b}+b,\quad\textnormal{and}\\[0.25cm]
  \lambda_{\overline{a}}(-b)=-b-x^{(k(a)+1)}_{\overline{a},b}
\end{array}\] belong to $\Fm_+(\pm (A_n\setminus A_{k-1}))$. This proves the statement for an element of $A_n\setminus A_{k-1}$. The complete result now follows using the fact that $\lambda_a$
is a monoid homomorphism and that 
$\Fm_+(\pm (A_n\setminus A_{k-1}))$ is a submonoid of $\Fm_+( \pm A_n)$.   
\end{proof}

\begin{lem}\label{lambdapusheslevelupevenmore}
Let $1\leq k,l<n$ be integers, $a \in \Fm_{\circ}((A_n \setminus A_{l-1})^{\pm})$ and
\hbox{$c\in \pm X^{(k)}$.}
If $m=\max\{l,k\}$, then there exist $w,t\in \Fm_+(\pm (A_n\setminus A_m))$ and $c' \in \pm X^{(k)}$  such that $\lambda_a^{(n)}(c) = w+c'+t$. Moreover, if $l\geq k$, then one can take $c'=c$. 
\end{lem}
\begin{proof}
    First, let $a\in (A_n\setminus A_{l-1})^{\pm}$.
    Suppose that $c\in X^{(k)}$.
    If $k(a)=n$, then we have nothing to prove.
    If $k(a)<n$, then $\lambda_a(c)\in \pm X^{(k)}$ for $k(a)< k$,
    or $\lambda_a(c)=x^{k(a)+1}_{a,c}+c$ for $k(a)\geq k$.
    In the former case, we can easily conclude with $\lambda_a(c)=c'$.
    In the latter case, $ x^{(k(a)+1)}_{a,c} \in X^{(k(a)+1)}\subseteq \Fm_+(\pm (A_n\setminus A_m))$, since $k(a)\geq m$ as one has that $k(a)\geq l$ and $k(a)\geq k$.
    So we can conclude with $\lambda_a(c)=w+c'$ 
    for $w=x^{k(a)+1}_{a,c}$ and $c'=c$.
    Now, if $c=-d$ for some $d\in X^{(k)}$, we can use the
    previous calculation for $d$ and obtain that there exist $w_1,t_1\in \Fg_+(A_n\setminus A_m)$ and $c_1 \in X^{(k)}\cup -X^{(k)}$  such that $\lambda_a^{(n)}(d) = w_1+c_1+t_1$.
    
    So $\lambda_a^{(n)}(c)=\lambda_a^{(n)}(-d) = -t_1-c_1-w_1$ and we can conclude 
    with $w=-t_1$, $c'=c_1$ and $t=-w_1$.

    Therefore, for every $a\in (A_n\setminus A_{l-1})^{\pm}$ and every $c\in \pm X^{(k)}$ there exist $w,t\in \Fm_+( \pm (A_n\setminus A_m))$ and $c' \in \pm X^{(k)}$  such that $\lambda_a^{(n)}(c) = w+c'+t$.

    Finally, we can conclude the same 
    for any $a\in \Fm_\circ((A_n \setminus A_{l-1})^{\pm})$ observing that if we have the thesis for two words $a_1,a_2\in \Fm_\circ((A_n \setminus A_{l-1})^{\pm})$, then we can also conclude for $a_1\circ a_2$.
    More precisely, let $w_2,t_2\in \Fm_+(\pm (A_n\setminus A_m))$ and $c' \in \pm X^{(k)}$  be such that $\lambda_{a_2}(c) = w_2+c_2+t_2$.
    Then we can apply the thesis for $\lambda_{a_1}(c_2)$ and obtain $w_1,t_1\in \Fm_+(\pm (A_n\setminus A_m))$ and $c_1 \in \pm X^{(k)}$  such that $\lambda_{a_1}(c_2) = w_1+c_1+t_1$. 
    With this, we have
    \[
    \lambda_{a_1\circ a_2}(c)
    =\lambda_{a_1}(w_2)+w_1+c_1+t_1+\lambda_{a_1}(t_2).
    \]
    Hence we can also conclude, using \cref{level of lamda images}, that $\lambda_{a_1\circ a_2}(c)=w+c'+t$ for
    $w=\lambda_{a_1}(w_2)+w_1\in \Fm_+(\pm (A_n\setminus A_m))$,
    $t=t_1+\lambda_{a_1}(t_2)\in \Fm_+(\pm (A_n\setminus A_m))$
    and $c'=c_1\in \pm X^{(k)}$.
\end{proof}

\begin{lem}\label{lambdapusheslevelup}
Let $1\leq k\leq l<n$ and $0\leq i$ be integers,
and \hbox{$a \in \Fm_{\circ}((A_n \setminus A_{l-1})^{\pm})$.}
If $t\in \Fm_+(\pm(A_n \setminus A_{k-1}))$ has at most $i$ occurrences 
of elements in $\pm X^{(k)}$,
and \hbox{$c\in \pm X^{(k)}$,} then there exist $w\in \Fm_+(\pm(A_n\setminus A_l))$ and $t_1 \in \Fm_+(\pm(A_n\setminus A_{k-1}))$ with at most $i$ occurrences of elements in $\pm X^{(k)}$ such that $\lambda_a^{(n)}(c+t) = w+c+t_1$ 
\end{lem}
\begin{proof}
Write $c_1=c$ and $t=w_2 + c_2 + \ldots +w_{i+1}+c_{i+1}+w_{i+2}$,
for some $c_2,\ldots,c_{i+1}\in\pm X^{(k)}$ and $w_2,\ldots,w_{i+2}\in \Fm_+(\pm (A_n\setminus A_k))$. Then, for any $a \in \Fm_{\circ}(\pm (A_n \setminus A_{l-1}))$, one has by \cref{level of lamda images} that $\lambda_a(w_v)\in \Fm_+(\pm (A_n\setminus A_k))$, for any $2\leq v \leq i+2$. Moreover, by \cref{lambdapusheslevelupevenmore}, for any $1 \leq v \leq i+1$ that there exist $w_v',t_v' \in \Fm_+(\pm (A_n\setminus A_k))$ such that $\lambda_a(c_v) = w_v'+c_v+t_v'$. 
Hence,
\begin{align*}\lambda_a(c+t) &= \lambda_a(c) + \lambda_a(w_2)+ \lambda_a(c_2) + \ldots +\lambda_a(w_{i+1}) + \lambda_a(c_{i+1})+\lambda_a(w_{i+2})\\ 
&= w_1'+c+t_1' + \lambda_a(w_2) + w_2'+c_2+t_2'+\ldots \\ &\qquad\qquad\qquad\qquad\ldots+\lambda_{a}(w_{i+1}) +w_{i+1}'+c_{i+1}+t_{i+1}'+\lambda_a(w_{i+2}),\end{align*}
from which the result follows.
\end{proof}

\vspace{10pt}
\subsection{The map   \texorpdfstring{$\varphi^{(n)}:\Fm_+(\pm A_n) \rightarrow \Fm_{\circ}(A_n^{\pm})$}{}}

$\;\; $

\vspace{6pt}
Now, we define a map $\varphi^{(n)}:\Fm_+(\pm A_n) \rightarrow \Fm_{\circ}(A_n^{\pm})$ in the following way. Let $w=c_1+c_2+\ldots+c_s\in \Fm_+(\pm A_n)$, where $c_1,\dots,c_s\in \pm A_n$.
\begin{enumerate}
    \item[(1)] If $w=0$, then we define $\varphi(w)=1$.
    \item[(2)] If $s=1$ and $c_1=w\in A_n$, then we define
    $\varphi(c_1)=c_1$; while, if $s=1$ and $c_1=-a$, for some $a\in A_n$, then we define
    \[
\varphi(-a)=\overline{a_1}\circ \overline{a_2}\circ\ldots\circ \overline{a_{n-k(a)}}\circ \overline{a_{n-k(a)+1}},
    \]
    where $a_1=a$ and $a_{k+1}=x_{a_k,a_k}^{(k(a_k)+1)}$ for every $1\leq k\leq n-k(a)$; this range of values can be better understood by noting that $k(a_{n-k(a)+1})=n$.
    \item[(3)] If $s>1$, then we define $\varphi(w)$ using the following procedure: 
    \begin{itemize}
        \item[(3a)] Write \[
    \varphi(w)=\varphi(c_1)\circ \varphi(w'),
    \] where $w'=\lambda_{\overline{\varphi(c_1)}}(c_2+\ldots+c_s)$;
        \item[(3b)] Apply (1), (2) and (3) to compute $\varphi(w')$.
    \end{itemize}
\end{enumerate}

Of course, we need to show that the procedure is well-defined, or in other words that the recursion procedure described in (3) stops after finitely many steps and thus provides a well-defined value for $\varphi(w)$.

\begin{lem}
\label{pre phi of sum}
    Let $w\in \Fm_+(\pm A_n)$, say
    $w=c_1+c_2+\ldots+c_s$
    where $c_1,\dots,c_s\in \pm A_n$ 
    and let 
    $m$ be a positive integer.
    Denote
    $t=c_1+\ldots +c_m$ and $p=c_{m+1}+\ldots +c_s$.
    Suppose that 
    the procedure stops for $t$ and for $\lambda_{\overline{\varphi(t)}}(p)$. 
    Then it also stops for $w$ and 
    \[
\varphi(w)=\varphi(t+p)
=\varphi(t)\circ\varphi\left(\lambda_{\overline{\varphi(t)}}(p)\right).
    \]
\end{lem}
\begin{proof}
    We proceed by induction on the number of steps needed
    to define 
    $\varphi(t)$.
    If the procedure requires only one step, it means
    that $t=0$ or $m=1$.
    In both cases, the claim is trivial. 
    Suppose now that $m>1$ and that the procedure requires $k$ steps 
    for $t$. In particular, 
    it requires $k-1$ steps to stop for 
    $t'=\lambda_{\overline{\varphi(c_1)}}(c_2+\ldots +c_m)$
    and $\varphi(t)=\varphi(c_1)\circ\varphi(t')$. 
    Let $p'=\lambda_{\overline{\varphi(c_1)}}(p)$, and note that 
    \[
   \lambda_{\overline{\varphi(t')}}(p')
    =\lambda_{\overline{\varphi(t')}}\left(\lambda_{\overline{\varphi(c_1)}}(p)\right)
    =\lambda_{\overline{\varphi(c_1)\circ \varphi(t')}}(p)
    =\lambda_{\overline{\varphi(t)}}(p),
    \]
    for which we know, by assumption, that the procedure stops.
    So, by inductive hypothesis, the procedure also stops for 
    $t'+p'$ 
    and 
    \begin{align*}
        \varphi(t'+p')
    &=\varphi(t')\circ\varphi\left(\lambda_{\overline{\varphi(t')}}(p')\right)\\
    &=\varphi(t')\circ\varphi\left(\lambda_{\overline{\varphi(t')}}\left(\lambda_{\overline{\varphi(c_1)}}(p)\right)\right)\\
    &=\varphi(t')\circ\varphi\left(\lambda_{\overline{\varphi(t)}}(p)\right).
    \end{align*}
    Hence the procedure stops for $t+p$ and
    \[
    \varphi(t+p)
    =\varphi(c_1)\circ\varphi\left(t'+p'\right)
    =\varphi(c_1)\circ\varphi(t')\circ\varphi\left(\lambda_{\overline{\varphi(t)}}(p)\right)
    =\varphi(t)\circ\varphi\left(\lambda_{\overline{\varphi(t)}}(p)\right),
    \]
    as required.
\end{proof}

We are now in a position to show that the procedure stops for any element
$w=c_1+c_2+\ldots+c_s\in \Fm_+(\pm A_n)$, where $c_1,\dots,c_s\in \pm A_n$.
Suppose first that $w\in \Fm_+(\pm (A_n\setminus A_{n-1}))$,
so $c_1,\dots, c_s\in \pm X^{(n)}$.
If $w=0$ or $s=1$, then we are done by definition. Otherwise, \[
\varphi(w)=\varphi(c_1)\circ\varphi(c_2)\circ\ldots\circ\varphi(c_s)
\] because by the definitions of $\lambda$ and $\varphi$, we have 
$\lambda_{\varphi(c_i)}=\id$ for every $i\in\{1,\ldots,n\}$. Note that $\varphi(w)\in S_\circ((A_n\setminus A_{n-1})^\pm)$.

Suppose now that there exists $1\leq k\leq n-1$ such that the procedure stops for every element of $\Fm_+(A_n\setminus A_k)$, and that if $u\in \Fm_+(\pm(A_n\setminus A_i))$ with $k\leq i\leq n-1$, then $\varphi(u)\in \Fm_\circ((A_n\setminus A_i)^\pm))$. Let $w\in \Fm_+(\pm (A_n\setminus A_{k-1}))$. We show that the procedure stops for $w$ by induction 
on the number $i$ of occurrences of elements in $\pm X^{(k)}$ in the expression $w=c_1+\ldots+c_s$. If $i=0$, then $w\in \Fm_+(\pm(A_n\setminus A_k))$ and the procedure stops by induction hypothesis. If $i>0$, then there exist $w_1 \in \Fm_+(\pm (A_n\setminus A_k))$, $c\in\pm X^{(k)}$, and $t_1 \in \Fm_+(\pm (A_n\setminus A_{k-1}))$ with at most $i-1$ occurrences of elements in $\pm X^{(k)}$, such that $w=w_1+c+t_1$.

Using \cref{lambdapusheslevelup}, 
starting from $w_1$ and $t_1$, we recursively define  elements 
$w_{j+1}\in \Fm_+(\pm (A_n\setminus A_{k+j}))$, and 
$t_{j+1}\in \Fm_+((A_n\setminus A_{k-1}))$ with at most $i-1$ occurrences of elements in $\pm X^{(k)}$, such that
\[
\lambda_{\overline{\varphi(w_j)}}(c+t_j)=w_{j+1}+c+t_{j+1}
\] (the additional induction hypothesis is used here to correctly apply \cref{lambdapusheslevelup}).
Because of the inductive hypothesis, we
know that the procedure stops for the elements $w_j$ and $t_j$, where $j\in\{1,\dots, n-k+1\}$.
Note that $w_{n-k+1}$ belongs to $\Fm_+(\pm (A_n\setminus A_{n-1}))$, so $\lambda_{\overline{\varphi(w_{n-k+1})}}=\id$.
Hence, by applying the procedure
to $w_{n-k+1}+c+t_{n-k+1}$, we have
$$
\varphi(w_{n-k+1}+c+t_{n-k+1})=
\varphi(w_{n-k+1})\circ\varphi(c)\circ\varphi\left(\lambda_{\overline{\varphi(c)}}(t_{n-k+1})\right).
$$ But, by induction, the procedure stops for $t_{n-k+1}$, so also for $\lambda_{\overline{\varphi(c)}}(t_{n-k+1})$ by \cref{lambdapusheslevelup}, and hence even for $w_{n-k+1}+c+t_{n-k+1}$.

Since
\[
\lambda_{\overline{\varphi(w_{n-k})}}(c+t_{n-k})=w_{n-k+1}+c+t_{n-k+1},
\]
so the procedure stops for $\lambda_{\overline{\varphi(w_{n-k})}}(c+t_{n-k})$.
Hence, by \cref{pre phi of sum}, it also stops for
$w_{n-k}+c+t_{n-k}$.
Continuing in this way, we have that the procedure stops for $w=w_1+c+t_1$ and 
\begin{align*}
\varphi(w)
&=\varphi(w_1)\circ\varphi\left(w_2+c+t_2\right)
=\varphi(w_1)\circ\varphi(w_2)\circ\varphi\left(w_3+c+t_3\right)
=\dots\\
&\ldots=\varphi(w_1)\circ\varphi(w_2)\circ\ldots\circ\varphi(w_{n-k+1})\circ\varphi(c)\circ\varphi\left(\lambda_{\overline{\varphi(c)}}(t_{n-k+1})\right).
\end{align*}
Note also that $\varphi(w)\in \Fm_\circ((A_n\setminus A_{k-1})^{\pm})$.
This finishes the proof of the definition of the map $\varphi^{(n)}:\Fm_{+}(\pm A_n) \rightarrow \Fm_{\circ}(A_n^\pm)$.

\medskip

\vspace{10pt}
\subsection{Induced maps \texorpdfstring{$\lambda^{(n)}\colon \Fg_{\circ}(A_n) \rightarrow \Aut(\Fg_+(A_n))$}{} and \texorpdfstring{$\varphi^{(n)}\colon \Fg_+(A_n) \rightarrow \Fg_{\circ}(A_n)$}{}}

$\;$

The following lemmas will show that the functions $\lambda^{(n)}$ and $\varphi^{(n)}$ naturally induce mappings (which will be denoted with the same symbols) that are defined in terms of free groups instead of free monoids. In the following, let 
$$s\colon \Fm_{\circ}(A_n^{\pm})\rightarrow \Fg_{\circ}(A_n) \text{ and }s'\colon S_{+}(\pm A_n)\rightarrow \Fg_+(A_n)$$ stand for the canonical epimorphisms. 

\begin{lem}\label{lambda of a reduced word}
Let $t=y_1+\ldots +y_l$, where $y_l \in \pm A_n$, and $w \in \Fm_{\circ}((A_n\setminus A_k)^{\pm})$ such that $s(w) = 1$. Then there exist $w_1,\ldots,w_{l+1} \in \Fm_+(\pm( A_n \setminus A_{k}))$  with $s'(w_i)=0$ for all $i$ such that $$ \lambda_{w}(t) = w_1 + y_1+w_2+\ldots +w_l+y_l+w_{l+1}.$$
\end{lem}
\begin{proof}
    Let $a\in A_n^{\pm}$. Then it is easily checked that for $t \in A_n$ one has $$\lambda_{\overline{a}\circ a}(\pm t) = \begin{cases}
    \pm(-x_{\overline{a},t}^{(k(a)+1)}+x_{\overline{a},t}^{(k(a)+1)}+t),  & \text{ if }k(a)\geq k(t),\\
        \pm t & \text{ if } k(a)<k(t).
    \end{cases}$$

Thus, if $t'=y_1+\ldots+y_l \in \Fm_+(\pm A_n)$ with $y_i \in \pm A_n$, then there exist
$$w_1 \dots, w_{l+1}\in \Fm_+(\pm A_n\setminus A_{k(a)})$$
such that $s'(w_1)=\ldots=s'(w_{l+1})=0$ and 
$$ \lambda_{\overline{a} \circ a}(t) = \lambda_{\overline{a} \circ a}(y_1)+\ldots + \lambda_{\overline{a} \circ a}(y_l) = w_1+y_1+w_2+\ldots +w_l+y_l+w_{l+1}.$$
This proves the result when the length of $w$ is either $0$ or $2$ (which are the smallest possible values for the length of $w$).
We continue by induction on the length of $w$. 

    Let $t \in A_n$ and $w \in \Fm_{\circ}((A_n\setminus A_k)^{\pm})$ be such that $s(w)=1$. 
    Then, either there exist elements $w_1,t_1 \in \Fm_{\circ}((A_n\setminus A_k)^{\pm})$
    with $s(w_1)=s(t_1)=1$ such that $w=w_1\circ t_1$ 
    or there exists an element $w_1 \in  \Fm_{\circ}((A_n\setminus A_k)^{\pm})$ with $s(w_1)=1$ and $a \in (A_n\setminus A_k)^{\pm}$
    such that $w=\overline{a} \circ w_1 \circ a$.
    First, we deal with the former case.
    By the induction hypothesis, there exist $y,y',z,z' \in \Fm_+(\pm A_n\setminus A_k)$ with $s'(y)=s'(y')=s'(z)=s'(z')=0$ 
    such that $\lambda_{t_1}(t)=y+t+y'$ and $\lambda_{w_1}(t)=z+t+z'$.
    Then, 
    $$ \lambda_{w_1 \circ t_1}(t)
    =\lambda_{w_1}(y)+\lambda_{w_1}(t)+\lambda_{w_1}(y') 
    = \lambda_{w_1}(y)+z+t+z'+\lambda_{w_1}(y').$$
    Note that since $\lambda_{w_1}$ is an endomorphism of $\Fm_+(A_n)$ such that 
    $\lambda_{w_1}(-c)=-\lambda_{w_1}(c)$, it follows that $s'(\lambda_{w_1}(y))=0$. 
    Thus, the former case is shown.

    Suppose now that $w=\overline{a}\circ w_1 \circ a.$ Again, there exist \hbox{$y,y',z,z' \in S_+(\pm (A_n\setminus A_k))$} with $s'(y)=s'(y')=s'(z)=s'(z')=0$ such that $\lambda_{w_1}(t)=y+t+y' $ and $\lambda_{w_1}(x_{a,t}^{(k(a)+1)})=z+x_{a,t}^{(k(a)+1)}+z'$. If $k(a)\geq k(t)$, then \begin{align*} \lambda_{\overline{a}\circ w_1 \circ a}(t) &= \lambda_{\overline{a}\circ w_1}(x_{a,t}^{(k(a)+1)}+t) \\ &=\lambda_{\overline{a}}(z+x_{a,t}^{(k(a)+1)}+z'+y+t+y') \\ &=\lambda_{\overline{a}}(z)-x_{\overline{a},t}^{(k(a)+1)}+\lambda_{\overline{a}}(z')+\lambda_{\overline{a}}(y)+x_{\overline{a},t}^{(k(a)+1)}+t+\lambda_{\overline{a}}(y').\end{align*} Since $$ s'(\lambda_{\overline{a}}(z)-x_{\overline{a},t}^{(k(a)+1)}+\lambda_{\overline{a}}(z')+\lambda_{\overline{a}}(y)+x_{\overline{a},t}^{(k(a)+1)}) =0,$$ we are done.  If $k(a)<k(t)$, then $\lambda_{\overline{a}\circ w_1 \circ a}(t)=\lambda_{\bar a}(y)+t+\lambda_{\bar a}(y')$ and similarly  we are done. This completes the proof of the latter case.

    The general result now follows from the fact that $\lambda$ is an additive morphism.
\end{proof}

Next, we need to show that for $w,t \in \Fm_+(\pm A_n)$ with $s'(w)=s'(t)$ one has  $s(\varphi(w)) = s(\varphi(t))$. 
For any $w\in \Fm_+(\pm A_n)$, there exists $\widetilde w\in \Fm_+(\pm A_n)$ of smallest length (which we call the {\it reduced form} of $w$) such that $s'(w)=s'(\widetilde w)$. Clearly, if for $t,w \in \Fm_+(\pm A_n)$ one has  $s'(t)=s'(w)$, then $\widetilde t=\widetilde w$. Hence, it is sufficient to prove that for $w \in \Fm_+(\pm A_n)$ and $\widetilde w$ its reduced form one has  $s(\varphi(w)) = s(\varphi(\widetilde w))$. Since $\widetilde w$ is obtained from $w$ by subsequently removing pieces of the form $a-a$ from its standard expression, by induction, it is sufficient to prove that $s(\varphi(w)) = s(\varphi(t))$ for words $w, t \in \Fm_+(\pm A_n)$ where $t$ is obtained from $w$ by removing a piece of the form $a-a$.

Concretely, we show for $w,t\in \Fm_+(\pm A_n)$ and $a \in \pm A_n$ that 
$$ s(\varphi(w+a-a+t)) = s(\varphi(w+t)),$$ 
by induction on $k(a)$ downwards.
We start with the basis step $a \in \pm X^{(n)}$. 
Then, denoting $a'=\lambda_{\overline{\varphi(w)}}(a)\in \pm X^{(n)}$ one has  \begin{align*}
    \varphi(w +a -a+t) &= \varphi(w) \circ \varphi(a'-a'+\lambda_{\overline{\varphi(w)}}(t)) \\ &=\varphi(w) \circ \varphi(a') \circ \varphi(-a'+\lambda_{\overline{\varphi(w)}}(t)) \\ &=\varphi(w) \circ \varphi(a') \circ \overline{\varphi(a')} \circ \varphi(\lambda_{\overline{\varphi(w)}}(t)),
\end{align*} 
where we used that 
$\lambda_{\varphi(a')}=\lambda_{\varphi(-a)}=\textup{id}.$
Hence, we find that 
 \begin{align*}
    s(\varphi(w+a-a+t)) &= s(\varphi(w) \circ \varphi(a')\circ\overline{\varphi(a')} \circ \varphi(\lambda_{\overline{\varphi(w)}}(t))) \\
    &= s(\varphi(w)\circ \varphi(\lambda_{\overline{\varphi(w)}}(t))) \\ 
    &=s(\varphi(w+t)),
\end{align*}
which shows the base case.

We now formulate the induction hypothesis. 
Suppose that 
\begin{equation}
\label{IH1}
\tag{IH1}
\begin{aligned}
    s(\varphi(u+c-c+v)) &= s(\varphi(u+v))\\
    \text{ for any }
    u,v \in \Fm_+(\pm A_n)
    &\text{ and }
    c \in \pm (A_n\setminus A_k)
\end{aligned}
\end{equation}

Note that this implies that 
\begin{equation}
\label{IH1'}    
\tag{IH1'}
\begin{aligned}
    s(\varphi(u+q+v)) &= s(\varphi(u+v))\\
    \text{ for any }
    u,v \in \Fm_+(\pm A_n)
    \text{ and }
    &q \in \Fm_+(\pm A_n\setminus A_k)
    \text{ such that } s'(q)=0.
\end{aligned}
\end{equation}

Let $w,t\in \Fm_+(\pm A_n)$
and $a \in X^{(k)}$. 
We have to prove that
\[
s(\varphi(w+a-a+t)) = s(\varphi(w+t)).
\]
We will first start with the case $w=t=0$.

First, since
$$\varphi(-a)=\overline{a_1} \circ \overline{a_2}\circ \ldots \circ\overline{a_{n-k(a)+1}},$$ 
where $a_1=a$ and $a_{i+1}=x_{a_i,a_i}^{(k+i)}$,  we see that 
 \begin{equation}
 \label{phimina}
 \varphi(-a)=\overline{a} \circ \overline{a_2}\circ \ldots\circ\overline{a_{n-k(a)+1}}=\overline{a}\circ \varphi(-a_2). 
\end{equation}
Hence, 
\begin{align*}
    \varphi(a-a) &= \varphi(a) \circ \varphi\left(-\lambda_{\overline{a}}(a)\right) \\ 
    &=a \circ \varphi\left(-a-x_{\overline{a},a}^{(k+1)}\right) \\
    &=a \circ \varphi(-a) \circ \varphi\left(-\lambda_{\overline{\varphi(-a)}}\left(x_{\overline{a},a}^{(k+1)}\right)\right)\\
    &=a\circ \overline{a}\circ \varphi(-a_2)\circ\varphi\left(-\lambda_{\overline{\varphi(-a_2)}\circ a}\left(x_{\overline{a},a}^{(k+1)}\right)\right)\\ 
    &=a\circ \overline{a}\circ \varphi(-a_2)\circ\varphi\left(\lambda_{\overline{\varphi(-a_2)}}(x_{a,a}^{(k+1)})\right) \\
    &=a \circ \overline{a} \circ \varphi(-a_2+a_2),
\end{align*}
where we used that $x_{a,a}^{(k+1)}=a_2$.
By the induction hypothesis (\ref{IH1}) with $u=v=0$ and $c=-a_2\in -X^{(k+1)}\subseteq \pm(A_n\setminus A_k)$,
it follows that $s(\varphi(-a_2+a_2))=1$. 
Hence,
\begin{align*}
s(\varphi(a-a))=s(a\circ \overline{a}\circ \varphi(-a_2+a_2))
=s(\varphi(-a_2+a_2))=1.
\end{align*}

Now we proceed with the case $w=0$ and  $t \in \Fm_+(\pm A_n)$. Let $t=y_1+\ldots +y_l$ with $y_i \in \pm A_n$.
By \cref{lambda of a reduced word}, 
there exist $w_1,\ldots, w_{l+1} \in \Fm_+(\pm ( A_n \setminus A_k))$ with $s'(w_i)=0$ such that 
$$ \lambda_{\overline{\varphi(a-a)}}(t)=w_1+y_1+\ldots +w_l+y_l+w_{l+1}.$$ Then, 
\begin{align*} 
\varphi(a-a+t) &= \varphi(a-a) \circ \varphi\left(\lambda_{\overline{\varphi(a-a)}}(t)\right)\\
&= \varphi(a-a) \circ \varphi(w_1+y_1+\ldots +w_l+y_l+w_{l+1}),
\end{align*}

Note that the induction hypothesis (\ref{IH1'}) applied $l+1$ times (where the $i+1$-th time is using $u=y_i$, $v=y_{i+1}$ and $q=w_{i+1}$ with $0\leq i\leq l+1$ and the convention $y_0=y_{l+1}=0$)
implies that 
$$ s(\varphi(w_1+y_1+\ldots +w_l+y_l+w_{l+1})) = s(\varphi(y_1+\ldots+y_l)).$$

Thus, 
\begin{align*}
    s(\varphi(a-a+t)) &= s(\varphi(a-a) \circ \varphi(w_1+y_1+\ldots +w_l+y_l+w_{l+1})) \\ 
    &= s\big(s(\varphi(a-a)) \circ s(\varphi(w_1+y_1+\ldots +w_l+y_l+w_{l+1}))\big)\\ 
    &= s\big(1 \circ s(\varphi(y_1+\ldots+y_l))\big)\\ 
    &= s(\varphi(t)).
\end{align*}

Before we show the general case, we first need to  consider the dual versions of the previous, i.e., when $a\in -X^{(k)}$, say $a=-b$ for some $b\in X^{(k)}$.
As before, we first deal with the case $w=t=0$.

Recall from \cref{phimina} that $\varphi(-b)=\overline{b}\circ\varphi(-b_2)$.
Then, we find that 
\begin{align*}
    \varphi(a-a)&
    =\varphi(-b+b) = \varphi(-b) \circ \varphi( \lambda_{\overline{\varphi(-b)}}(b)) \\ 
    &= \overline{b} \circ \varphi(-b_2) \circ \varphi(\lambda_{\overline{\varphi(-b_2)}\circ b}(b)) \\
    &=\overline{b} \circ \varphi(-b_2) \circ \varphi(\lambda_{\overline{\varphi(-b_2)}}(b_2+b))\\
    &=\overline{b}\circ \varphi(-b_2+b_2+b).
\end{align*}
By the induction hypothesis (\ref{IH1}) with $u=0$, $v=b$ and $c=-b_2\in -X^{(k+1)}\subseteq \pm(A_n\setminus A_k)$,
it holds that $s(\varphi(-b_2+b_2+b))=s(\varphi(b))=b$.
Hence, 
$$ s(\varphi(a-a))=s\left(\overline{b}\circ \varphi(-b_2+b_2+b)\right)=s\big(\overline{b}\circ s(\varphi(-b_2+b_2+b))\big) = s(\overline{b}\circ b)=1.$$

Next, we deal with the case $a=-b\in -X^{(k)}$, $w=0$, and 
$t=y_1+\ldots +y_l\in \Fm_+(\pm A_n)$ for some elements $y_i \in \pm A_n$. 
By \cref{lambda of a reduced word},
there exist $w_1,\ldots, w_{l+1} \in \Fm_+(\pm ( A_n \setminus A_k))$ 
with $s'(w_i)=0$ such that 
$$ \lambda_{\overline{\varphi(a-a)}}(t)=w_1+y_1+\ldots +w_l+y_l+w_{l+1}.$$ 
Recall that the induction hypothesis (\ref{IH1'}) applied $l+1$ times (where the $i+1$-th time is using $u=y_i$, $v=y_{i+1}$ and $q=w_{i+1}$ with $0\leq i\leq l+1$ and the convention $y_0=y_{l+1}=0$) yields that $$ s(\varphi(w_1+y_1+\ldots +w_l+y_l+w_{l+1}))=s(\varphi(y_1+\ldots+y_l)) = s(\varphi(t)).$$
Then we have that  \begin{align*}
    s(\varphi(a-a+t)) 
    &=s(\varphi(-b+b+t) 
    =s\big(s(\varphi(-b+b)) \circ s(\varphi(\lambda_{\overline{\varphi(-b+b)}}(t)))\big)\\ &=s\big(1 \circ s(\varphi(w_1+y_1+\ldots +w_l+y_l+w_{l+1}))\big)\\ 
    &=s(\varphi(y_1+\ldots +y_l))\\
    &=s(\varphi(t)).
\end{align*}
We have now proved that
$$ s(\varphi( a-a+t))=s(\varphi(t)).$$

We are left to deal with the general case, i.e., to show that
$$ s(\varphi( w +a-a+t))=s(\varphi(w+t))$$ for $a\in \pm X^{(k)}$
We proceed by downward induction on the minimal level of an element that appears in $w$.

First, let $w \in \Fm_+(\pm X^{(n)})$. Then $\lambda_{\overline{\varphi(w)}}=\textup{id}$, and hence, 
$$ s(\varphi(w+a-a+t)) 
= s\big(s(\varphi(w))\circ s(\varphi(a-a+t))\big)
= s\big(s(\varphi(w))\circ s(\varphi(t))\big)
=s(\varphi(w+t)),$$ which proves the induction basis.

We formulate the (second) induction hypothesis. Suppose that 
\begin{equation}
\label{IH2}
\tag{IH2}
\begin{aligned}
    s(\varphi(u+c-c+v)) &= s(\varphi(u+v))\\
    \text{ for any }
    u \in \Fm_+(A_n\setminus A_m)&\text{, }v \in \Fm_+(\pm A_n)
    \text{ and }
    c \in \pm X^{(k)}.
\end{aligned}
\end{equation}

Let $w \in \Fm_+(A_n\setminus A_{m-1})$. Denote $M=\max(k,m)$.  Then, by \cref{lambdapusheslevelupevenmore}, there exist 
$w_2,q_2\in \Fm_+(\pm (A_n \setminus A_M))$ and $a' \in \pm X^{(k)}$ such that 
$$ \lambda_{\overline{\varphi(w)}}(a)=w_2+a'+q_2.$$ 
Then,  
$$ \varphi(w+a-a+t)
=\varphi(w)\circ\varphi( w_2+a'+q_2-q_2-a'-w_2+\lambda_{\overline{\varphi(w)}}(t)).$$ As $M\geq k$, then 
$q_2-q_2\in \Fm_+(\pm (A_n \setminus A_M))\subseteq \Fm_+(\pm (A_n \setminus A_k))$ and $s'(q_2-q_2)=0$.
So we can use the inductive hypothesis (\ref{IH1'})
with $u=w_2+a',$ $v=-a'-w_2+\lambda_{\overline{\varphi(w)}}(t)$ 
and $q=q_2-q_2$ to obtain
$$ s(\varphi( w_2+a'+q_2-q_2-a'-w_2+\lambda_{\overline{\varphi(w)}}(t))=s(\varphi( w_2+a'-a'-w_2+\lambda_{\overline{\varphi(w)}}(t))).$$ As $M\geq m$, then 
$w_2\in S_+(\pm (A_n \setminus A_M))\subseteq \Fm_+(\pm (A_n \setminus A_m))$.
So we can use the inductive hypothesis (\ref{IH2})
with $u=w_2,$ $v=-w_2+\lambda_{\overline{\varphi(w)}}(t)$ 
and $c=a'$ to obtain
$$ s(\varphi( w_2+a'-a'-w_2+\lambda_{\overline{\varphi(w)}}(t)))= s(\varphi( w_2-w_2+\lambda_{\overline{\varphi(w)}}(t))).$$ Since 
$M\geq k$, then 
$w_2-w_2\in \Fm_+(\pm (A_n \setminus A_M))\!\subseteq\! \Fm_+(\pm (A_n \setminus A_k))$ and \hbox{$s'(w_2-w_2)\!=\!0$.}
So we can use the inductive hypothesis (\ref{IH1'})
with $u=0,$ $v=\lambda_{\overline{\varphi(w)}}(t)$ 
and $q=w_2-w_2$ to obtain
$$ s(\varphi( w_2-w_2+\lambda_{\overline{\varphi(w)}}(t))) = s(\varphi( \lambda_{\overline{\varphi(w)}}(t))).$$ Putting all of this together, we find that 
\begin{align*}
    s(\varphi(w+a-a+t)) &=s(\varphi(w)\circ\varphi( w_2+a'+q_2-q_2-a'-w_2+\lambda_{\overline{\varphi(w)}}(t))) \\
    &=s\big(s(\varphi(w)) \circ s(\varphi(\lambda_{\overline{\varphi(w)}}(t)))\big)\\
    &=s(\varphi(w+t)),
\end{align*}
which shows the claim. Thus, we have shown the following.

\begin{lem}
    The map $\varphi^{(n)}\colon \Fm_+(\pm A_n) \rightarrow \Fm_{\circ}(A_n^{\pm})$ 
    induces a natural map     
    $$\varphi^{(n)}\colon \Fg_+(A_n) \rightarrow \F_{\circ}(A_n)$$
    such that the following diagram is commutative
$$    \begin{tikzcd}
        \Fm_+(\pm A_n)\arrow[r,"\varphi^{(n)}"] \arrow[d,"s'"] &  \Fm_{\circ}(A_n^{\pm})\arrow[d,"s"] \\
        \Fg_+(A_n)\arrow[r,"\varphi^{(n)}"]& \Fg_{\circ}(A_n)
    \end{tikzcd}
$$
\end{lem}

\medskip

We now need to show an analogue result for $\lambda$. We start by noting that if $w=t_1+c-c+t_2\in \Fm_+(\pm A_n)$ for some $t_1,t_2\in \Fm_+(\pm A_n)$ and $c\in\pm A_n$, then
\[
\lambda_a(w)
=\lambda_a(t_1+c-c+t_2)
=\lambda_a(t_1)+\lambda_a(c)-\lambda_a(c)+\lambda_a(t_2),
\]
for every $a\in A_n^{\pm}$.
Hence $s'(\lambda_a(w))=s'(\lambda_a(t))$
for every $a\in A_n^{\pm}$,
which implies that for any $u,v  \in \Fm_+(\pm A_n)$
with $s'(u)=s'(v)$ one has that $s'(\lambda_a(u))=s'(\lambda_a(v))$
for every $a\in A_n^{\pm}$.
Therefore for any $a\in A_n^{\pm}$ the morphism $\lambda_a$ induces an endomorphism of $\Fg_+(A_n)$ such that the following diagram
\begin{center}
    \begin{tikzcd}
        \Fm_+(\pm A_n)\arrow[r,"\lambda_a"] \arrow[d,"s'"] &  \Fm_+(\pm A_n)\arrow[d,"s'"] \\
        \Fg_+(A_n)\arrow[r,"\lambda_a"]& \Fg_+(A_n)
    \end{tikzcd}
\end{center}
is commutative.

Moreover, for every $a\in A_n^{\pm}$, $\lambda_a$ is invertible as an endomorphism of $\Fg_+(A_n)$ with inverse $\lambda_{\overline{a}}$.
This is because
$\Fg_+(A_n)$ is the free group on $A_n$ and,
by definition, for every $t \in A_n$ 
\[
\lambda_{\overline{a}}\left(\lambda_{a}(t)\right) =
\begin{cases}
        s'\left(-x_{\overline{a},t}^{(k(a)+1)}+x_{\overline{a},t}^{(k(a)+1)}+t\right) =t & \text{ if }k(a)\geq k(t)\text{ and }k(t)\neq n\\
        s'(t)=t & \text{ if } k(a)<k(t)\text{ or }k(t)=n
    \end{cases}
\]
and 
\[
\lambda_{a}\left(\lambda_{\overline{a}}(t)\right) =
\begin{cases}
        s'\left(-x_{a,t}^{(k(a)+1)}+x_{a,t}^{(k(a)+1)}+t\right)=t  & \text{ if }k(a)\geq k(t)\text{ and }k(t)\neq n\\
        s'(t)=t & \text{ if } k(a)<k(t)\text{ or }k(t)=n
    \end{cases}
\]
Therefore, $\lambda_w\in \Aut(\Fg_+(A_n))$ for every $w\in \Fm_\circ(A_n^{\pm})$.

Finally, by \cref{lambda of a reduced word} it follows that if $w\in \Fm_\circ(A_n^{\pm})$ with $s(w)=1$, then $\lambda_w\in \Aut(\Fg_+(A_n))$ is the identity.
Hence, for every 
$u,v\in \Fm_\circ(A_n^{\pm})$
such that $u$ is obtained by applying exactly one reduction on $v$,
i.e., $v=u_1\circ w\circ u_2$ where $u=u_1\circ u_2$ and $s(w)=1$,
one has
\[
\lambda_v
=\lambda_{u_1\circ w\circ u_2}
=\lambda_{u_1}\circ \lambda_w\circ \lambda_{u_2}
=\lambda_{u_1}\circ \lambda_{u_2}
=\lambda_u.
\]
Therefore
$\lambda\colon \Fg_{\circ}(A_n) \rightarrow \Aut(\Fg_+(A_n))$ with $\lambda(w)=\lambda_w$ is a group morphism.  In summary, by \cref{pre phi of sum} one obtains the following result.

\begin{pro}
\label{phi of sum}
    The group $\Fg_{\circ}(A_n)$ acts via $\lambda$ on $\Fg_+(A_n)$.
    Moreover, 
    $$ \varphi^{(n)}(t+p) = \varphi^{(n)}(t) \circ \varphi^{(n)}(\lambda_{\overline{\varphi(t)}}(p)),$$
    for every $t,p \in \Fg_+(A_n)$.
\end{pro}

\begin{cor}\label{third step of phi(a-a)}
For every $a\in A_n$,
\[
\varphi\left(\lambda_{\varphi(-a)}^{-1}\left(-x_{\overline{a},a}^{(k(a)+1)}\right)\right)
=\overline{\varphi(-a)}\circ \overline{a}.
\]  
\end{cor}
\begin{proof}
Since $0=\varphi(0)$, by \cref{phi of sum} we have that
    \begin{align*}
       0
    &=\varphi(a-a)\\
    &=\varphi(a)\circ\varphi\left(\lambda_{\overline{\varphi(a)}}(-a)\right)\\
    &=\varphi(a)\circ\varphi\left(-\lambda_{\overline{a}}(a)\right)\\
    &=\varphi(a)\circ\varphi\left(-a-x_{\overline{a},a}^{(k(a)+1)}\right)\\
    &=\varphi(a)\circ\varphi(-a)\circ\varphi\left(\lambda_{\overline{\varphi(-a)}}\left(-x_{\overline{a},a}^{(k(a)+1)}\right)\right).
    \end{align*}
    Hence, 
    \[
    \varphi\left(\lambda_{\overline{\varphi(-a)}}\left(-x_{\overline{a},a}^{(k(a)+1)}\right)\right)
    =\overline{\varphi(-a)}\circ \overline{a}
    \] and the statement is proved.
\end{proof}

\vspace{10pt}
\subsection{Construction of a bijective 1-cocycle  \texorpdfstring{$\pi^{(n)}\colon \Fg_{\circ}(A_n)\rightarrow \Fg_+(A_n)$}{}.}

$\;$

\vspace{6pt}
Now, we will construct a  skew brace  with multiplicative group $\Fg_{\circ}(A_n)$ and additive group $\Fg_{+}(A_n)$  by explicitly defining a bijective $1$-cocycle 
  $$\pi^{(n)}\colon\Fg_{\circ}(A_n)\rightarrow \Fg_+(A_n),$$ 
with respect to the action defined by the homomorphism $\lambda^{(n)}$. (Again, whenever it is clear from the context, the superscript $(n)$ will be omitted to simplify the notation.)

The image of an element $w\in \Fg_{\circ}(A_n)$ will be defined recursively via the length $\ell(w)$ of $w$. Put $\pi (1)=0$. For $a\in A_n$, we define
 $$\pi(a)=a \quad \mbox{ and } \quad  \pi(\overline{a})=-\lambda_a^{-1}(a), \quad \mbox{for } a\in A_n.$$  If $\pi(w)$ is defined for some $w \in \Fg_{\circ}(A_n)$,
 then for $a\in A_n \cup\overline A_n$ with $\ell(a\circ w)=1+\ell(w)$ (note that this means that $a\circ w$ is in reduced form in case $w$ is in reduced form), we define 
 \begin{eqnarray}\label{picocycle}     
\pi(a\circ w) = \pi(a) + \lambda_a(\pi(w)).
\end{eqnarray}

We are now in the position to prove that $\pi^{(n)}$ is invertible with inverse map $\varphi^{(n)}$.
This will be done in a series of lemmas.

\begin{lem}\label{injective}
$\varphi\circ\pi=\textup{id}_{\Fg_\circ(A_n)}$, so $\pi$ is injective.
\end{lem}
\begin{proof}
Clearly, if $a\in A_n$, then $\varphi(\pi(a))=a$, while
if $a\in X^{(n)}$, then
\[
\varphi(\pi(\overline{a}))
=\varphi\left(-\lambda_{\overline{a}}(a)\right)
=\varphi\left(-a\right)
=\overline{a_1}
=\overline{a}.
\]
Suppose now that $a\in A_n$ with $k(a)<n$. 
Then, by \cref{third step of phi(a-a)} and \cref{phi of sum},
\begin{align*}
    \varphi(\pi(\overline{a}))&=
    \varphi\left(-\lambda_{\overline{a}}(a)\right)=
    \varphi\left(-\left(x_{\overline{a},a}^{(k(a)+1)}+a\right)\right)=
    \varphi\left(-a-x_{\overline{a},a}^{(k(a)+1)}\right)\\
    &=\varphi(-a)\circ\varphi\left(\lambda_{\overline{\varphi(-a)}}\left(-x_{\overline{a},a}^{(k(a)+1)}\right)\right)\\
    &=\varphi(-a)\circ\overline{\varphi(-a)}\circ \overline{a}
    =\overline{a}.
\end{align*}
So for $c\in A_n\cup \overline{A_n}$ we proved that $\varphi(\pi(c))=c$. Moreover, for $w\in  \Fg_\circ(A_n)$
\begin{align*}
    \varphi(\pi(c\circ w))&=
    \varphi(\pi(c)+\lambda_c(\pi(w))
    =\varphi(\pi(c))+\varphi\left(\lambda_{\overline{\varphi(\pi(c))}}\lambda_c(\pi(w))\right)\\
    &=c+\varphi(\lambda^{-1}_{c}\lambda_c(\pi(w)))
    =c+\varphi(\pi(w)).
\end{align*}

 Therefore, we can conclude, by induction on the length of $w\in \Fg_\circ(A_n)$ that
 \[
 \varphi(\pi(w))=w,
 \]
 for every $w\in \Fg_\circ(A_n)$.
 \end{proof}
Hence, we have shown that $\pi$ is injective. In order to prove that $\pi$ is surjective, we first prove that $\pi$ is a $1$-cocycle. This will allow us to use induction in the proof of surjectivity.

\begin{lem}\label{cocycle}
    The map $\pi$ is a $1$-cocycle for the action of $\lambda$.
\end{lem}
\begin{proof}
We need to show that
$\pi(a\circ w) = \pi(a) + \lambda_a(\pi(w)),$
for any $a,w\in \Fg_\circ(A_n)$. 
We prove this by induction on the length $\ell(a)$ of $a$. 

If $\ell(a)=0$, then this is obvious as $\lambda_1=\id$. Assume  $\ell(a)=1$.  If $\ell(a\circ w)=1+\ell(w)$, then the claim is valid by the definition of $\pi$.  Otherwise, 
the first letter in the reduced form of $w$ is $\overline{a}$, say $\widetilde w=\overline{a}\circ w'$ is the reduced form of $w$.
By definition 
\[
\pi(a)+\lambda_a(\pi(\overline{a}\circ w'))
=\pi(a)+\lambda_a(\pi(\overline{a})+\lambda_{\overline{a}}(\pi(w')))
=\pi(a)+\lambda_a(\pi(\overline{a}))+\pi(w')
\]

We claim that  $0=\pi (a) +\lambda_a (\pi (\overline{a}))$, so the previous expression becomes
$\pi(w')=\pi(a\circ \overline{a}\circ w')=\pi(a\circ w)$, which will complete the proof for $\ell(a)=1$.

In case $a\in A_n$, we have 
\[
\lambda_a (\pi (\overline{a}))
= \lambda_a(- \lambda_{a}^{-1}(a))
= 
-\lambda_a (\lambda_{a}^{-1} (a))
=-a
=-\pi (a)
\]
and thus the claim follows.
In case  $a\in \overline{A_n}$, put $a=\overline{b}$ with $b\in A_n$.
We need to show that  $0=\pi \left(\overline{b}\right) +\lambda_{\overline{b}}(\pi (b))$, i.e.
\[
-\pi \left(\overline{b}\right) 
= \lambda_{\overline{b}}(\pi (b))
=\lambda_{b}^{-1}(b),
\]
and this holds by definition.
 
 Suppose now that  $\ell(a)\geq 2$ and write $a=b\circ w'$, with $\ell(b)=1$ and $\ell(w')<\ell(a) $.
Then, by the previous step, 
$ \pi (a\circ w) =\pi (b \circ w' \circ w)
=\pi (b) +\lambda_b (\pi ( w'\circ w))$.
Hence, by
the inductive hypothesis,  $ \pi (a\circ w)= \pi (b) +\lambda_b \left(\pi (w') +\lambda_{w'}(\pi (w ))\right)$.
The additivity of the map $\lambda_b$ yields
\[
\pi (a\circ w)   = \pi (b) +\lambda_b (\pi (w')) +\lambda_b \lambda_{w'} \pi (w)=  \pi (b) +\lambda_b (\pi (w')) +\lambda_{b\circ w'}(\pi (w)).
\]
Since $\ell(b)=1$, applying again the previous case,
we thus obtain
\[
 \pi (a\circ w)=\pi (b\circ w') +\lambda_{b\circ w'}(\pi (w))
    = \pi (a) +\lambda_{a} \pi (w),
\]
as desired.
\end{proof}

\begin{lem}\label{lemma10.3}
For any $a \in \pm A_n $, we have $ \pi(\varphi(a))=a$.
\end{lem}
 \begin{proof}
If $a\in A_n$, then by definitions $\pi(\varphi(a))=a$, and moreover, 
     \begin{align*}
         \pi (\varphi(-a))
         &=\pi\left(\overline{a_1}\circ \overline{a_2}\circ\ldots\circ \overline{a_{n-k(a)+1}}\right)
         =\pi\left(\overline{a_1}\right)+\lambda_{\overline{a_1}}\left(\pi\left(\overline{a_2}\circ\ldots\circ \overline{a_{n-k(a)+1}}\right)\right)
     \end{align*} by \cref{cocycle}.
     To prove that $\pi (\varphi(-a))=-a$,
     we will proceed by induction on~$k(a)$.
     If $k(a)=n$, then 
     $\lambda_{a}$ is the identity and $\varphi(-a)=\overline{a}$,
     so 
     \[
     \pi (\varphi(-a))
     =\pi(\overline{a})=-\lambda_a^{-1}(a)
     =-a.
     \]
     Let now $k(a)<n$ and suppose that 
     $\pi(\varphi(-b))=-b$ for every $b\in X^{(k(a)+1)}$.
     Note that $a_2=x^{(k(a)+1)}_{a,a}\in X^{(k(a)+1)}$ and that 
     \[
     \varphi(-a_2)=\overline{a_2}\circ\ldots\circ \overline{a_{n-k(a)+1}}.
     \]
     Therefore, the inductive hypothesis on $a_2$ implies that
     \[
     \pi(\overline{a_2}\circ\ldots\circ \overline{a_{n-k(a)+1}})=-a_2.
     \]
     Hence,
     \begin{align*}
        \pi (\varphi(-a))
        &=\pi\left(\overline{a}\right)+\lambda_{\overline{a}}\left(\pi\left(\overline{a_2}\circ\ldots\circ \overline{a_{n-k(a)+1}}\right)\right)\\
     &=-\lambda_a^{-1}(a)+\lambda_{\overline{a}}\left(-a_2\right)\\
     &=-\lambda_a^{-1}\left(x^{(k(a)+1)}_{a,a}+a\right)\\
      &=-\lambda_a^{-1}\left(\lambda_a(a)\right)\\
      &=-a. 
     \end{align*}
     The statement is proved.
 \end{proof}

\begin{lem}\label{lemsurjective}
    The map $\pi$ is surjective.
\end{lem}
\begin{proof}
    We prove the surjectivity by showing, by induction on $k$ from $n$ to 1,
    that for every 
    $w\in \Fg_+(A_n\setminus A_{k-1})$,
    there exists $z\in \Fg_\circ(A_n\setminus A_{k-1})$ 
    such that $\pi(z)=w$.

    The base step for the induction is to start with an element of $\Fg_{+}(A_n\setminus A_{n-1})$. Let $w \in \Fg_{+}(A_n)$ such that $w= x_1 + \ldots +x_l$ with $x_i \in X^{(n)}\cup -X^{(n)}$. Denote $$ t_i = \begin{cases}
        x_i & \text{ if } x_i \in X^{(n)},\\
        \overline{y_i} & \text{ if } x_i=-y_i \in -X^{(n)}. 
    \end{cases}$$
    Then, by \cref{cocycle} and because $\lambda_{t_i}=\id$, we get $\pi(t_1 \circ \ldots \circ t_l) = x_1 + \ldots +x_l=w.$ This proves the base step.
    
    Suppose that $\pi: \Fg_\circ(A_n\setminus A_k)\to \Fg_+(A_n\setminus A_k)$
    is surjective.
    To prove that $\pi: \Fg_\circ(A_n\setminus A_{k-1})\to \Fg_+(A_n\setminus A_{k-1})$ is surjective, we proceed by a second induction on the number of occurrences of elements of $X^{(k)}$ in its reduced form.

    Let $w\in \Fg_+(A_n\setminus A_{k-1})$ 
    with exactly $j$ occurrences of elements of $X^{(k)}$ in its reduced form $\widetilde w$.
    Write $\widetilde w=w_1 + a + t_1$, where $w_1$ is a reduced word with only letters in $A_n\setminus A_k$,
    $a \in \pm X^{(k)}$ and $t_1 \in \Fg_+(A_n\setminus A_{k-1})$ 
    with exactly $j-1$ occurrences of elements of $X^{(k)}$.
    Then, by the first induction hypothesis,
    since $w_1\in \Fg_+(A_n\setminus A_{k})$, there exists a $z_1\in \Fg_{\circ}(A_n\setminus A_{k})$ such that $\pi(z_1) = w_1$. 
    
    Note that $z_1=s_1\circ\ldots\circ s_m$ for some $s_1,\dots,s_m\in (A_n\setminus A_k)\cup \overline{(A_n\setminus A_k)}$, so $k(s_i) > k=k(a)$.
    Hence 
    \[
    \lambda_{s_i}^{-1}(a)
    =\begin{cases} 
    x_{\overline{s_i},a}^{(k(s_i)+1)}+a &\text{ if } s_i\in A_n\setminus A_{k-1}\\[0.2cm]
    x_{s_i,a}^{(k(s_i)+1)}+a &\text{ if } s_i \in \overline{(A_n\setminus A_{k-1})}
    \end{cases},
    \]
    for every $i\in\{1,\dots, m\}$.
    More generally, we have 
    \begin{eqnarray} \label{lambdaprop} 
    \lambda^{-1}_{s_{1}\circ\ldots\circ s_i}(a)&=&v_i+a, 
    \end{eqnarray}
    where $v_i\in \Fg_+(A_n\setminus A_{k+1})$ for every $i\in\{1,\dots,m\}$.   
    This  statement can be proven by induction on $i$. 
    The case $i=1$ is clear from the above.
    Now suppose $1\leq i<m$   and $\lambda^{-1}_{s_{1}\circ\ldots\circ s_i}(a)=v_i+a$
    for some $v_i\in \Fg_+(A_n\setminus A_{k+1})$.
    Recall from \cref{level of lamda images} that also 
    $\lambda^{-1}_{s_{i+1}}(v_i)\in \Fg_+(A_n\setminus A_{k+1})$. Hence
    \[
    \lambda^{-1}_{s_{1}\circ\ldots\circ s_i\circ s_{i+1}}(a)
    =\lambda^{-1}_{s_{i+1}}\lambda^{-1}_{s_{1}\circ\ldots\circ s_i}(a)
    =\lambda^{-1}_{s_{i+1}}(v_i+a)
    =v_{i+1}+a,
    \]
    with
    \[
       v_{i+1} 
    =\begin{cases} 
    \lambda^{-1}_{s_{i+1}}(v_i)+x_{\overline{s_{i+1}},a}^{(k(s_{i+1})+1)} &\text{ if } s_{i+1}\in A_n\setminus A_{k-1}\\[0.2cm]
    \lambda^{-1}_{s_{i+1}}(v_i)+x_{s_{i+1},a}^{(k(s_{i+1})+1)} &\text{ if } s_{i+1} \in \overline{(A_n\setminus A_{k-1})}
    \end{cases},
    \]
    which is again in $\Fg_+(A_n\setminus A_{k+1})$, since $k(s_{i+1})\geq k$.
    This proves \eqref{lambdaprop}.

    Therefore, 
    $\lambda_{z_1}^{-1}(a)=\lambda_{s_1\circ\ldots\circ s_m}^{-1}(a)=v_m+a$, with $v_m\in \Fg_+(A_n\setminus A_{k+1})$.

    \smallskip
    
    Put $w_2=v_m$, so that
    $\lambda_{z_1}^{-1}(a) = w_2+a$, and put 
    $t_2=\lambda_{z_1}^{-1}(t_1)$, where $w_2$ is in its reduced form (which only contains elements of $A_n\setminus A_{k+1}$).
    By \cref{lambdapusheslevelup} there are at most $j-1$ occurrences of elements of $X^{(k)}$ in the reduced form of $t_2$.
    Continue this procedure inductively, defining $z_i\in \Fg_{\circ}(A_n\setminus A_{k+i-1})$
    such that $\pi(z_i) =w_i$ and $\lambda_{z_i}^{-1}(a) = w_{i+1}+a$
    for some $w_{i+1}\in \Fg_{+}(A_n\setminus A_{k+i})$ 
    and $t_{i+1} = \lambda_{z_i}^{-1}(t_i)$. 
    At some point (actually for  $l=n-k$), we have that $w_l\in \Fg_+(A_n\setminus A_{n-1})$.  
   So, we have that $\pi(z_{l-1})=w_{l-1}$ and $\lambda^{-1}_{z_{l-1}}(a)=w_l+a$
    with  $w_l\in \Fg_{+}(A_n\setminus A_{n-1})$. 
   Then we define $t_l=\lambda_{\varphi(a)}^{-1}(t_{l-1})$. 
Note again that by \cref{lambdapusheslevelup}  each $t_i$ belongs to $\Fg_+(A_n\setminus A_{k-1})$ and its reduced form has $j-1$ occurrences of elements of $X^{(k)}$.
    
    Finally, by the first base step, let  $z_l \in \Fg_{\circ}(A_n\setminus A_{n-1})$ be
    such that $\pi(z_l) =w_l$ and, by the induction hypothesis, there exists 
    $t' \in \Fg_{\circ}(A_n\setminus A_{k-1})$ 
    such that $\pi(t')=\lambda_{\varphi(a)}^{-1}(t_{l})$.
Then, we claim that $$\pi(z_1\circ\ldots \circ z_{l} \circ \varphi(a)\circ t') = w.$$
    Indeed, first utilizing that $\pi$ is a $1$-cocycle, one finds that 
    \begin{align*}
        \pi&(z_1\circ\ldots \circ z_{l-1} \circ z_l\circ \varphi(a)\circ t')\\
        &= \pi(z_1)+\lambda_{z_1}(\pi(z_2)+\ldots+\lambda_{z_{l-1}}(\pi(z_l)+\lambda_{z_l}(\pi \varphi(a) + \lambda_{\varphi(a)}(\pi(t'))))\dots)
    \end{align*}
    Recall that $z_l \in \Fg_{\circ}(A_n\setminus A_{n-1})$, so $\lambda_{z_l}=\textup{id}_{\Fg_+(A_n)}$. Using this, and splitting the term containing $t'$, we find that the previous expression is equal to
    \begin{align*}
        \pi&(z_1)+\lambda_{z_1}(\pi(z_2)+\ldots+\lambda_{z_{l-1}}(\pi(z_l)+\pi \varphi(a))\dots) + \lambda_{z_1}\ldots \lambda_{z_{l-1}}\lambda_{\varphi(a)}(\pi(t')) \\ 
        &=\pi(z_1)+\lambda_{z_1}(\pi(z_2)+\ldots+\lambda_{z_{l-1}}(w_l+a)\dots) + \lambda_{z_1}\ldots \lambda_{z_{l-1}}\lambda_{\varphi(a)}(\lambda_{\varphi(a)}^{-1}(t_l)),
    \end{align*}
    where in the last equality we used that $\pi \varphi(a)=a$ 
    (because of \cref{lemma10.3}),
    $\pi(z_l)=w_l$, and the definition of $t'$.
    Recall that $w_l+a=\lambda_{z_{l-1}}^{-1}(a)$ and $t_l= \lambda_{z_{l-1}}^{-1}(t_{l-1})$. Then, the previous expression becomes $$\pi(z_1)+\lambda_{z_1}(\pi(z_2)+\lambda_{z_2}(\pi(z_3)\ldots+\lambda_{z_{l-2}}(w_{l-1}+a)\ldots)) + \lambda_{z_1}\ldots\lambda_{z_{l-2}}(t_{l-1})). $$ Repeating this inductively, we find that this is equal to $$ w_1+\lambda_{z_1}(\lambda_{z_1}^{-1}(a)) + \lambda_{z_1}(\lambda_{z_1}^{-1}(t_1)) =w_1+a+t_1=w,$$
    which shows the induction step. Hence, the result follows.
\end{proof}

\vspace{10pt}
\subsection{Construction of the free skew brace.}

$\;$

\vspace{6pt}
Because of \cref{injective}, \cref{lemma10.3} and \cref{lemsurjective} we have  shown that $\pi^{(n)}$ is a bijective $1$-cocycle (whose inverse is $\varphi^{(n)}$). Hence, 
we obtain  
$$\textit{a skew brace, denoted by }  \FSB_{\mathcal{RN}_n,X},$$ 
with additive group $\Fg_+(A_n)$ and multiplicative group $\Fg_{\circ}(A_n)$. 
In \cref{FreeNilpn} we will show that this is indeed the free skew brace in the category of skew braces that are right nilpotent of class at most $2$ on the set $X$. So~$\FSB_{\mathcal{RN}_n,X}$ can be seen as a skew brace with 
    $(\Fg_+(A_n),+,\circ)$, where  $\circ$ (with an abuse of notation, we denote the multiplication on  $\Fg_+(A_n)$  with the same symbols as the operation of $F_\circ(A_n)$)
    can be transferred using the $1$-cocycle in the following way:
    \[
    u\circ v= \pi^{(k)}(\varphi^{(k)}(u)\circ \varphi^{(k)}(v)).
    \]

\begin{rem}
{\rm It follows from the construction that $\FSB_{\mathcal{RN}_n,X}$ is generated by the set~$X$.}
\end{rem}

In the following lemma, we examine the role of the subsets $X^{(i)}$ in the skew  brace~$F_{R_n,X}$.
\begin{lem}
\label{star prod in FRn}
For any positive integer $i<n$, let $a \in X^{(i)}$ and $b \in A_i$. Then,
    $$ x^{(i+1)}_{a,b} = a\ast b \text{ and }x^{(i+1)}_{\overline{a},b}=\overline{a} \ast b.$$
    Furthermore, for $f \in F_{i+1}$ which in reduced form is $f'\circ a$ \textnormal(resp. $f'\circ \overline{a}$\textnormal) one has that
    $$ x^{(i+1)}_{f,b} = \lambda_{f'}(a\ast b) \quad\left(\text{resp. } x^{(i+1)}_{f,b} = \lambda_{f'}(\overline{a}\ast b)\right).$$
\end{lem}
\begin{proof}
Since $\FSB_{\mathcal{RN}_n,X}$ is a skew  brace, the first equality follows from the definition of the map $\lambda_a$. The second equality follows from $\lambda^{-1}_{a}=\lambda_{\overline{a}}$.
We now prove the remaining equalities.

Let $f'=l_1\circ l_2 \circ\ldots \circ l_s$ be the reduced form of $f'$  in $\Fg_\circ(A_i)$.
In particular, $l_s\neq \overline{a}$ and 
\(
k\big(x^{(i+1)}_{a,b}\big)=i+1>i\geq k(l_s),
\)
so
\[
\lambda_{l_s}(a*b)
=\lambda_{l_s}\left(x^{(i+1)}_{a,b}\right)=x^{(i+1)}_{l_s\circ a,b}.
\]
Therefore, since $l_{j-1}\neq \overline{l_j}$ and 
\(
k\left(x^{(i+1)}_{l_j\circ\ldots \circ l_s\circ a,b}\right)=i+1>i\geq l_{j-1},
\)
\[
\lambda_{l_{j-1}}\left(x^{(i+1)}_{l_j\circ\ldots \circ l_s\circ a,b}\right)
=x^{(i+1)}_{l_{j-1}\circ\ldots \circ l_s\circ a,b}.
\]
So, by induction on $j$ from $s-1$ to 1, we obtain that
\[
\lambda_{l_j\circ l_{j+1}\circ \ldots \circ l_{s-2}\circ l_{s-1}}\left(x^{(i+1)}_{l_s\circ a,b}\right)
=x^{(i+1)}_{l_j\circ \ldots \circ l_s\circ a,b}.
\]
Hence
\[
\lambda_{f'}(a*b)
=\lambda_{l_1\circ l_2\circ \ldots \circ l_{s-1}\circ l_s}\left(x^{(i+1)}_{a,b}\right)
=x^{(i+1)}_{l_1\circ\ldots l_s\circ a,b}
=x^{(i+1)}_{f,b}.
\] The statement is proved.
\end{proof}

\begin{pro}
\label{right nilp class n}
    The skew brace $\FSB_{\mathcal{RN}_n,X}$ is right nilpotent of class $n$.
\end{pro}
\begin{proof}
    We proceed by induction on $n$. For $n=1$, this is clear because $\FSB_{\mathcal{RN}_1,X}$ is defined as the trivial skew brace on the group $\Fg_+(X)$.

    Suppose that we have shown that $\FSB_{\mathcal{RN}_{n-1},X}$ is right nilpotent of class $n-1$. Then, denote $\varepsilon\colon \Fg_+(A_n) \rightarrow \Fg_+(A_{n-1})$ for the group morphism that is the extension of the rule $\varepsilon(a)=a$ for $a \in A_{n-1}$ and $\varepsilon(a)=0$ for $a \in X^{(n)}$. We claim that $\varepsilon$ is a skew brace morphism between $\FSB_{\mathcal{RN}_n,X}$ and $\FSB_{\mathcal{RN}_{n-1},X}$.
    
    Note that  for every $u,v\in \Fg_+(A_k)$,
    \begin{align*}
        \varepsilon(u\circ v)
    &=\varepsilon\left(\pi^{(n)}\left(\varphi^{(n)}(u)\circ \varphi^{(n)}(v)\right)\right)\\
    &=\pi^{(n-1)}\left(\delta\left(\varphi^{(n)}(u)\circ \varphi^{(n)}(v)\right)\right)
    \end{align*}
    and
    \begin{align*}
    \varepsilon(u)\circ\varepsilon(v)
    &=\pi^{(n-1)}\left(\varphi^{(n-1)}\left(\varepsilon(u)\right)\circ \varphi^{(n-1)}\left(\varepsilon(v)\right)\right)\\
    &=\pi^{(n-1)}\left(\delta\left(\varphi^{(n)}(u)\right)\circ \delta\left(\varphi^{(n)}(v)\right)\right),
    \end{align*}
    where $\delta=\varphi^{(n-1)}\circ \varepsilon\circ \pi^{(n)}:\Fg_\circ(A_n)\to \Fg_\circ(A_{n-1})$. Thus, in order to prove the claim, we need to prove that $\delta$ is a group homomorphism.
    
    We first show by induction on the length of $u\in \Fg_\circ(A_n)$ the following claim

   \begin{eqnarray}\label{epsilonlambda}
    \varepsilon\circ\lambda^{(n)}_u=\lambda^{(n-1)}_{\delta(u)}\circ\varepsilon,
    \end{eqnarray}

where we denote $\lambda^{(k)}$ the $\lambda$-action of  $F_{R_k,X}$ for $k\in\N$.
    Clearly, for every $a\in A_n$, we have  
    $$\delta(a)=\begin{cases}
    a&\text{ if } a\in A_{n-1}\\
    0&\text{ if } a\in X^{n}\\
    \end{cases}.$$
    Moreover, by definition, for every $a,b\in A_n$
    \[
    \varepsilon\left(\lambda^{(n)}_a(b)\right)
    =\begin{cases}
    0&\text{ if } k(a)=n=k(b)\text{ or } k(a)<k(b)=n\\
    b&\text{ if } k(a)=n>k(b)\\
    \lambda^{(n-1)}_a(b)&\text{ if } a,b\in A_{n-1}
    \end{cases}
    =\lambda^{(n-1)}_{\delta(a)}(\varepsilon(b)).
    \]
     So, since $\lambda^{(n)}_a,\lambda^{(n-1)}_{\delta(a)}$ and $\varepsilon$
     are all group homomorphisms, we conclude that 
      \begin{eqnarray} \label{epsilondelta}
     \varepsilon\circ\lambda^{(n)}_a=\lambda^{(n-1)}_{\delta(a)}\circ\varepsilon, \;   \text{ for every } a\in A_n.
     \end{eqnarray}
    This implies 
    \begin{align*}
        \delta(\overline{a})
        &=\varphi^{(n-1)}\left(\varepsilon\left(\pi^{(n)}\left(\overline{a}\right)\right)\right)\\
        &=\varphi^{(n-1)}\left(\varepsilon\left(-\left(\lambda_a^{(n)}\right)^{-1}(a)\right)\right)   \\
        &=\varphi^{(n-1)}\left(-\left(\lambda_{\delta(a)}^{(n-1)}\right)^{-1}(\varepsilon(a))\right) \\   
        &=\begin{cases}
    \varphi^{(n-1)}\left(-\left(\lambda_{a}^{(n-1)}\right)^{-1}(a)\right)&\text{ if } a\in A_{n-1}\\
    \varphi^{(n-1)}(0)&\text{ if } a\in X^{n}
    \end{cases}\\
    &=\begin{cases}
    \varphi^{(n-1)}\left(\pi^{(n-1)}\left(\overline{a}\right)\right)&\text{ if } a\in A_{n-1}\\
    0&\text{ if } a\in X^{n}
    \end{cases}\\
    &=\begin{cases}
    \overline{a}&\text{ if } a\in A_{n-1}\\
    0&\text{ if } a\in X^{n}
    \end{cases}\\
    &=\overline{\varepsilon(a)},
    \end{align*}
    i.e. 
    $$\delta(\overline{a})=\overline{\varepsilon(a)}  \;\;  \text{ 
    for every }a\in A_n.$$

    Now, because of (\ref{epsilondelta}), we also have that for every $a\in A_n,$
    \[
    \varepsilon\circ\lambda_{\overline{a}}^{(n)}
    =\varepsilon\circ\left(\lambda_a^{(n)}\right)^{-1}
    =\left(\lambda_{\delta(a)}^{(n-1)}\right)^{-1}\circ\varepsilon\circ\lambda_a^{(n)}\circ\left(\lambda_a^{(n)}\right)^{-1}
    =\left(\lambda_{\delta(a)}^{(n-1)}\right)^{-1}\circ\varepsilon.
    \]
    Let us now assume that $c\in A_n\cup\overline{A_n}$
    and that $\varepsilon\circ\lambda^{(n)}_u=\lambda^{(n-1)}_{\delta(u)}\circ\varepsilon$ for some $u\in \Fg_\circ(A_n)$.
    Then, because of \cref{phi of sum}, \cref{cocycle} and (\ref{epsilondelta}),
    \begin{align*}
    \delta&(u\circ c)\\
    &=\varphi^{(n-1)}\left(\varepsilon\left(\pi^{(n)}\left(u\circ c\right)\right)\right)\\
    &=\varphi^{(n-1)}\left(\varepsilon\left(\pi^{(n)}(u)+\lambda^{(n)}_u\left(\pi^{n}(c)\right)\right)\right)\\
    &=\varphi^{(n-1)}\left(\varepsilon\left(\pi^{(n)}(u)\right)+\varepsilon\left(\lambda^{(n)}_u\left(\pi^{n}(c)\right)\right)\right)\\
    &=\varphi^{(n-1)}\left(\varepsilon\left(\pi^{(n)}(u)\right)+\lambda^{(n-1)}_{\delta(u)}\left(\varepsilon\left(\pi^{n}(c)\right)\right)\right)\\
    &=\varphi^{(n-1)}\left(\varepsilon\left(\pi^{(n)}(u)\right)\right)\circ\varphi^{(n-1)}\left(\left(\lambda^{(n-1)}_{\varphi^{(n-1)}\left(\varepsilon\left(\pi^{(n)}(u)\right)\right)}\right)^{-1}\left(\lambda^{(n-1)}_{\delta(u)}\left(\varepsilon\left(\pi^{n}(c)\right)\right)\right)\right)\\
    &=\delta(u)\circ\varphi^{(n-1)}\left(\left(\lambda^{(n-1)}_{\delta(u)}\right)^{-1}\left(\lambda^{(n-1)}_{\delta(u)}\left(\varepsilon\left(\pi^{n}(c)\right)\right)\right)\right)\\
     &=\delta(u)\circ\varphi^{(n-1)}\left(\varepsilon\left(\pi^{n}(c)\right)\right)\\
     &=\delta(u)\circ\delta(c).
    \end{align*}
    Hence
    \begin{align*}
        \varepsilon\circ\lambda^{(n)}_{u\circ c}
    &=\varepsilon\circ\lambda^{(n)}_u\circ \lambda^{(n)}_c\\
    &= \lambda^{(n-1)}_{\delta(u)}\circ\varepsilon\circ\lambda^{(n)}_c\\
    &= \lambda^{(n-1)}_{\delta(u)}\circ\lambda^{(n-1)}_{\delta(c)}\circ \varepsilon\\
    &=\lambda^{(n-1)}_{\delta(u)\circ\delta(c)}\circ\varepsilon\\
    &= \lambda^{(n-1)}_{\delta(u\circ c)}\circ\varepsilon.
    \end{align*}
    So we proved the claim (\ref{epsilonlambda}).
    
    Finally we get
    \begin{align*}
       \delta &(v\circ w)
    \; =\; \varphi^{(n-1)}\left(\varepsilon\left(\pi^{(n)}\left(v\circ w\right)\right)\right)\\
   &=\varphi^{(n-1)}\left(\varepsilon\left(\pi^{(n)}(v)+\lambda^{(n)}_u\left(\pi^{n}(w)\right)\right)\right)\\
    &=\varphi^{(n-1)}\left(\varepsilon\left(\pi^{(n)}(v)\right)+\varepsilon\left(\lambda^{(n)}_u\left(\pi^{n}(w)\right)\right)\right)\\
    &=\varphi^{(n-1)}\left(\varepsilon\left(\pi^{(n)}(v)\right)+\lambda^{(n-1)}_{\delta(v)}\left(\varepsilon\left(\pi^{n}(w)\right)\right)\right)\\
    &=\varphi^{(n-1)}\left(\varepsilon\left(\pi^{(n)}(v)\right)\right)\circ\varphi^{(n-1)}\left(\left(\lambda^{(n-1)}_{\varphi^{(n-1)}\left(\varepsilon\left(\pi^{(n)}(v)\right)\right)}\right)^{-1}\left(\lambda^{(n-1)}_{\delta(v)}\left(\varepsilon\left(\pi^{n}(w)\right)\right)\right)\right)\\
    &=\delta(v)\circ\varphi^{(n-1)}\left(\left(\lambda^{(n-1)}_{\delta(v)}\right)^{-1}\left(\lambda^{(n-1)}_{\delta(v)}\left(\varepsilon\left(\pi^{n}(v)\right)\right)\right)\right)\\
     &=\delta(v)\circ\varphi^{(n-1)}\left(\varepsilon\left(\pi^{n}(w)\right)\right)\\
     &=\delta(v)\circ\delta(w),
    \end{align*}
    for every $v,w\in \Fg_\circ(A_n)$. So, indeed, $\varepsilon$ is a skew brace homomorphism.

    Since, by inductive hypothesis, $\FSB_{\mathcal{RN}_{n-1},X}$
    is right nilpotent of class $n-1$
    and $\varepsilon$ is a skew brace homomorphism,
    $\varepsilon\left(\FSB_{\mathcal{RN}_n,X}^{{(n)}}\right)\subseteq \FSB_{\mathcal{RN}_{n-1},X}^{(n)}=\{0\}$, 
    hence $$\FSB_{\mathcal{RN}_n,X}^{(n)} \subseteq \ker(\varepsilon).$$ 
    
    We claim that  $\ker(\varepsilon)  = \left< a+b-a \mid a \in \Fg_+(A_n),\quad b\in X^{(n)}\right>$. 
    To prove this, put  $K=\left< a+b-a \mid a \in \Fg_+(A_n), b\in X^{(n)}\right>_+$. Clearly $K\subseteq \ker(\varepsilon)$.
    To prove the other inclusion, we proceed by induction on the 
    length of the reduced form.
    Let $w\in \ker(\varepsilon)$. If $w\in A_n\cup -A_n$,
    then 
    \[
    0=\varepsilon(w)=\begin{cases}
        w &\text{ if }w\in \pm A_{n-1}\\
        0 &\text{ if }w\in \pm X^{(n)}.
    \end{cases}
    \]
    So $w\in X^{(n)}\subseteq K$.
    
     Assume now that $w\in\ker(\varepsilon)$ and that
    $w=l_1+l_2+\ldots +l_s$ is the reduced form of $w$, with $s>1$ 
    and $l_j\in A_n\cup -A_n$ for every $j\in\{1,\dots, s\}$.
    If $l_j\in A_{n-1}\cup -A_{n-1}$ for every $j\in\{1,\dots, s\}$,
    then 
    \[
    0=\varepsilon(w)=l_1+l_2+\ldots+l_s,
    \]
    but this is a contradiction with the fact that it is a reduced form.
    Hence, there is $i\in\{1,\ldots,s\}$ such that
    $l_i\in X^{(n)}$ and $l_j\in A_{n-1}$ for every $j<i$.
    Let $u_1=l_1+\ldots+l_{i-1}$ and $v_1=l_{i+1}+\ldots +l_s$,
    so that $w=u_1+l_i+v_1$ and
    \[
    0=\varepsilon(w)
    =\varepsilon(u_1+l_i+v_1)
    =\varepsilon(u_1+v_1).
    \]
    Then, since $w_1=u_1+v_1\in\ker(\varepsilon)$ and
    has a reduced form of length $\leq s$,
    by inductive hypothesis,
    $u_1+v_1\in K$.
    Moreover, $u_1+l_i-u_1\in K$ and so
    \[
    w=u_1+l_i+v_1=u_1+l_i-u_1+w_1\in K.
    \]
    This proves the claim.

     Assume now that $a \in \Fg_+(A_n)$ and $b\in X^{(n)}$.
Using  $\lambda_a(b)\in X^{(n)}\subseteq \ker(\lambda)$, we have that
    \begin{align*}
       a\circ\lambda_a(b)\circ\overline{a}
    &=a\circ\left(\lambda_a(b)+\lambda_{\lambda_a(b)}(\overline{a})\right)\\
    &=a\circ\left(\lambda_a(b)+\overline{a}\right)\\
    &=a\circ\lambda_a(b)-a+a\circ\overline{a}\\
    &=a+\lambda_a (\lambda_a (b))-a. 
    \end{align*} 
    As $b\in \ker(\lambda)$ we obtain  $a\circ\lambda_a(b)\circ\overline{a}\in\ker(\lambda)$ for all $b\in X^n$.
    
   Hence, $a+\lambda_a (\lambda_a (b))-a \in \ker(\lambda)$ for every $ a \in \Fg_+(A_n)$ and $b\in X^{(n)}$.
   It follows from the previous claim that $\ker(\varepsilon)\subseteq \ker(\lambda)$. Thus, $\FSB_{\mathcal{RN}_n,X}^{(n)} 
   \subseteq \ker(\varepsilon)\subseteq \ker(\lambda)$. Consequently  $\FSB_{\mathcal{RN}_n,X}^{(n+1)}=\{0\}$, as required.
   Hence we have shown that $\FSB_{\mathcal{RN}_n,X}$ is right nilpotent of class  at most $n$. Because of  \cref{star prod in FRn} we know that $\{ 0\} \neq X^{(n)} \subseteq \FSB_{\mathcal{RN}_n,X}^{(n)}$. Since $\FSB_{\mathcal{RN}_n,X}^{(n)}$ is an ideal, it is in particular a
normal subgroup of the additive group. Hence it contains the normal additive
subgroup generated by $X^{(n)}$, namely
\[
\left\langle a+b-a \mid a\in \Fg_+(A_n),\ b\in X^{(n)}\right\rangle_+.
\]
By the claim proved above, this subgroup is exactly $\ker(\varepsilon)$.

Therefore $\ker(\varepsilon)\subseteq \FSB_{\mathcal{RN}_n,X}^{(n)}$ and hence $\FSB_{\mathcal{RN}_n,X}^{(n)}=\ker(\varepsilon)$.
In particular, since $X^{(n)}\neq\{0\}$ and $X^{(n)}\subseteq
\FSB_{\mathcal{RN}_n,X}^{(n)}$, we have $\FSB_{\mathcal{RN}_n,X}^{(n)}\neq 0$.
Therefore, $\FSB_{\mathcal{RN}_n,X}$ is right nilpotent of class $n$.
\end{proof}

As a direct consequence of the proof of \cref{right nilp class n},  we obtain the following.

\begin{cor}\label{quotient-by-last-term}
    Denote $B=\FSB_{\mathcal{RN}_n,X}$. Then, $B/B^{(n)}\cong \FSB_{\mathcal{RN}_{n-1},X}$ and $$B^{(n)}=\langle a+b-a\mid b\in X^{(n)},a\in \Fg_{+}(A_n)\rangle$$ is a trivial skew brace whose additive group is free with
 free basis given by the set $\left\lbrace a+b-a \mid a \in F_{+}(A_{n-1}), b \in X^{(n)}\right\rbrace$.
\end{cor}
\begin{proof}
    It remains to show that the given set is a free basis.
    We know that the group $(B,+)/(B^{(n)},+)$
 is isomorphic with the free group $\Fg_{+}(A_{n-1})$, so $\Fg_{+}(A_{n-1})$ is a right (Schreier) transversal for $B^{(n)}$
in~$B$. Furthermore, $(B,+)$ is a free group with basis $A_n=A_{n-1}\cup X^{(n)}$.
Hence, by the~Reide\-meister--Schreier method, $(B^{(n)},+)$ is generated by the elements $f +a -\overline{f+a}$, with $f\in \Fg_{+}(A_{n-1})$ and $a\in A_n$, where $\overline{f+a}$ is the the unique element $f'\in \Fg_{+}(A_{n-1})$ such that $ f +a \in B^{(n)}+f'$. Of course, if~\hbox{$a\in A_{n-1}$}, then $f'=f+a$.
If~\hbox{$a\in X^{(n)}$}, then 
$f+a+B^{(n)}=f+B^{(n)} = B^{(n)}+f$. Hence, in this case $f +a -\overline{f+a}=f+a-f$. Thus, the result follows.
\end{proof}

Finally, we are in a position to prove the main result of this section.

\begin{thm}\label{FreeNilpn}
$\FSB_{\mathcal{RN}_n,X}$  is the free skew brace on non-empty $X$ in the category of skew braces that are right nilpotent of class at most $n$.
\end{thm}
\begin{proof}
We have already shown that $\FSB_{\mathcal{RN}_n,X}$ is a skew brace of right nilpotency class~$n$. Hence, it only remains to prove that this is the free object in this category. We will show that $\FSB_{\mathcal{RN}_n,X}$ satisfies the required universal property.

    Let $B$ be a skew brace of right nilpotency class at most $n$ and $\phi\colon X \rightarrow B$ a map. Then, we need to show that there exists a unique skew brace morphism $\psi\colon \FSB_{\mathcal{RN}_n,X} \rightarrow B$ that factors $\phi$.

    We first inductively define a map $\psi\colon A_n \rightarrow B$. As $A_1=X$, we define $\psi(a)=\phi(a)$ for $a \in A_1$.

    Assume we have defined the map on $A_i$. Then, we define for $a \in X^{(i)}$ and $b \in A_i$ that
    $$ \psi\left( x^{(i+1)}_{a,b}\right)= \psi(a)\ast \psi(b)
    \quad\text{ and }\quad\psi(x^{(i+1)}_{\overline{a},b})=\overline{\psi(a)} \ast \psi(b).$$
    For $f \in F_{i+1}$ such that $f=f'\circ a$ (resp. $f=f'\circ\overline{a}$) in reduced form with $a \in X^{(i)}$ and $b\in A_i$ we define 
    $$ \psi(x^{(i+1)}_{f,b}) = \lambda_{\psi(\pi(f'))}(\psi(a)\ast \psi(b))$$
    $$(\textnormal{resp.}\;     \psi(x^{(i+1)}_{f,b})=\lambda_{\psi(\pi(f'))}(\overline{\psi(a)} \ast \psi(b))).$$
    Now, $\psi$ can be extended to a unique group morphism $\psi\colon \Fg_+(A_n)\rightarrow B$. It remains to show that this is a skew brace morphism.
    To prove this, we will prove that $\psi\circ\pi:\Fg_\circ(A_n)\to (B,\circ)$
    is a group homomorphism.
    Pre-composing with the bijective 1-cocycle allows us to work with words in $\Fg_\circ(A_n)$.
    In fact, we first prove, by induction on the length of $u\in F_\circ(A_n)$, that
    \[
    \psi\circ\lambda_u=\lambda_{\psi(\pi(u))}\circ\psi.
    \]

    We start with proving by induction on $i$
    that the equality holds in every $A_i$.
    Let $a\in X=A_1$.
    Then $\pi(a)=a$ and for every $b\in A_1$, using that $B$ is right nilpotent of class at most $n$,
    \begin{align*}
        \psi\left(\lambda_a(b)\right)
    &=\begin{cases}
        \psi(b)&\text{ if } n=1\\
        \psi\left(x^{(2)}_{a,b}+b\right)&\text{ if }n>1 
    \end{cases}\\
    &=\begin{cases}
        \psi(b)&\text{ if } n=1\\
        \psi(a)*\psi(b)+\psi(b)&\text{ if }n>1 
    \end{cases}\\
    &=\lambda_{\psi(a)}(\psi(b)).
    \end{align*}
    Moreover,
    \begin{align*}
        \psi(\pi(\overline{a}))
    &=\psi\left(-\lambda_a^{-1}(a)\right)
    =\begin{cases}
        \psi(-a)&\text{ if } n=1\\
        \psi\left(-a-x^{(2)}_{\overline{a},a}\right)&\text{ if }n>1 
    \end{cases}\\
    &=\begin{cases}
        -\psi(a)&\text{ if } n=1\\
        -\psi(a)-\overline{\psi(a)}*\psi(a)&\text{ if }n>1 
    \end{cases}\\
    &=\overline{\psi(a)}.
    \end{align*}

    Assume now we know that 
    \[
    \psi\left(\lambda_a(b)\right)=\lambda_{\psi(\pi(a))}(\psi(b))
    \]
    for every $a,b\in A_i$.
    Then 
    \begin{align*}
        \psi(\pi(a\circ b))
    &=\psi(\pi(a)+\lambda_a(\pi(b)))\\
    &=\psi(\pi(a))+\psi\left(\lambda_a(\pi(b))\right)\\
    &=\psi(\pi(a))+\lambda_{\psi(\pi(a))}(\psi(b))\\
    &=\psi(\pi(a))\circ\psi(\pi(b)),
    \end{align*}
    for every $a\in \Fg_\circ(A_i)$, and $b\in A_i$.
    Let $a\in A_i$ and $v\in X^{(i+1)}$, i.e., $v=x^{(i+1)}_{f\circ c,t}$,
    for some $f\in \Fg_\circ(A_i)$ and $c\in X^{(i)}\cup \overline{X^{(i)}}$.
    Then, since $k(a)\leq i <i+1=k(v)$,
    \begin{align*}
        \psi\left(\lambda_a\left(v\right)\right)
        &=\psi\left(\lambda_a\left(x^{(i+1)}_{f\circ c,t}\right)\right)
        =\begin{cases}
            \psi\left(x^{(i+1)}_{a\circ f\circ c,t}\right)&\text{ if }f\circ c\neq \overline{a}\\
            \psi\left(-x^{(i+1)}_{a,t}\right)&\text{ if }f\circ c=\overline{a}
        \end{cases}\\
         &=\begin{cases}
            \psi\left(\lambda_{\psi\left(\pi\left(a\circ f\right)\right)}\left(\psi(c)*\psi(t)\right)\right)&\text{ if }c\in X^{(i)}\\
            \psi\left(\lambda_{\psi\left(\pi\left(a\circ f\right)\right)}\left(\overline{\psi(d)}*\psi(t)\right)\right)&\text{ if }f\circ c\neq \overline{a}\text{, }c=\overline{d},\; d\in X^{(i)}\\
            -\left(\psi(a)*\psi(t)\right)&\text{ if }f\circ c=\overline{a}
        \end{cases}
    \end{align*}
    Since $a\in A_i$ and $f\in \Fg_\circ(A_i)$,
    \[
    \psi\left(\pi\left(a\circ f\right)\right)
    =\psi\left(\pi\left(a\right)\right)\circ\psi\left(\pi\left(f\right)\right)
    =\psi\left(\pi(a)\right)\circ\psi\left(\pi\left(f\right)\right),
    \]
    hence,
    \begin{align*}
        \psi(\lambda_a&(v))\\
        &=\begin{cases}
            \psi\left(\lambda_{\psi\left(\pi(a)\right)}\left(\lambda_{\psi\left(\pi\left(f\right)\right)}\left(\psi(c)*\psi(t)\right)\right)\right)&\text{ if }c\in X^{(i)}\\
            \psi\left(\lambda_{\psi\left(\pi(a)\right)}\left(\lambda_{\psi\left(\pi\left(f\right)\right)}\left(\overline{\psi(d)}*\psi(t)\right)\right)\right)&\text{ if }f\circ c\neq \overline{a}\text{, }c=\overline{d}, d\in X^{(i)}\\
            \lambda_{\psi(a)}\left(\psi\left(x^{(i+1)}_{\overline{a},t}\right)\right)&\text{ if }f\circ c=\overline{a}
        \end{cases}\\
        &=\begin{cases}
            \lambda_{\psi\left(\pi(a)\right)}\left(\psi\left(x^{(i+1)}_{f\circ c,t}\right)\right)&\text{ if }c\in X^{(i)}\\
            \lambda_{\psi\left(\pi(a)\right)}\left(\psi\left(x^{(i+1)}_{f\circ c,t}\right)\right)&\text{ if }f\circ c\neq \overline{a}\text{, }c=\overline{d}, d\in X^{(i)}\\
            \lambda_{\psi(a)}\left(\psi\left(x^{(i+1)}_{\overline{a},t}\right)\right)&\text{ if }f\circ c=\overline{a}
        \end{cases}\\
        &=\lambda_{\psi\left(\pi(a)\right)}\left(v\right)
    \end{align*}
    Finally, if $a\in X^{(i+1)}$ and $b\in A_{i+1}$,
    then, since $k(a)\geq k(b)$ and $B$ is right nilpotent of class at most $n$,
    \begin{align*}
       \psi(\lambda_a(b))
        &=\begin{cases}
        \psi(b) &\text{ if } i+1=n\\
        \psi\left(x^{(i+2)}_{a,b}+b\right)&\text{ if } i+1<n
        \end{cases}\\
        &=\begin{cases}
        \psi(b) &\text{ if } i+1=n\\
        \psi(a)*\psi(b)+\psi(b)&\text{ if } i+1<n
        \end{cases}\\
        &=\lambda_{\psi(a)}(\psi(b)). 
    \end{align*}
    Therefore 
    \[
    \psi\left(\lambda_a(b)\right)=\lambda_{\psi(\pi(a))}(\psi(b)),
    \]
    for every $a,b\in A_n$.
    Hence, as before, we can also conclude that
    \begin{align*}
        \psi(\pi(\overline{a}))
    &=\psi\left(-\lambda_a^{-1}(a)\right)\\
    &=\begin{cases}
        \psi(-a) &\text{ if } k(a)=n\\
        \psi\left(-a-x^{(k(a)+1)}_{\overline{a},a}\right)&\text{ if } k(a)<n
    \end{cases}\\
    &=\begin{cases}
        -\psi(a) &\text{ if } k(a)=n\\
        -\psi(a)-\overline{\psi(a)}*\psi(a)&\text{ if } k(a)<n
    \end{cases}
    =\overline{\psi(a)}
    \end{align*}
    and
    \begin{align*}
        \psi(\pi(a\circ b))
    &=\psi(\pi(a)+\lambda_a(\pi(b)))\\
    &=\psi(\pi(a))+\psi\left(\lambda_a(\pi(b))\right)\\
    &=\psi(\pi(a))+\lambda_{\psi(\pi(a))}(\psi(b))\\
    &=\psi(\pi(a))\circ\psi(\pi(b)),
    \end{align*}
    for every $a,b\in A_n$.
    Moreover, since $\psi\circ\lambda_a$ and $\lambda_{\psi(\pi(a))}\circ\psi$ are
    group homomorphisms $\Fg_+(A_n)\to (B,+)$ and they coincide on $A_n$,
    we have that
    \[
    \psi\circ\lambda_a=\lambda_{\psi(\pi(a))}\circ\psi,
    \]
    for every $a\in A_n$.

    We can now proceed with the inductive step.
    Let $u\in \Fg_\circ(A_n)$ and $c\in A_n\cup\overline{A_n}$
    Assume that 
    \[
    \psi\circ\lambda_u=\lambda_{\psi(\pi(u))}\circ\psi.
    \]
    Then 
    \begin{align*}
        \psi(\pi(u\circ c))
    &=\psi(\pi(u)+\lambda_u(\pi(b)))\\
    &=\psi(\pi(u))+\psi\left(\lambda_u(\pi(c))\right)\\
    &=\psi(\pi(u))+\lambda_{\psi(\pi(u))}(\psi(c))\\
    &=\psi(\pi(u))\circ\psi(\pi(c)).
    \end{align*}
    Therefore,
    \begin{align*}
        \psi\circ\lambda_{u\circ c}
    &=\psi\circ\lambda_u\circ\lambda_c\\
    &=\lambda_{\psi(\pi(u))}\circ\psi\circ\lambda_c\\
    &=\lambda_{\psi(\pi(u))}\circ\lambda_{\psi(\pi(c))}\circ\psi\\
    &=\lambda_{\psi(\pi(u))\circ\psi(\pi(c))}\circ\psi\\
    &=\lambda_{\psi(\pi(u\circ c))}\circ\psi.
    \end{align*}
    With this, we proved that 
    \[
    \psi\circ\lambda_u=\lambda_{\psi(\pi(u))}\circ\psi,
    \]
    for every $u\in \Fg_\circ(A_n)$.

    Finally, we can conclude that
    \begin{align*}
        \psi(\pi(v\circ w))
    &=\psi(\pi(v)+\lambda_v(\pi(w)))\\
    &=\psi(\pi(v))+\psi\left(\lambda_v(\pi(w))\right)\\
    &=\psi(\pi(v))+\lambda_{\psi(\pi(v))}(\psi(w))\\
    &=\psi(\pi(v))\circ\psi(\pi(w)),
    \end{align*}
    for ever $v,w\in \Fg_\circ(A_n)$. Thus, $\psi$ is a skew brace morphism.

    The last thing to prove for the universal property
    is the uniqueness of $\psi$.
    Let $\psi':\FSB_{\mathcal{RN}_n,X}\to B$ be a skew brace homomorphism
    such that $\psi'(a)=\phi(a)=\psi(a)$
    for every $a\in X=A_1$.
    We will prove by induction on $i$ that $\psi'$
    coincides with $\psi$ on $A_i$.
    Assume that $\psi'(a)=\psi(a)$ for every $a\in A_i$.
    Let 
    $f\in \Fg_\circ(A_i),$ $a\in X^{(i)}$  and $b\in A_i$.
    Then, by \cref{star prod in FRn},
    \[
    \psi'\left(x^{(i+1)}_{a,b}\right)
    =\psi'(a\ast b)
    =\psi'(a)\ast\psi'(b)
    =\psi(a)\ast\psi(b)
    =\psi(a\ast b)
    =\psi\left(x^{(i+1)}_{a,b}\right),
    \]
    \[
    \psi'\left(x^{(i+1)}_{\overline{a},b}\right)
    =\psi'(\overline{a}\ast b)
    =\overline{\psi'(a)}\ast\psi'(b)
    =\overline{\psi(a)}\ast\psi(b)
    =\psi(\overline{a}\ast b)
    =\psi\left(x^{(i+1)}_{\overline{a},b}\right)
    \]
    \begin{align*}
        \psi'\left(x^{(i+1)}_{f\circ a,b}\right)
    &=\psi'\left(\lambda_{f}(a\ast b)\right)
    =\lambda_{\psi'(f)}(\psi'(a)\ast\psi'(b))
    =\lambda_{\psi(f)}(\psi(a)\ast\psi(b))\\
    &=\psi\left(\lambda_{f}(a\ast b)\right)
    =\psi'\left(x^{(i+1)}_{f\circ a,b}\right)    
    \end{align*}
    \begin{align*}
        \psi'\left(x^{(i+1)}_{f\overline{a},b}\right)
    &=\psi'\left(\lambda_{f}(\overline{a}\ast b)\right)
    =\lambda_{\psi'(f)}(\overline{\psi'(a)}\ast\psi'(b))
    =\lambda_{\psi(f)}(\overline{\psi(a)}\ast\psi(b))\\
    &=\psi\left(\lambda_{f}(\overline{a}\ast b)\right)
    =\psi'\left(x^{(i+1)}_{f\circ \overline{a},b}\right).
    \end{align*}
    Hence $\psi'(v)=\psi(v)$ for every $v\in X^{(i+1)}$.
    So $\psi'(a)=\psi(a)$ for every $v\in A_{i+1}$ 
    and this concludes the induction.

    Finally, since $\psi$ and $\psi'$ are also additive group homomorphisms
    that coincide on~$A_n$, they coincide on $\Fg_+(A_n)$.
    Hence $\psi=\psi'$.
\end{proof}

\section{Structural consequences for free right nilpotent skew braces}\label{rightnilclass2}

The aim of this section is to exploit the construction of the free right nilpotent skew brace of class at most $n$ given in Section \ref{sec: right nilp}. We do this in two respects. On  one hand, we prove that such free objects $\FSB_{\mathcal{RN}_n,X}$ are residually finite (see~The\-o\-rem~\ref{classnresfinhopfia}), so even Hopfian provided $X$ is finite (although $\FSB_{\mathcal{RN}_n,X}$ is not co-Hopfian in general). On the other hand, for $n=2$, we describe the regular subgroup corresponding to $\FSB_{\mathcal{RN}_2,\{ x_0\}}$ (see \cref{descriptionregular}),
which in turn allows us to show that 
$\FSol_{\mathcal{RN}_n,X}$ is not an injective solution (see \cref{solnotinjective}).

\medskip

The reader should also note the following auxiliary results, which give further insights on the structure of $\FSB_{\mathcal{RN}_n,X}$: \cref{quotient-by-last-term} and \cref{freebasisy}.
In particular, in the case $n=2$ and $X=\{x_0\}$, we obtain a more explicit description of the relevant quotients and associated solutions, which allows us to compare the free right nilpotent skew brace with the free right nilpotent skew brace of abelian type and to describe the free injective solution in the corresponding category.

\subsection{Residual Finiteness}

In order to prove that free right nilpotent skew braces of class at most $n$ are residually finite, we need some preliminary results that allow us to find suitable behaving quotients. Throughout this subsection $X$ is a non-empty finite set and  $n\geq 2$.  In  $B=\FSB_{\mathcal{RN}_{n},X}$ we define the following elements: $$y_{f,a}=\lambda_f(a) \text{ with } f\in \Fg_{\circ}(A_{n-1}),\, a \in X.$$
In particular, $y_{e,x}=x$ for $x\in X$. Our key auxiliary result (see \cref{freebasisy}) shows that these elements form a free basis for $(B,+)$.

We also define the functions 
$$h\colon A_n \rightarrow \mathbb{N}, \quad t\colon A_n\rightarrow X\quad\textnormal{ and }\quad j\colon A_n \rightarrow F_{\circ}(A_{n-1})$$ inductively as follows.

Set $h(A_1)=0$, $t|_{A_1}=\id_X$ and $j(A_1)=e$, where \hbox{$\Fg_{\circ}(A_{0})=\{e\}$.}  Suppose we have defined $h,t$ and $j$ on $A_i$ with $i<n$. Then, for $w=x^{(i+1)}_{f,b}$, we set $h(w)=\ell(f)+h(b)$, $t(w)=t(b)$ and $j(w) = f\circ j(b)$. 

\begin{lem}\label{helpfullemma}
    For a positive integer $k$, if $b\in A_n$ with $h(b)= k$ then $$b\in \langle  y_{w,a} \mid \ell(w) \leq k,\, a\in X \rangle_{+}.$$ Moreover, $$ b = y_{j(b),t(b)} + r, \text{ with } r\in \langle y_{w,a} \mid \ell(w) < k,\, a\in X \rangle_{+}.$$
\end{lem}
\begin{proof}
Let $b\in A_n$. We prove the result by induction on $k(b)$. If $k(b)=1$, i.e. $b\in X^{(1)}=X=A_1$ then the result clearly holds. Suppose we have shown the result for $b \in A_{i-1}$.

    Let $b \in X^{(i)}$. Then there exists $f \in F_i$ and $a \in A_{i-1}$
    such that $b=x_{f,a}^{(i)}$. Denote $f=w\circ f'$ in reduced form. Then, by Lemma \ref{star prod in FRn}, $$ b = x_{f,a}^{(i)} = \lambda_w(\lambda_{f'}(a)-a) = \lambda_{f}(a)-\lambda_w(a).$$

By the induction hypothesis applied to \(a\), 
we may write
\[
a=y_{j(a),t(a)}+r,
\]
where \(r\) belongs to the additive subgroup generated by the elements $y_{u,x}$ with $x\in X$ and $\ell(u)<h(a)$.
Applying \(\lambda_f\), we obtain
\[
\lambda_f(a)
=
\lambda_f(y_{j(a),t(a)})+\lambda_f(r).
\]
Since \(y_{u,x}=\lambda_u(x)\), we have
\[
\lambda_f(y_{j(a),t(a)})
=
\lambda_f(\lambda_{j(a)}(t(a)))
=
\lambda_{f\circ j(a)}(t(a)).
\]
By definition,
\[
j(b)=f\circ j(a)
\qquad\text{and}\qquad
t(b)=t(a),
\]
and thus 
\[
\lambda_f(y_{j(a),t(a)})=y_{j(b),t(b)}.
\]

Moreover, since \(r\) is an additive word in elements \(y_{u,x}\) with
\(x\in X\) and \(\ell(u)<h(a)\), the element \(\lambda_f(r)\) is an additive
word in elements of the form
\[
\lambda_f(y_{u,x})=y_{f\circ u,x}.
\]
For each such element we have
\[
\ell(f\circ u)\leq \ell(f)+\ell(u)<\ell(f)+h(a)=h(b).
\]
Therefore, all terms appearing in \(\lambda_f(r)\) have length strictly smaller
than \(h(b)\). Similarly, applying the induction hypothesis to \(a\), the
terms appearing in \(\lambda_w(a)\) are of the form \(y_{w\circ u,x}\), with
\(\ell(u)\leq h(a)\). Since \(\ell(w)<\ell(f)\), we have
\[
\ell(w\circ u)\leq \ell(w)+\ell(u)<\ell(f)+h(a)=h(b).
\]
Thus all terms appearing in \(\lambda_w(a)\) also have length strictly smaller
than \(h(b)\).
\end{proof}

Also the following map is needed:
 $$q_{n-1}\colon \Fg_{\circ}(A_{n-1}) \rightarrow \mathbb{N},$$ 
with $q_{n-1}(a)$  the number of elements of $X^{(n-1)}$ that appear in the  reduced form of $a$. Define the sets $$ X_{i,k}^{(n)}=\left\{x^{(n)}_{f,b} \in X^{(n)} \mid h(b)\geq k, \; q_{n-1}(f)=i \right\}$$ and $$Y_{i,k}^{(n)}=\left\{ y_{f\circ b,a} \mid a \in X,\;  f \in F_n,\; \ b \in F_{\circ}(A_{n-2}),\;  h(b)<k, \; q_{n-1}(f)=i \right\}.$$ We set 
$$Y_{\leq i}^{(n)}= \bigcup_{j=1}^{i}\bigcup_{k=0}^{\infty}Y_{j,k}^{(n)} \quad \text{ and } \quad
X_{\geq i}^{(n)}=\bigcup_{j=i}^{\infty}\bigcup_{k=0}^{\infty}X_{j,k}^{(n)}.$$

\begin{pro}\label{freebasisy}
    Let $X$ be a non-empty set and let $B=\FSB_{\mathcal{RN}_{n},X}$. Then 
    $$ \left\lbrace y_{f,a}=\lambda_f(a) \mid f \in F_{\circ}(A_{n-1}),\, a \in X \right\rbrace$$ is a basis for the free group $(B,+)$.
\end{pro}
\begin{proof}
We prove by induction on $i\geq 1$ that  $\left\{y_{f,a} \mid f \in \Fg_{\circ}(A_{i-}),\;  a \in X\right\}$ is a free basis for $\Fg_{+}(A_{i})$. Clearly, $X$ is a free basis for $F_+(X)$ and thus the case $i=1$ holds. Assume we have shown that $Y=\left\{y_{f,a} \mid f \in \Fg_{\circ}(A_{n-2}),\;  a \in X\right\}$ is a free basis for $\Fg_{+}(A_{n-1})$. Hence, we have that $X^{(n)}\cup Y$ is a free basis for $\Fg_+(A_n)$. Assume  $f \in F_n$ with $f=w\circ f'$  in reduced form (so, $f'\in X^{(n-1)}$) with that $w\in  \Fg_{\circ}(A_{n-2})$, 
 and $a \in X$. Then $$y_{f,a} = \lambda_f(a) = \lambda_w\lambda_{f'}(a) = \lambda_w( x^{(n)}_{f',a}+a) = x^{(n)}_{f,a}+\lambda_w(a). $$   
Since
\(w\in \Fg_{\circ}(A_{n-2})\) and \(a\in X\), we have $\lambda_w(a)=y_{w,a}\in Y$.
Hence
\[
y_{f,a}=x^{(n)}_{f,a}+y_{w,a},
\]
with \(y_{w,a}\in Y\). Thus \(y_{f,a}\) is obtained from the generator
\(x^{(n)}_{f,a}\) by adding an element of the subgroup generated by the  free
basis \(Y\). Therefore, by Nielsen transformations, we may replace each
generator \(x^{(n)}_{f,a}\) with \(q_{n-1}(f)=1\) and \(a\in X\) by the
corresponding element \(y_{f,a}\). Equivalently, we replace the part
\[
X_{1,0}^{(n)}\setminus X_{1,1}^{(n)}
\]
of \(X^{(n)}\) by \(Y_{1,1}^{(n)}\). The remaining elements of \(X^{(n)}\) are
precisely
\[
X_{1,1}^{(n)}\cup X_{\geq 2}^{(n)}.
\]
Consequently,
\[
Y\cup Y_{1,1}^{(n)}\cup X_{1,1}^{(n)}\cup X_{\geq 2}^{(n)}
\]
is again a free basis of \(\Fg_+(A_n)\).

Now suppose that 
$\mathcal B_k=Y\cup Y_{1,k}^{(n)}\cup X_{1,k}^{(n)}\cup X_{\geq 2}^{(n)}$
is a free basis of $\Fg_+(A_n)$. We prove that
$\mathcal B_{k+1}=Y\cup Y_{1,k+1}^{(n)}\cup X_{1,k+1}^{(n)}\cup X_{\geq 2}^{(n)}$
again is  a free basis. In order to do this, it is enough to show that we may replace the elements of
$X_{1,k}^{(n)}\setminus X_{1,k+1}^{(n)}$ by the corresponding elements of
$Y_{1,k+1}^{(n)}\setminus Y_{1,k}^{(n)}$. 
Let
$x_{f,b}^{(n)}\in X_{1,k}^{(n)}\setminus X_{1,k+1}^{(n)}$. Then
$q_{n-1}(f)=1$ and $h(b)=k$. Since $f\in F_n$, the reduced form of $f$ ends
with an element of $X^{(n-1)}\cup\overline{X^{(n-1)}}$. As $q_{n-1}(f)=1$, we
can write $f=w\circ\varepsilon$ in reduced form, where
$w\in\Fg_\circ(A_{n-2})$ and
$\varepsilon\in X^{(n-1)}\cup\overline{X^{(n-1)}}$. In particular, $w$ is
obtained from $f$ by deleting its last letter.

By the definition of the action, we have
$x_{f,b}^{(n)}=\lambda_f(b)-\lambda_w(b)$. By Lemma~\ref{helpfullemma}, we may
write $b=y_{j(b),t(b)}+w'$, where $w'$ is an additive word in elements
$y_{t,a}$ with $\ell(t)<k$. Hence
\[
\lambda_f(b)
=
y_{f\circ j(b),t(b)}+\lambda_f(w')
\quad\textnormal{
and}\quad
\lambda_w(b)
=
y_{w\circ j(b),t(b)}+\lambda_w(w').
\]
Therefore, in the additive group,
\[
y_{f\circ j(b),t(b)}= x_{f,b}^{(n)} + y_{w\circ j(b),t(b)} + \lambda_w(w') - \lambda_f(w').
\]

Now $y_{w\circ j(b),t(b)}\in Y$, because
$w\circ j(b)\in\Fg_\circ(A_{n-2})$. Moreover, since $w'$ is an additive word
in elements $y_{t,a}$ with $\ell(t)<k$, the element $\lambda_w(w')$ is an additive
word in elements of $Y$, while $\lambda_f(w')$ is an additive word in elements
of $Y_{1,k}^{(n)}$. Indeed,
$\lambda_w(y_{t,a})=y_{w\circ t,a}\in Y$ and
$\lambda_f(y_{t,a})=y_{f\circ t,a}\in Y_{1,k}^{(n)}$ whenever $\ell(t)<k$.

Thus the correction term
$y_{w\circ j(b),t(b)}+\lambda_w(w')-\lambda_f(w')$
belongs to the subgroup generated by $Y\cup Y_{1,k}^{(n)}$. Hence each element
of $Y_{1,k+1}^{(n)}\setminus Y_{1,k}^{(n)}$ is obtained from the corresponding
generator in $X_{1,k}^{(n)}\setminus X_{1,k+1}^{(n)}$ by adding an additive
word in the already existing basis elements. Therefore, by Nielsen
transformations,
\[
Y\cup Y_{1,k+1}^{(n)}\cup X_{1,k+1}^{(n)}\cup X_{\geq 2}^{(n)}
\]
is again a free basis of $\Fg_+(A_n)$.

So far we have proved that for every $k\geq1$, the set
\[
\mathcal B_k=
Y\cup Y_{1,k}^{(n)}\cup X_{1,k}^{(n)}\cup X_{\geq 2}^{(n)}
\]
is a free basis of $\Fg_+(A_n)$. Since every
$b\in A_{n-1}$ has finite $h(b)$, every
$x_{f,b}^{(n)}$ with $q_{n-1}(f)=1$ belongs to
$X_{1,k}^{(n)}\setminus X_{1,k+1}^{(n)}$ for $k=h(b)$, and hence can be 
replaced after finitely many steps. Therefore, after all these Nielsen transformations,
all the $x_{f,b}^{(n)}$ with~\hbox{$q_{n-1}(f)=1$} have been replaced by the
corresponding elements of
\[
Y_1^{(n)}=\bigcup_{k\geq 0}Y_{1,k}^{(n)}.
\]
The generators with $q_{n-1}(f)\geq 2$ are not changed. Hence, we have shown that
\[
Y\cup Y_1^{(n)}\cup X_{\geq 2}^{(n)}
\]
is a free basis of $\Fg_+(A_n)$.

Assume now that, for some $k\geq 2$, we have proved that
$$
Y\cup Y_{\leq k-1}^{(n)}\cup X_{\geq k}^{(n)}
$$
is a free basis of $\Fg_+(A_n)$. We prove that
$$
Y\cup Y_{\leq k}^{(n)}\cup X_{\geq k+1}^{(n)}
$$
again is  a free basis.

\medskip
First, we claim that $$
Y\cup Y_{\leq k-1}^{(n)}\cup Y_{k,i}^{(n)}
\cup X_{k,i}^{(n)}\cup X_{\geq k+1}^{(n)}
$$
is a free basis of $\Fg_+(A_n)$ for every $i$. For $i=0$ there is nothing to prove, since $Y_{k,0}^{(n)}=\emptyset$ and
$X_{k,0}^{(n)}\cup X_{\geq k+1}^{(n)}=X_{\geq k}^{(n)}$. Hence the induction
hypothesis says exactly that
$$
Y\cup Y_{\leq k-1}^{(n)}\cup Y_{k,0}^{(n)}
\cup X_{k,0}^{(n)}\cup X_{\geq k+1}^{(n)}
$$
is a free basis.

Suppose therefore that
$$
\mathcal B_{k,i}
=
Y\cup Y_{\leq k-1}^{(n)}\cup Y_{k,i}^{(n)}
\cup X_{k,i}^{(n)}\cup X_{\geq k+1}^{(n)}
$$
is a free basis. We show that the same holds with $i+1$ in place of $i$.
It is enough to replace the elements of
$X_{k,i}^{(n)}\setminus X_{k,i+1}^{(n)}$ by the corresponding elements of
$Y_{k,i+1}^{(n)}\setminus Y_{k,i}^{(n)}$.

Let $x_{f,b}^{(n)}\in X_{k,i}^{(n)}\setminus X_{k,i+1}^{(n)}$. Then
$q_{n-1}(f)=k$ and $h(b)=i$. Write $f=w\circ\varepsilon$ in reduced form,
where $\varepsilon$ is the last letter of $f$. Thus
$\varepsilon\in X^{(n-1)}\cup\overline{X^{(n-1)}}$ and $w$ is obtained from
$f$ by deleting its last letter. Hence $q_{n-1}(w)=k-1$. Clearly,
$$
x_{f,b}^{(n)}=\lambda_f(b)-\lambda_w(b).
$$

By Lemma~\ref{helpfullemma}, we may write
$b=y_{j(b),t(b)}+w'$, where $w'$ is an additive word in elements $y_{s,a}$
with $\ell(s)<h(b)=i$. Therefore
$$
\lambda_f(b)=y_{f\circ j(b),t(b)}+\lambda_f(w')
\quad\text{and}\quad
\lambda_w(b)=y_{w\circ j(b),t(b)}+\lambda_w(w').
$$
Using $x_{f,b}^{(n)}=\lambda_f(b)-\lambda_w(b)$, we obtain
$$
y_{f\circ j(b),t(b)}
=
x_{f,b}^{(n)}
+
y_{w\circ j(b),t(b)}
+
\lambda_w(w')
-
\lambda_f(w').
$$

We now check that the correction term belongs to the subgroup generated by the
old basis elements. Since $b\in A_{n-1}$, the definition of $j$ gives
$j(b)\in\Fg_\circ(A_{n-2})$, and $t(b)\in X$. As $q_{n-1}(w)=k-1$, it follows
that $y_{w\circ j(b),t(b)}$ belongs to $Y_{\leq k-1}^{(n)}$.

Moreover, $w'$ is an additive word in elements $y_{s,a}$ with $\ell(s)<i$. For
such elements we have
$\lambda_w(y_{s,a})=y_{w\circ s,a}\in Y_{\leq k-1}^{(n)}$, because
$q_{n-1}(w)=k-1$, and
$\lambda_f(y_{s,a})=y_{f\circ s,a}\in Y_{k,i}^{(n)}$, because
$q_{n-1}(f)=k$ and $\ell(s)<i$. Hence
$y_{w\circ j(b),t(b)}+\lambda_w(w')-\lambda_f(w')$ belongs to the subgroup
generated by
$Y\cup Y_{\leq k-1}^{(n)}\cup Y_{k,i}^{(n)}$.

Thus each element of $Y_{k,i+1}^{(n)}\setminus Y_{k,i}^{(n)}$ is obtained from
the corresponding generator in $X_{k,i}^{(n)}\setminus X_{k,i+1}^{(n)}$ by
adding an additive word in the already existing basis elements. Therefore, by
Nielsen transformations,
$$
Y\cup Y_{\leq k-1}^{(n)}\cup Y_{k,i+1}^{(n)}
\cup X_{k,i+1}^{(n)}\cup X_{\geq k+1}^{(n)}
$$
is again a free basis of $\Fg_+(A_n)$.

This proves, by induction on $i$, that
$$
Y\cup Y_{\leq k-1}^{(n)}\cup Y_{k,i}^{(n)}
\cup X_{k,i}^{(n)}\cup X_{\geq k+1}^{(n)}
$$
is a free basis for every $i$. Since every $b\in A_{n-1}$ has finite
$h(b)$, every generator $x_{f,b}^{(n)}$ with $q_{n-1}(f)=k$ is replaced at a
finite stage, namely at $i=h(b)$. Hence, after passing through all $i$, we get
that $$
Y\cup Y_{\leq k}^{(n)}\cup X_{\geq k+1}^{(n)}
$$ is a free basis. This completes the induction on $k$.

Therefore $Y\cup Y_{\leq k}^{(n)}\cup X_{\geq k+1}^{(n)}$
is a free basis of $\Fg_+(A_n)$ for every $k\geq 1$. Since every reduced word
$f\in \Fg_\circ(A_{n-1})$ contains only finitely many letters from
$X^{(n-1)}\cup\overline{X^{(n-1)}}$, every generator $x_{f,b}^{(n)}$ belongs
to $X_{\geq k+1}^{(n)}$ only for finitely many $k$. Hence, after implementing all the
Nielsen transformations, all elements of $X^{(n)}$ have been replaced, and we
obtain that
$$
Y\cup \bigcup_{k\geq 1}Y_{\leq k}^{(n)}
$$
is a free basis of $\Fg_+(A_n)$.

It remains only to identify this set. By the induction hypothesis,
$$
Y=\{y_{f,a}\mid f\in \Fg_\circ(A_{n-2}),\ a\in X\}.
$$
On the other hand, let $g\in \Fg_\circ(A_{n-1})\setminus \Fg_\circ(A_{n-2})$.
Then the reduced form of $g$ contains at least one letter from
$X^{(n-1)}\cup\overline{X^{(n-1)}}$. Write
$g=f\circ b$, where $f$ is the initial segment of the reduced form of $g$
ending with the last such letter, and $b\in \Fg_\circ(A_{n-2})$ is the
remaining final segment. Then $f\in F_n$, and therefore $y_{g,a}=y_{f\circ b,a}$
belongs to some $Y_{i,k}^{(n)}$. Conversely, every element of every $Y_{i,k}^{(n)}$ is
of the form $y_{g,a}$ with $g\in \Fg_\circ(A_{n-1})$ and $a\in X$.

Thus
$$
Y\cup \bigcup_{k\geq 1}Y_k^{(n)}
=
\{y_{f,a}=\lambda_f(a)\mid f\in \Fg_\circ(A_{n-1}),\ a\in X\}.
$$
Hence this set is a free basis of $\Fg_+(A_n)$, as required.
\end{proof}

\begin{lem}\label{analogforgeneraln}
Let $X$ be a finite non-empty set, let $B=\FSB_{\mathcal{RN}_n,X}$, and let
$J$ be an ideal of $B$ containing $B^{(n)}$. Put $I=J/B^{(n)}$, and let
$$
\theta\colon \Fg_\circ(A_{n-1})\to (B/B^{(n)},\circ)
$$
be the group homomorphism induced by the natural map
$A_{n-1}\to B/B^{(n)}$. Set $H=\theta^{-1}(I)$. Then the additive group of $B/(J*B)$ is free. More precisely, if $R$ is a set
of representatives for the left cosets of $H$ in $\Fg_\circ(A_{n-1})$, then a
free basis is given by
$\{\, y_{g,t}+J*B \mid g\in R,\ t\in X\,\}.$
\end{lem}

\begin{proof}
Let
$\psi\colon \Fg_\circ(A_{n-1})\to (B,\circ)$
be the group homomorphism induced by the inclusion $A_{n-1}\subseteq B$.
Thus $\theta$ is the composition of $\psi$ with the quotient map given by
$(B,\circ)\to (B/B^{(n)},\circ)$.

Recall that $J\ast B$ is an ideal of $B$. Let $K$ be the normal subgroup of
$(B,+)$ generated by the elements
$y_{h\circ f,t}-y_{f,t}$ with $h\in H,\ f\in \Fg_\circ(A_{n-1}),\ t\in X$. We claim that $K=J*B$.

First we show that $K\subseteq J*B$. Let $h\in H$. Since
$H=\theta^{-1}(I)$ and $I=J/B^{(n)}$, there exists $a\in J$ such that the
image of $a$ in $B/B^{(n)}$ is equal to $\theta(h)$. Equivalently, the
elements $a$ and $\psi(h)$ have the same image in $B/B^{(n)}$. Hence they
differ by an element of $B^{(n)}$ in the multiplicative group of $B$. Since $B$ is right nilpotent of class at most $n$, we have $B^{(n)}\ast B=0$.
Thus $B^{(n)}\subseteq\ker(\lambda)$, and therefore
$\lambda_a=\lambda_{\psi(h)}$. Hence, for every
$f\in \Fg_\circ(A_{n-1})$ and $t\in X$,
$$
a*y_{f,t}
=
\lambda_a(y_{f,t})-y_{f,t}
=
\lambda_{\psi(h)}(y_{f,t})-y_{f,t}
=
y_{h\circ f,t}-y_{f,t}.
$$
Since $a\in J$, the left-hand side belongs to $J*B$. Hence every generator of
$K$ belongs to $J*B$, and so $K\subseteq J*B$.

Conversely, we prove that $J*B\subseteq K$. Let $a\in J$. Since $\theta$ is an epimorphism, we may choose
$h\in \Fg_\circ(A_{n-1})$ such that $\theta(h)$ is the image of $a$ in
$B/B^{(n)}$. Then $h\in H$. As above, $a$ and $\psi(h)$ differ by an element
of $B^{(n)}$, and hence $\lambda_a=\lambda_{\psi(h)}$. Therefore, for every
$f\in \Fg_\circ(A_{n-1})$ and $t\in X$,
$$
a*y_{f,t}
=
\lambda_a(y_{f,t})-y_{f,t}
=
y_{h\circ f,t}-y_{f,t}
\in K.
$$

By Lemma~\ref{freebasisy}, the additive group of $B$ is freely generated by
the elements $y_{f,t}$, with $f\in \Fg_\circ(A_{n-1})$ and $t\in X$. Thus
every element $b\in B$ is an additive word in these generators and their
additive inverses. Moreover, for every $u,v\in B$ one has
$$
a*(u+v)=a*u+u+a*v-u\quad\textnormal{and}
\quad
a*(-u)=-u-a*u+u.
$$
Since $K$ is normal in $(B,+)$, these formulas imply by induction on the
length of an additive word that $a*b\in K$ for every $b\in B$. Hence
$J*B\subseteq K$, and therefore $J\ast B=K$.

It remains to find a free basis for the quotient $B/J\ast B$. Since $(B,+)$ is the free group on
$$
\{\,y_{f,t}\mid f\in\Fg_\circ(A_{n-1}),\ t\in X\,\},
$$
and since $J*B=K$ is the normal subgroup generated by the elements
$$
y_{h\circ f,t}-y_{f,t},
\qquad
h\in H,\ f\in \Fg_\circ(A_{n-1}),\ t\in X,
$$ the quotient $(B/J\ast B,+)$ is obtained by identifying generators whose first
indices lie in the same left coset of $H$. Therefore, if $R$ is a set of
representatives for the left cosets of $H$ in $\Fg_\circ(A_{n-1})$, a free
basis is given by
$$
\{\, y_{g,t}+J*B \mid g\in R,\ t\in X\,\}.
$$
This proves the result.
\end{proof}

\begin{thm}\label{classnresfinhopfia}
     Let $X$ be a non-empty set. The skew brace
    $\FSB_{\mathcal{RN}_n,X}$ is residually finite. If, moreover, $X$ is finite, then  $\FSB_{\mathcal{RN}_n,X}$ is Hopfian.
\end{thm}
\begin{proof} 
Let $0\neq b \in B=\FSB_{\mathcal{RN}_n,X}$. To prove that  $B$ is residually finite we need to show that $B$ contains an ideal $I$ such that $b\not\in I$ and $B/I$ is finite.

    First, we remark that it is sufficient to consider only finite sets $X$. Indeed, 
    as~$(B,+)$ is a free group, there is a finite subset~$X_1$ of $X$ 
such that the reduced form of $b$ is contained in the sub-skew brace generated by~$X_1$. Extending the canonical embedding \hbox{$\iota\colon X_1 \rightarrow B_1=\FSB_{\mathcal{RN}_n,X_1}$} to a mapping \hbox{$\iota \colon X\rightarrow B$} such that \hbox{$\iota(X\setminus X_1)=0$,} the universal property guarantees a skew brace epimorphism $\varphi\colon B\rightarrow B_1$ such that $\varphi(b) \neq 0$. Hence, if the result is shown for finite $X$, it follows for arbitrary $X$.

    So, from now on we assume  $X$ is finite. 
   First, if $b \not\in B^{(n)}$, then $b$ is non-zero in the skew brace $B/B^{(n)}$, which is residually finite by induction (see  \cref{quotient-by-last-term}).

Hence, assume $b\in B^{(n)}$. Since, by Proposition~\ref{freebasisy}, the additive
group of $B$ is freely generated by the elements $y_{f,a}$, with
$f\in \Fg_\circ(A_{n-1})$ and $a\in X$, we may write
$$
b=\varepsilon_1 y_{f_1,a_1}+\ldots+\varepsilon_l y_{f_l,a_l},
$$
where $f_i\in \Fg_\circ(A_{n-1})$, $a_i\in X$, and
$\varepsilon_i\in\{1,-1\}$. 
Denote by 
$T$ the finite subset of $B/B^{(n)}$
defined by
$$
T=
\left\{
\theta(f_j\circ f_i^{-1})
\;\middle|\;
a_i=a_j,\ f_i\neq f_j
\right\},
$$ where $$
\theta\colon \Fg_\circ(A_{n-1})\to (B/B^{(n)},\circ)
$$
is the group homomorphism induced by the natural map
$A_{n-1}\to B/B^{(n)}$. 
Since $B/B^{(n)}\simeq \FSB_{\mathcal{RN}_{n-1},X}$ is
residually finite by induction, we can choose an ideal $J$ of $B$ containing
$B^{(n)}$ such that $J/B^{(n)}$ has finite index in $B/B^{(n)}$ and
$$
T\cap J/B^{(n)}=\emptyset.
$$
Consider $J\ast B$, which is an ideal of $B$. 
Then, the image of $b$ in $B/(J \ast B)$ is non-zero. Indeed, by \cref{analogforgeneraln}, the additive group of $B/(J*B)$ is freely
generated by
$$
\{\, y_{g,t}+J*B \mid g\in R,\ t\in X\,\},
$$
where $R$ is a set of representatives for the left cosets of
$H=\theta^{-1}(J/B^{(n)})$ in $\Fg_\circ(A_{n-1})$. The condition
$T\cap J/B^{(n)}=\emptyset$ says precisely that, whenever $a_i=a_j$ and
$f_i\neq f_j$, the elements $f_i$ and $f_j$ do not belong to the same left
coset of $H$. Hence the distinct basis elements appearing in the reduced
expression of $b$ remain distinct in $B/(J*B)$. Therefore the image of $b$ is
non-zero.

Note that in $B/(J\ast B)$ one has that
$$
J + B^{(n)} \subseteq \ker(\lambda).
$$
As the former is of finite index, the latter is of finite index. 
Let $\overline B=B/(J*B)$, and denote by $\overline b$ the image of $b$ in
$\overline B$. Since $\overline b\neq 0$ and $(\overline B,+)$ is a free group
of finite rank by \cref{analogforgeneraln}, there exists a subgroup $M$ of
finite index in $(\overline B,+)$ such that $\overline b\notin M$. Put
$$
Q=\ker(\lambda)\cap M.
$$
Since both $\ker(\lambda)$ and $M$ have finite index in $(\overline B,+)$,
also $Q$ has finite index in~$(\overline B,+)$. Moreover,
$Q\subseteq\ker(\lambda)$ and $\overline b\notin Q$. As $(\overline B,+)$ is finitely generated, the subgroup $Q$ contains a characteristic subgroup $Q_2$ of finite index in $(\overline B,+)$. In particular, $Q_2\subseteq Q\subseteq\ker(\lambda)$ and $\overline b\notin Q_2$. Since $Q_2$ is characteristic in the additive group, it is normal in
$(\overline B,+)$ and invariant under all automorphisms of $(\overline B,+)$;
in particular, it is invariant under all maps $\lambda_x$, with
$x\in\overline B$. As $Q_2\subseteq\ker(\lambda)$, it follows that $Q_2$ is an
ideal of $\overline B$. Moreover, $Q_2$ has finite index and does not contain
$\overline b$. Hence $\overline B/Q_2$ is a finite skew brace in which the
image of $b$ is non-zero. This proves that $B$ is residually finite.

    \smallskip
    
    The second part of the result follows at once from \cref{hopfian}.
\end{proof}

\begin{thm}\label{notcohop}
The free skew brace $\FSB_{\mathcal{RN}_{n}, X}$ is not co-Hopfian.
\end{thm}
\begin{proof} 
Put $B=\FSB_{\mathcal{RN}_n,X}$. 
Let $T$ be the subset of $B$ consisting of the elements having even additive length in the  basis
\[
\left\lbrace y_{f,a}=\lambda_f(a) \mid f \in F_{\circ}(A_{n-1}),\, a \in X \right\rbrace
\]
of $(B,+)$ (see \cref{freebasisy}). Clearly, this is a proper subset of $B$.
We claim that $T$ is a sub-skew brace. To prove this we only have to show that it is a multiplicative subgroup.
This follows  from the fact that every $\lambda_a$, with $a\in B$,  maps an element of even length to an element of even length, and thus for  $w,s\in T$ we get that that $ w \circ s = w+\lambda_w(s)\in T$.
Similarly, because $\overline{w}=-\lambda_{\overline{w}}(w)$ we also get that $\overline{w}\in T$. Hence the claim follows.
As a sub-skew brace  of $B$, the skew brace $T$ is  right nilpotent of class at most $n$. Hence, there exists a skew brace morphism $\varphi\colon B \rightarrow T$ with $\varphi(x)=2x$ for $x \in X$.
We claim that $\varphi$  is injective. We prove this by induction on $n$. It is clear for $n=1$. Hence, we assume that the proposed map is injective for the free right nilpotent skew brace of class at most $n-1$. 
Note that $\varphi(y_{f,a})=2y_{\varphi(f),a}$, with $f \in F_{\circ}(A_{n-1})$ and $ a \in X$. Hence, as $\varphi$ is injective on $B/B^{(n)}=\FSB_{\mathcal{RN}_{n-1},X}$ (see~\cref{quotient-by-last-term}), the indices $\varphi(f)$ are distinct for all $f \in \Fg_{\circ}(A_{n-1})$. By \cref{freebasisy}, the $y_{f,a}$ form a free basis of $(B,+)$, which shows that $\varphi$ is injective. As $T$ is a proper sub-skew brace of $B$, we thus have shown that~$B$ is not co-Hopfian.
\end{proof}

\subsection{One generated skew braces}

Here we focus on a very specific case of the construction in \cref{sec: right nilp}: the free right nilpotent skew brace of class at most $2$ on a set with only one element $X=\{x_0\}$. \cref{skew braces and regular subgroups} allows us to compute the corresponding regular subgroup and to analyse the induced $\lambda$-action, from which we obtain a description of the additive commutator subgroup and the non-Hopficity of such objects. Another relevant consequence is the non-injectivity of the free solution $\FSol_{\mathcal{RN}_2,\{x_0\}}$.

\medskip

Let $X=\{x_0\}$ be a set with one element and thus let $\FSB_{\mathcal{RN}_{2},\{x_0\}}$ be the free right nilpotent skew brace of class $2$ on $X$.
Specializing the construction in~\cref{sec: right nilp} to this case,
we obtain $\FSB_{\mathcal{RN}_{2},\{x_0\}}=(\Fg_{+}(A_2),+,\circ)$,
where $A_2=\{x_0\}\cup\left\{x^{(2)}_{x_0^i,x_0}\mid i\in\Z\setminus\{0\}\right\}$
and the $\lambda$-action is given by $\lambda_{x^{(2)}_{x_0^i,x_0}}=\id$ for every $i\in\Z\setminus\{0\}$ and
\begin{align*}
    \lambda_{x_0}:\qquad x_0 &\longmapsto x^{(2)}_{x_0,x_0}+x_0\\
    x^{(2)}_{x_0^i,x_0}&\longmapsto \begin{cases}
    -x^{(2)}_{x_0,x_0} &\text{ if }i=-1\\
    x^{(2)}_{x_0^{i+1},x_0} &\text{ if }i\neq-1,0
\end{cases}
\end{align*}

\begin{lem}\label{isomorpshimsmth}
Let $Z=\{x_i\mid i\in \Z\}$ be a set indexed by $\Z$. 
The additive group~$\Fg_+(A_2)$ of $\FSB_{\mathcal{RN}_{2},\{x_0\}}$ is isomorphic to $\Fg_+(Z)$.
\end{lem}
\begin{proof}
Define two homomorphisms
\[
\varphi : \Fg_+(Z) \to \Fg_+(A_2),
\quad
\psi : \Fg_+(A_2) \to \Fg_+(Z),
\]
and show that they are inverse to each other. Since $\Fg_+(Z)$ is the free group on $Z$, it suffices to define $\varphi$ on generators:
\[
\varphi(x_i) = \begin{cases}
x^{(2)}_{x_0^i,x_0} + x^{(2)}_{x_0^{i-1},x_0} + \ldots + x^{(2)}_{x_0,x_0} + x_0 & \text{if } i>0,\\
x_0&\text{if } i=0,\\
x^{(2)}_{x_0^i,x_0} + x^{(2)}_{x_0^{i+1},x_0} + \ldots + x^{(2)}_{x_0^{-1},x_0} + x_0 & \text{if } i<0.
\end{cases}
\in \Fg_+(A_2),
\]
This defines a homomorphism
$\varphi : \Fg_+(Z)\to \F_+(A_2)$.

We define $\psi$ on the generators $A_2$ of $\Fg_+(A_2)$ by $\psi(x_0) = x_0$
and for $i\neq 0$,
\[
\psi\left(x^{(2)}_{x_0^i,x_0}\right) =
\begin{cases}
x_i - x_{i-1} & \text{if } i>0,\\[6pt]
x_i - x_{i+1} & \text{if } i<0.
\end{cases}
\]
This extends uniquely to a homomorphism $\psi : \Fg_+(A_2)\to \Fg_+(Z)$.

For $i>0$, we compute:
\begin{align*}
\psi(\varphi(x_i))
&= \psi\big(x^{(2)}_{x_0^i,x_0} + x^{(2)}_{x_0^{i-1},x_0} + \ldots + x^{(2)}_{x_0,x_0} + x_0\big)\\
&= (x_i - x_{i-1}) + (x_{i-1} - x_{i-2}) + \ldots + (x_1 - x_0) + x_0\\
&= x_i
\end{align*}
and
\begin{align*}
\varphi\left(\psi\left(x^{(2)}_{x_0^i,x_0}\right)\right)
&= \varphi(x_i - x_{i-1})= \varphi(x_i) - \varphi(x_{i-1})= x^{(2)}_{x_0^i,x_0}.
\end{align*}

Similarly, for $i<0$,
\begin{align*}
\psi(\varphi(x_i))
&= \psi\big(x^{(2)}_{x_0^i,x_0} + x^{(2)}_{x_0^{i+1},x_0} + \ldots + x^{(2)}_{x_0^{-1},x_0} + x_0\big)\\
&= (x_i - x_{i+1}) + (x_{i+1} - x_{i+2}) + \ldots + (x_{-1} - x_0) + x_0\\
&= x_i
\end{align*}
and
\begin{align*}
\varphi\left(\psi\left(x^{(2)}_{x_0^i,x_0}\right)\right)
&= \varphi(x_i - x_{i+1})= \varphi(x_i) - \varphi(x_{i+1})= x^{(2)}_{x_0^i,x_0}.
\end{align*}

Hence indeed
$
\psi \circ \varphi = \mathrm{id}_{\Fg_+(Z)}\quad \text{and} \quad \varphi \circ \psi = \mathrm{id}_{\Fg_+(A_2)}
$.
Therefore, the maps $\varphi$ and $\psi$
are mutually inverse isomorphisms, hence $\Fg_+(A_2) \simeq \Fg_+(Z)$.
\end{proof}

\begin{thm}\label{descriptionregular}
Let $Z=\{x_i\mid i\in \Z\}$ be a set indexed by $\Z$. The free skew brace $\FSB_{\mathcal{RN}_{2},\{x_0\}}$ corresponds with the regular subgroup
    \[
 \left\lbrace \left(w,\theta^{\varepsilon(w)}\right) \mid w \in \Fg_+(Z) \right\rbrace\leq \Fg_+(Z) \rtimes \textup{Aut}(\Fg_+(Z)),
\] 
where $\varepsilon: \Fg_+(Z) \rightarrow \mathbb{Z}$ 
is the group epimorphism defined by $\varepsilon(x_i)=1$ and $\theta$ is  the group automorphism of $\Fg_+(Z)$ defined by $\theta(x_i)=x_{i+1}$. 
\end{thm}
\begin{proof}
To compute the regular subgroup corresponding to $\FSB_{\mathcal{RN}_{2},\{x_0\}}$,
we need to check how the $\lambda$-action translates on $\Fg_+(Z)$. Let \[
\varphi : \Fg_+(Z) \to \Fg_+(A_2)
\quad\textnormal{and}\quad
\psi : \Fg_+(A_2) \to \Fg_+(Z),
\] be the mappings from the proof of \cref{isomorpshimsmth}. We claim that    $\lambda_{\varphi(w)}=\varphi\circ\theta^{\varepsilon(w)}\circ \psi ,$
for every $w\in \Fg_+(Z)$.
We first start with $w=x_0$.
By definition, $\lambda_{x_0}(x_0)=x^{(2)}_{x_0,x_0}+x_0$ and
\[
\lambda_{x_0}\left(x^{(2)}_{x_0^i,x_0}\right)
=\begin{cases}
    -x^{(2)}_{x_0,x_0} &\text{ if }i=-1\\
    x^{(2)}_{x_0^{i+1},x_0} &\text{ if }i\neq-1,0
\end{cases}
\]
Hence 
\[
\lambda_{\varphi(x_0)}(x_0)
=x^{(2)}_{x_0,x_0}+x_0
=\varphi(x_1)\\
=\varphi(\theta(x_0))
=\varphi(\theta(\psi(x_0)).
\]
and
\begin{align*}
\lambda_{\varphi(x_0)}\left(x^{(2)}_{x_0^{-1},x_0}\right)
&=-x^{(2)}_{x_0,x_0}
=\varphi(x_0-x_1)
=\varphi(\theta(x_{-1}-x_0))\\
&=\varphi\left(\theta\left(\psi\left(x^{(2)}_{x_0^{-1},x_0}\right)\right)\right).
\end{align*}
Moreover, if $i>0$
\begin{align*}
\lambda_{\varphi(x_0)}\left(x^{(2)}_{x_0^i,x_0}\right)
&=x^{(2)}_{x_0^{i+1},x_0}\\
&=\varphi(x_{i+1}-x_i)\\
&=\varphi(\theta(x_i-x_{i-1}))\\
&=\varphi\left(\theta\left(\psi\left(x^{(2)}_{x_0^i,x_0}\right)\right)\right)
\end{align*}
while if $i<-1$
\begin{align*}
\lambda_{\varphi(x_0)}\left(x^{(2)}_{x_0^i,x_0}\right)
&=x^{(2)}_{x_0^{i+1},x_0}\\
&=\varphi(x_{i+1}-x_{i+2})\\
&=\varphi(\theta(x_i-x_{i+1}))\\
&=\varphi\left(\theta\left(\psi\left(x^{(2)}_{x_0^i,x_0}\right)\right)\right).
\end{align*}
So $\lambda_{\varphi(x_0)}=\varphi\circ\theta^{\varepsilon(x_0)}\circ \psi$, and the claim is proved.

\smallskip

Now, we claim that 
 $\varphi(x_{i+1})=\lambda_{x_0}(\varphi(x_i))$
for every $i\in \Z$. 
To prove this, we  distinguish the cases $i\geq 0$ and $i<0$
and proceed by induction.
We have already noted that $\varphi(x_1)=\lambda_{x_0}(x_0)$.
Assuming the claim holds  for $i-1\geq 0$, i.e.
$\varphi(x_i)=\lambda_{x_0}(\varphi(x_{i-1}))$, then
\begin{align*}
   \varphi(x_{i+1})
&=x^{(2)}_{x_0^{i+1},x_0}+\varphi(x_i)\\
&=x^{(2)}_{x_0^{i+1},x_0}+\lambda_{x_0}(\varphi(x_{i-1}))\\
&=\lambda_{x_0}\left(x^{(2)}_{x_0^i,x_0}+\varphi(x_{i-1})\right)\\
&=\lambda_{x_0}(\varphi(x_i)).
\end{align*}

On the other hand, if $i=-1$, then
\[
\varphi(x_0)=x_0
=-x^{(2)}_{x_0,x_0}+x^{(2)}_{x_0,x_0}+x_0
=\lambda_{x_0}\left(x^{(2)}_{x_0^{-1},x_0}+x_0\right)
=\lambda_{x_0}(\varphi(x_{-1})).
\]
Assuming the claim holds  for $i<0$, i.e.
$\varphi(x_{i+1})=\lambda_{x_0}(\varphi(x_i))$, then
\begin{align*}
   \varphi(x_i)
&=x^{(2)}_{x_0^i,x_0}+\varphi(x_{i+1})\\
&=x^{(2)}_{x_0^i,x_0}+\lambda_{x_0}(\varphi(x_i))\\
&=\lambda_{x_0}\left(x^{(2)}_{x_0^{i-1},x_0}+\varphi(x_i)\right)\\
&=\lambda_{x_0}(\varphi(x_{i-1})).
\end{align*}
Therefore, 
\(\varphi(x_{i+1})=\lambda_{x_0}(\varphi(x_i))\) for every $i\in \Z$, and the claim is proved.

\smallskip

Since $\FSB_{\mathcal{RN}_{2},\{x_0\}}$ is a skew brace of right nilpotency class $2$,
by \cref{lambda action for rclass 2},
\[
\lambda_{\varphi(x_i)}=\lambda_{\varphi(x_0)},
\]
for every $i\in \Z$. So
\[
\lambda_{\varphi(x_i)}=\varphi\circ\theta^{\varepsilon(x_i)}\circ \psi
\]
for every $i\in \Z$.
Finally, combining this result with the second part of \cref{lambda action for rclass 2}, we have that
\[
\lambda_{\varphi(w)}=\varphi\circ\theta^{\varepsilon(w)}\circ \psi,
\]
for every $w\in \Fg_+(Z)$.
Hence, the translation of the $\lambda$-action on $\Fg_+(Z)$ of a word $w\in \Fg_+(Z)$ is given by
\[
\psi\circ\lambda_{\varphi(w)}\circ\varphi
=\theta^{\varepsilon(w)}
\]
and it follows that the regular subgroup corresponding to $\FSB_{\mathcal{RN}_{2},\{x_0\}}$ is
\[
\left\{\left(w, \theta^{\varepsilon(w)}\right)\mid w\in \Fg_+(Z)\right\},
\]
as desired.
\end{proof}

\begin{cor} 
The free solution right nilpotent class $2$  on $\{ x_0\}$ 
is \[
\widetilde{\{x_0\}}_{\mathcal{RN}_{2}}=\{a+x_i-a\mid i\in\Z, a\in G\left(\FSol_{\mathcal{E},X},r_{\FSol_{\mathcal{E},X}}\right)\}
\]
\end{cor}
\begin{proof}
This follows at once from Theorem~\ref{descriptionregular} 
and Corollary~\ref{freerightnilpotentsoldescriptinj}
\end{proof}

As an application of \cref{descriptionregular}, we obtain the following result.

\begin{thm}
The solution $\FSol_{\mathcal{RN}_2,\{ x_0\}}$ is not injective.
\end{thm}
\begin{proof} 
Set $S=\FSol_{\mathcal{RN}_2,\{x_0\}}$. Suppose, by contradiction, that \(S\) is injective. By~The\-o\-rem~\ref{StrGrpFSBn}, the
structure skew brace of \(S\) is the free skew brace $\FSB_{\mathcal{RN}_2,\{x_0\}}$. Since \(S\) is injective, we may identify \(S\) with its image inside its
structure skew brace.
Using right nilpotency class $2$,
    \begin{align*}
    \rho_b(a)
    &=\overline{\lambda_a(b)}\circ a-\overline{\lambda_a(b)}\\
    &=\overline{\lambda_a(b)}+\lambda_{\overline{\lambda_a(b)}}(a)-\overline{\lambda_a(b)}\\
    &=\overline{\lambda_a(b)}+\lambda_{\overline{b}}(a)-\overline{\lambda_a(b)},
    \end{align*}
    for every $a,b\in \FSB_{RN_2,\{x_0\}}$.

In particular, for $a=x_i$ and $b=x_j$, 
where $x_k=\lambda_{x_0}^k(x_0)$ for every $k\in \Z$,
we get
\[
\rho_{x_j}(x_i)
=
\overline{\lambda_{x_i}(x_j)}
+
\lambda_{\overline{x_j}}(x_i)
-
\overline{\lambda_{x_i}(x_j)}.
\]
Since $\lambda_{x_i}=\theta$, we have
$\lambda_{x_i}(x_j)=x_{j+1}$.
Moreover, $\overline{x_{j+1}}=-x_j$ and $\lambda_{\overline{x_j}}=\theta^{-1}$.

Thus $\lambda_{\overline{x_j}}(x_i)=x_{i-1}$
and hence
\[
\rho_{x_j}(x_i)
=
-x_j+x_{i-1}+x_j.
\]
In particular, taking $j=i-1$, we obtain
\[
\rho_{x_{i-1}}(x_i)
=
-x_{i-1}+x_{i-1}+x_{i-1}
=
x_{i-1}.
\]
So, for $i=1$, $\rho_{x_0}(x_1)=x_0$.

Now consider the solution $Y=\{0,1\}$ defined by
\[
r(a,b)=(\sigma(b),a),
\]
where $\sigma=(0\,1)$. Then $\lambda_a=\sigma$ and $\rho_b=\id$
for all $a,b\in Y$. Hence
\[
\lambda_{\lambda_y(x)}=\lambda_x
\]
for all $x,y\in Y$, and therefore $(Y,r)\in\mathcal{RNS}_2$.

Since $S$ is free on $x_0$ in $\mathcal{RNS}_2$, there exists a morphism of
solutions
\[
\varphi:S\to Y
\]
such that $\varphi(x_0)=0$.
Since $x_1=\lambda_{x_0}(x_0)$,
we obtain
\[
\varphi(x_1)
=
\lambda_{\varphi(x_0)}(\varphi(x_0))
=
\lambda_0(0)
=
\sigma(0)
=
1.
\]
On the other hand, applying $\varphi$ to the equality $\rho_{x_0}(x_1)=x_0$
gives
\[
\rho_{\varphi(x_0)}(\varphi(x_1))
=
\varphi(x_0),
\] so $\rho_0(1)=0$.
But $\rho_0=\id$, and hence $\rho_0(1)=1$, a contradiction. Therefore, $S$ cannot be injective.
\end{proof}

\begin{cor}
\label{solnotinjective}
For any non-empty set $X$, the solution 
 $\FSol_{\mathcal{RN}_2,X}$ is not injective.
\end{cor}
\begin{proof}
Since there are maps $f:\{x_0\}\to X$ and $g:X\to\{x_0\}$
such that $g\circ f=\id_{\{x_0\}}$, 
the map $\tilde{f}:\FSol_{\mathcal{RN}_2,\{x_0\}}\to \FSol_{\mathcal{RN}_2,X}$ induced by $f$ is injective.
Therefore, if $\FSol_{\mathcal{RN}_2,X}$ were injective then 
also $\FSol_{\mathcal{RN}_2,\{x_0\}}$ would be injective.
Hence, by \cref{solnotinjective}, $\FSol_{\mathcal{RN}_2,X}$
cannot be injective.   
\end{proof}

\medskip

We end this section with a remark highlighting a structural difference between the free objects $\FSB_{\mathcal{RN}_2,\{ x_0\}}$ and $\FB_{\mathcal{RN}_{2},\{ x_0\}}$.

\begin{rem}\label{remarkrn22}
In general, the ideal generated by the additive commutator is not easy to describe explicitly.  However, in the particular case of one-generated right nilpotency
class $2$, the structure of this ideal can be
determined more precisely, which allows us to obtain a more explicit description of the corresponding free object.

 Set $B=\FSB_{\mathcal{RN}_{2},\{x_0\}}$. Then $[B,B]_+$ is an ideal of $B$ because it is a characteristic subgroup of $(B,+)$ which is included in $\ker(\lambda)$ (in fact, $\varepsilon(c)=0$ for every $c\in [B,B]_+$). Moreover, the free injective solution 
on $\{x_0\}$ in $\mathcal{RNS}_2$ is
\hbox{$\widetilde{X}=\{[x_i]\mid i\in \Z\}$} with
$ r_{B/C}([x_i],[x_j])=([x_{j+1}],[x_{i-1}])$.

 Finally, since  $\FSB_{\mathcal{RN}_{2},\{x_0\}}/\FSB_{\mathcal{RN}_{2},\{x_0\}}\ast \FSB_{\mathcal{RN}_{2},\{x_0\}}$ is a trivial brace on a cyclic group,
     \[
     [\FSB_{\mathcal{RN}_{2},\{x_0\}},\FSB_{\mathcal{RN}_{2},\{x_0\}}]_+\subseteq \FSB_{\mathcal{RN}_{2},\{x_0\}}\ast\FSB_{\mathcal{RN}_{2},\{x_0\}}.
     \]
     But, $(B/C,+)$ is free abelian of infinite rank, so the two ideals are not equal. This shows a difference between a free right nilpotent skew brace $\FSB_{\mathcal{RN}_2,\{ x_0\}}$ of class~$2$ and a free right nilpotent skew brace $\FB_{\mathcal{RN}_{2},\{ x_0\}}$ of abelian type and of class~$2$. The former has isomorphic additive and multiplicative groups (and these are free groups of infinite rank) and the latter has an additive free abelian group, while the multiplicative group is not abelian (because there exist skew braces of right nilpotency class $2$ whose additive group is infinite cyclic, while the multiplicative group is infinite dihedral).
\end{rem}

\section{Two-sided skew braces}\label{twosided}
This section is devoted to the most tractable class of skew braces, namely the category consisting of two-sided skew braces, denoted $\mathcal{T}$, with particular emphasis on two-sided skew braces of abelian type. That is, we require both the left and right skew distributivity \eqref{E:LeftSkewDistr} and \eqref{E:RightSkewDistr}, and often the commutativity of the operation $+$.  Recall that if $(B,+,\circ)$ is a two-sided skew brace of abelian type, then  $(B,+,\ast)$ is a (non-unital) \textit{radical ring} (and vice-versa); in particular, the operation $\ast$ is associative (which for a skew brace of abelian type is equivalent to two-sidedness, as shown in \cite{Lau}) and distributes over $+$. Such structures were thoroughly studied by ring- and skew brace-theorists;  for instance, see \cite{KN1,KN2, Sysak12, Smok18, TRAPPENIERS2023267}.

In this setting, the operation $*$ is compatible with the $\lambda$-actions in a particularly nice way.
\begin{lem}
\label{L:lamda_vs_*} 
In a two-sided skew brace of abelian type $(B,+,\circ)$, the following relations hold for all $a,b,c \in B$:
\[\lambda_a(b*c)= \lambda_a(b)*c.\]
\end{lem}
\begin{proof}
Using the properties of the ring $(B,+,*)$, one has
\begin{align*}
\lambda_a(b*c) &= a*(b*c)+b*c =(a*b)*c+b*c,\\
\lambda_a(b)*c &= (a*b+b)*c =(a*b)*c+b*c. \qedhere
\end{align*}
\end{proof}

Clearly, every skew brace with an abelian multiplicative group is two-sided, but there exist two-sided skew braces with a non-abelian multiplicative group.

\begin{exa}
\label{two-sided two-generated non-abelian}
    There exist two-generated two-sided skew braces of abelian type with a non-abelian multiplicative group.
    An example is the skew brace $B =(\Z_2\times\Z_2\times\Z_2, +, \circ)$ with the usual componentwise addition and  multiplication
    \[
    \begin{pmatrix}
        x_1\\y_1\\z_1
    \end{pmatrix}\circ \begin{pmatrix}
        x_2\\y_2\\z_2
    \end{pmatrix}=\begin{pmatrix}
        x_1+x_2+z_1y_2\\y_1+y_2\\z_1+z_2.
    \end{pmatrix}
    \]
    It is generated by the non-commuting elements $\begin{pmatrix} 0\\1\\0 \end{pmatrix}$ and $\begin{pmatrix} 0\\0\\1 \end{pmatrix}$.
    Note that $(B,+,\ast)$ is the ring with product   $\begin{pmatrix}
        x_1\\y_1\\z_1
    \end{pmatrix} \ast   \begin{pmatrix}
        x_2\\y_2\\z_2
    \end{pmatrix} =  \begin{pmatrix}
        z_1y_2\\0\\0
    \end{pmatrix}$. So $B$ is a nilpotent ring of nilpotency index $3$.

More generally, for any $n>1$, the skew brace $B_n=B\times \Triv\left(\Z_2^{n-2}\right)$ yields an example of a two-sided brace generated by $n$ elements and with a non-abelian multiplicative group.

\medskip

	Another interesting example of a two-generated two-sided skew brace of abelian type with a non-abelian multiplicative group is the ideal generated by $\{x,y\}$ in the ring $\Z_2\{x,y\}/\langle x,y\rangle^3$. Clearly, this ideal is the Jacobson radical of the ring, and it is nilpotent of index 3. In this skew brace, the $\circ$-inverse is given by the formula $\overline{a} = a+a^2$. Its additive group is~$\Z_2^6$, with as additive generators the $2$ monomials of degree $1$ and the $4$ monomials of degree $2$.
\end{exa}

\medskip

The abelianity requirement of the multiplicative group in fact is a very strong one:
it entails, for example, that the additive group is metabelian (see \cite{NASYBULLOV2019156}, Corollary~4.7). 
Our first result reduces the problem of verifying the commutativity of the multiplicative group of a two-sided skew brace to the generators, a fact that has some interesting consequences for a free two-sided skew brace.

\begin{pro}\label{two sided brace abelian generated is abelian }
Let $(B,+,\circ)$ be a two-sided skew brace admitting a generating subset $\emptyset\neq A\subseteq B$ whose elements pairwise $\circ$-commute. Then the group $(B,\circ)$ is abelian.
\end{pro}
\begin{proof}
    Consider the following filtration of $B$. Put $W_0=\{0\}\cup A$ and, for all $n\geq 0$, 
    \[W_{n+1}=-W_n\cup \overline{W_n}\cup (W_n+W_n)\cup (W_n\circ W_n).
    \]
    Since $0\in W_n$ for all $n\in \N$, we have $W_n\subseteq W_{n+1}$. In addition, the fact that the set~$A$ is a generating set  for the skew brace $B$ can be translated as $\bigcup_{n\in\N}W_n = B$. Thus, we obtain an increasing filtration for $B$.

    We now prove, by induction on $n \in\N \cup \{0\}$, that all elements of $W_n$ $\circ$-commute.
    
    For $n=0$, the claim follows by hypothesis. We will next show that, if some $\alpha\in B$ $\circ$-commutes with every $\beta \in W_n$, then it also $\circ$-commutes with every $w\in W_{n+1}$.
    \begin{itemize}
    \item     If $w\in -W_n$,
    then $w=-\beta$ for some $\beta\in W_n$. Since $\alpha$ $\circ$-commutes with $\beta$,
    \[
    \alpha\circ(-\beta)=
    \alpha-\alpha\circ\beta+\alpha=
    \alpha-\beta\circ\alpha+\alpha=
    (-\beta) \circ \alpha.
    \]
    \item If $w\in \overline{W_n}$,
    then $w=\overline{\beta}$ for some $\beta\in W_n$.
    Since $\alpha$ $\circ$-commutes with $\beta$, it also $\circ$-commutes 
    with its $\circ$-inverse $\overline{\beta}$.
    \item If $w\in W_n+W_n$, then $w=\beta_1+\beta_2$
    for some $\beta_1,\beta_2\in W_n$. Since $\alpha$ \hbox{$\circ$-com}\-mutes with $\beta_1$ and $\beta_2$,
    \[
    \alpha\circ(\beta_1+\beta_2)=
    \alpha\circ\beta_1-\alpha+\alpha\circ\beta_2=
    \beta_1\circ\alpha-\alpha+\beta_2\circ\alpha=
    (\beta_1+\beta_2)\circ\alpha.
    \]
    \item Finally, if $w\in W_n\circ W_n$, then $w=\beta_1\circ\beta_2$
    for some $\beta_1,\beta_2\in W_n$.
    Since $\alpha$ $\circ$-commutes with $\beta_1$ and $\beta_2$,
    it also $\circ$-commutes with $\beta_1\circ\beta_2$.
    \end{itemize}

    Assume now the induction claim for $n=k$. That is, the elements of $W_k$ \hbox{$\circ$-commutes} with themselves. What we have already proved shows that every element of $W_k$  $\circ$-commutes with every element of $W_{k+1}$, so a further application of the same fact yields that the elements of $W_{k+1}$ $\circ$-commute with themselves. This terminates the inductive proof and proves the statement.
\end{proof}

\begin{cor}\label{two-sided and abelianity}
Every one-generated two-sided skew brace has an abelian multiplicative group, while  the free two-sided skew brace of abelian type  on a set $X$ with more than one element, i.e. $\FB_{\mathcal{T},X}$, has a non-abelian multiplicative group. 
\end{cor}
\begin{proof}
The first part is a special case of \cref{two sided brace abelian generated is abelian }. On the other hand, any of the skew braces in \cref{two-sided two-generated non-abelian} is a quotient of the free two-sided skew brace on $n$ elements. Since their multiplicative groups are non-abelian, neither is the multiplicative group of the free two-sided skew brace of abelian type on $n$ elements.
\end{proof}

Next, we provide an explicit construction of the free two-sided skew brace of abelian type on a finite non-empty set $X$ (see Theorem~\ref{constructiontwosidedbrace}). (Note that the condition that  $X$ is finite is used in the proof of Theorem~\ref{thm:cohn}, especially in Section~3 of ~\cite{MR1151322}). Using this construction, we will show that such a free object is residually finite (see Corollary~\ref{cortwosidedbrace}). To this end, we first recall some definitions and results concerning the Cohn localisation of the free group ring, which can be found in~\cite{MR1151322}.

\begin{defn}
Let $R$ be a ring and $\Sigma$ a set of invertible square matrices with coefficients in $R$ (note that distinct elements of $\Sigma$ may have different orders). We say that a ring morphism  $f\colon R\to S$ is {\it $\Sigma$-invertible} if every matrix  $A\in \Sigma \cap M_{n}(R)$ is mapped, by the natural extension $M_{n}(R)\to M_n(S)$ of $f$, to an invertible matrix over $S$.
\end{defn}
\begin{defn}
    Let $R$ be a ring and $\Sigma$ a set of invertible square matrices over $R$. A $\Sigma$-invertible morphism $R\to R_\Sigma$ is said to be {\it universal} if for every $\Sigma$-invertible morphism $R\to S$ there exists a unique morphism $R_\Sigma\to S$ such that the following diagram commutes
    \[
    \begin{tikzcd}[column sep=small]
& R \arrow[dl] \arrow[dr] & \\
R_\Sigma \arrow[rr, dashed] & & S.
\end{tikzcd}
    \]
If such a $\Sigma$-invertible morphism exists, then the ring $R_\Sigma$ is unique and is called the~{\it Cohn localisation} of $R$. 
\end{defn}

Let $X$ be a non-empty finite set.
Consider the ring of non-abelian power series
$\Z\{\{X\}\}$. Let $M$ be the set of monomials (including $1$) in the variables $X$. For every $m\in M$ we define a $\Z$-linear mapping $$\dfrac{d}{dm}: \Z\{\{X\}\}  \to \Z\{\{X\}\}$$ as follows: If $w$ is a monomial, then \[
\dfrac{d}{dm}:w\mapsto\begin{cases} u &\text{ if } w=mu\\
0 &\text{ otherwise.}
\end{cases}
\]
For $f\in\Z\{\{X\}\}$, we say that $f$ is \emph{rational} if 
$\Z\left[\dfrac{d}{dm}\mid m\in M\right]f$ is a finitely generated $\Z$-module. The set of rational functions on $X$ is in fact a ring, we denote it by $\mathscr{R} (X)$, or simply by $\mathscr{R}$ if $X$ is clear from the context.

We embed the integral group ring $\Z[\Fg (X)]$ of the free group $\Fg (X)$ in $\Z\{\{X\}\}$ via the Magnus embedding:
\begin{align*}
    x&\mapsto 1+x,\\
    x^{-1}\mapsto 1&-x+x^2-x^3\dots
\end{align*}
for all $x\in X$. The image of the embedding is contained in $\mathscr{R}$, we denote it by $\Lambda (X)$, or simply by $\Lambda$ if $X$ is clear from the context.  Let 
     $$\epsilon\colon \Z\{\{X\}\}\to \Z$$ be the morphism mapping a power series to its constant coefficient. From now on, we set 
 $$\Sigma \text{ to be the set of square matrices }  (\lambda_{ij}) \text{ of any order with entries in } \Lambda $$
 $$\text{ and such that } \det((\epsilon(\lambda_{ij}))\in \{1,-1\},$$

\begin{thm}
[see \cite{MR1151322}, Theorem 5.1]
\label{thm:cohn}   The ring of rational functions $\mathscr{R} (X)$ on a finite non-empty set $X$ is the Cohn localisation of $\Lambda (X)$ with respect to $\Sigma (X)$.
\end{thm}

Our goal is to show that the free two-sided skew brace of abelian type (or free radical ring) generated by $X$ is $J(\mathscr{R})$, the Jacobson radical of $\mathscr{R}$. We recall the basic definitions of the theory of radical rings we need; more can be found in~\cite{MR1838439} and~\cite{MR393090}.

\begin{defn}
    Let $R$ be a ring. An element $u\in R$ is called {\it left quasiregular} if there exists $r\in R$ such that $r\circ u=r+ru+u=0$. {\it Right quasiregular} elements are defined symmetrically. An element that is both left and right quasiregular will be called {\it quasiregular}.
    
    The {\it Jacobson radical} $J(R)$ of $R$ is the set of elements $r\in R$ such that $sr$ is left quasiregular for all $s\in R$. $J(R)$ is a two-sided ideal. Equivalently, if $R$ is unital,  $J(R)$ is the set of elements $r\in R$ such that \hbox{$1+RrR\subset R^\times$,} where $R^\times$ is the set of invertible elements of $R$.

    The ring $R$ is {\it radical} if $J(R)=R$. Equivalently, $R$ is radical if all its elements are quasiregular.
\end{defn}

\begin{defn}
    Let $R$ be a ring, the {\it Dorroh extension} $D(R)$ of $R$ is the unital ring $\Z\times R$ with the canonical sum and multiplication defined by $(x,r)(y,s)=(xy,yr+xs+rs)$.
\end{defn}

\begin{pro}
    Let $R$ be a ring, $i\colon R\to D(R)$ the canonical embedding of $R$ in its Dorroh extension. Suppose there is a ring morphism $f\colon R\to S$ with $S$ a unital ring. Then, there exists a unique unital ring morphism $g\colon D(R)\to S$ such that $gi=f$.
\end{pro}
\begin{proof}
    Simply define $g(x,r)=x1+f(r)$.
\end{proof}

\begin{lem}
    \label{lem:dorojacobson}
    Let $R$ be a ring and consider it as embedded in its Dorroh extension $D(R)$. We have $J(R)=J(D(R))$. In particular, $1+R\subset D(R)^\times$ if and only if $R$ is a radical ring.
\end{lem}
\begin{proof}
    In general $J(D(R))\subset R$ since $J(D(R)/R)=J(\mathbb{Z})=\{ 0\}$ and surjective morphisms preserve the Jacobson radical. Theorem 3.14 of~\cite{MR393090} states that if $I$ is an ideal of a ring $S$ then $J(I)=J(S)\cap I$. In our case, this yields $J(R)=J(D(R))\cap R=J(D(R))$.
\end{proof}

\begin{exa}\label{ex:dorhogamma}
Let $\mathscr{R}_0=\epsilon^{-1}(\{0\})\cap \mathscr{R}$. Since the constant polynomials are rational functions, every element $\gamma\in \mathscr{R}$ can be written in a unique way as a sum $\gamma_n+\gamma_0$ with $\gamma_n\in \Z$ and $\gamma_0\in \mathscr{R}_0$. In fact, $\mathscr{R}$ is the Dorroh extension of $\mathscr{R}_0$.

    It is clear that $1+\mathscr{R}_0\subset \mathscr{R}^\times$ so that $\mathscr{R}_0$ is a radical ring and $\mathscr{R}_0=J(\mathscr{R})$.
\end{exa}

\begin{lem}[see \cite{MR393090}, Theorem 3.15]
\label{lem:radmat}
    Let $R$ be a ring and $M_n(R)$ the ring of square matrices of order $n$ over $R$. Then $J(M_n(R))=M_n(J(R))$.
\end{lem}

\begin{thm}\label{constructiontwosidedbrace}
Let $X$ be a finite non-empty set. Then, $\FB_{\mathcal{T},X}$, the  free two-sided skew brace of abelian type  on $X$,  is~$J(\mathscr{R}(X))$. 
\end{thm}
\begin{proof}
    Let $B$ be a radical ring, and let $g\colon X\to B$ be a map. We consider $X$ as a subset of $\mathscr{R}$ through the canonical inclusion $j\colon X\to \mathscr{R}$. First, we will construct a map that extends $g$ to $\mathscr{R}$. There exists a canonical morphism of unital rings $\mathbb{Z}[F(X)]\to D(B)$, given by 
    \begin{align*}
    x&\mapsto 1+g(x),\\
    x^{-1}&\mapsto (1+g(x))^{-1}.
\end{align*}
The corresponding map $f\colon\Lambda\to D(B)$ is such that $\epsilon|_\Lambda=\epsilon_Df$, where we define $\epsilon_D\colon D(B)\to \mathbb{Z}$ as the ring morphism that maps an element $(z,b)$ to $z$. We claim that $f$ is $\Sigma$-invertible. Consider a matrix $(\lambda_{ij})\in \Sigma$.  Let $M=(f(\lambda_{ij}))$, we have $\det(\epsilon_D(f(\lambda_{ij})))=\det(\epsilon(\lambda_{ij}))\in \{1,-1\}$. Hence, there exists an invertible square matrix $N$ over $\Z$, such that for $(u_{ij})=NM$, $\epsilon(u_{ij})=\delta_{ij}$ where $\delta_{ij}$ is the Kronecker symbol. Hence, $NM=\id + V$ for some matrix $V$ with entries in $B$. Therefore, by~Lem\-ma~\ref{lem:dorojacobson}, the entries of $V$ lie in the Jacobson radical of $D(B)$. Moreover,~Lem\-ma~\ref{lem:radmat} implies that $J(M_n(D(B)))=M_n(B)$, so that $NM$ is invertible and then~$M$ is invertible. This shows that $f$ is $\Sigma$-invertible.

Thus, by Theorem~\ref{thm:cohn}, there exists a unique morphism $\hat{f}\colon \mathscr{R}\to D(B)$ such that $\hat{f}\iota =f$, where $\iota$ is the inclusion $\Lambda\to \mathscr{R}$. Notice that the maps $\epsilon|_\Lambda$ and $\epsilon_Df$ coincide and are $\Sigma$-invertible by definition. Thus, both $\epsilon|_\mathscr{R}$ and $\epsilon_D\hat{f}$ are extensions of $\epsilon_Df$ to~$\mathscr{R}$, which implies that $\epsilon|_\mathscr{R}=\epsilon_D\hat{f}$. This means that $\hat{f}^{-1}(B)=\mathscr{R}_0$, where $\mathscr{R}_0$ is defined as in Example \ref{ex:dorhogamma}.

It is left to check that $\hat{f}|_{\mathscr{R}_0}$ is the unique morphism such that $\hat{f}|_{\mathscr{R}_0}j=g$. Let $h\colon \mathscr{R}_0\to B$ be a ring morphism such that $hj=g$. Hence, using the same notation as in~Example~\ref{ex:dorhogamma}, there is a unique extension $\hat{h}\colon \mathscr{R}\to D(B)$ of $h$ that maps $\gamma_n+\gamma_0$ to $\gamma_n+h(\gamma_0)$. As ring morphisms that coincide on $1+X$, we have $\hat{h}|_{\Lambda}=f$. By the universality of the Cohn localisation, $\hat{h}=\hat{f}$, which proves the uniqueness.
\end{proof}
    
\begin{cor}\label{cortwosidedbrace}
Let $X$ be a finite non-empty set. Then $\FB_{\mathcal{T},X}$, the free two-sided skew brace of abelian type on $X$, is residually finite.
\end{cor}
\begin{proof}
    Fix a well order on $X$ and order the monomials in $X$ following the graded lexicographic order. Let $f\in J(\mathscr{R})$. Thus, $f=\alpha m+\textnormal{other greater terms}$, where $\alpha\in \Z\setminus \{0\}$ and $m$ is some monomial. Choose a prime $p$ that does not divide $\alpha$, and let $B$ be the quotient of~$\F_p\{\{X\}\}$ by the two-sided ideal generated by all the monomials strictly greater than $m$. Since $X$ is finite, the number of monomials smaller than or equal to $m$ is finite. Hence, $B$ is finite and there is a surjective ring homomorphism $\mathscr{R}\to B$ which maps $f$ to a non-zero element.
\end{proof}

\section{Skew braces with commutative multiplication}\label{commutative}

We now consider the category  consisting of the commutative skew braces, denoted $\mathcal{C}$, which is a subcategory  of the category $\mathcal{T}$ consisting of the two-sided skew braces.
Recall that a commutative skew brace is a skew brace whose multiplicative group is abelian. In this setting, we will treat both the classical abelian case and the more general non-abelian type.

\medskip

Let $X$ be a set. In~\cite{LetFreeComm} it was shown that the free commutative skew brace on~$X$ can be obtained as the skew brace of fractions of the free commutative wire generated by $X$. We recall the necessary definitions here.

\begin{defn}
    A {\it commutative wire} is a triple $(W,+, \circ)$ such that $(W,+)$ is a group, $(W,\circ)$ is a commutative monoid and
     \begin{equation}
    \label{eq:leftdistrib}
        u\circ (v+w)=u\circ v-u+u\circ w
    \end{equation}
    holds for all $u,v,w\in W$. If $W,V$ are commutative wires, then a {\it wire morphism} $f\colon W\to V$ is a map that satisfies
    \[
    f(u+v)=f(u)+f(v)\quad \text{and}\quad f(u\circ v)=f(u)\circ f(v)
    \]
    for all $u,v\in W$. 
\end{defn}

\begin{defn}
    Let $(M,\circ)$ be a commutative monoid. On the cartesian product $M\times M$ define the equivalence relation $(u,v)\sim (u_1,v_1)$ if and only if there exists a~\hbox{$k\in M$} such that $u\circ v_1\circ k=u_1\circ v\circ k$. Then, the set 
    \[\operatorname{Q}(M)= (M\times M)/\sim\]
     is a group with operation $\frac{u}{v}\circ \frac{u_1}{v_1}=\frac{u\circ u_1}{v\circ v_1}$, where $\frac{x}{y}$ denotes the equivalence class of~$(x,y)$. This group is called the {\it Grothendieck group} of $M$.
\end{defn}

There is a canonical map $\iota \colon M\to \operatorname{Q}(M)$ that maps $m$ to $\frac{m}{e}$ where $e$ is the neutral element of the monoid $M$. This morphism has the following universal property.

\begin{pro}
    \label{pro:universalgroth}
    Let $(M,\circ)$ be a commutative monoid and $(H,\circ)$ a commutative group. Let $f\colon M\to H$ be a monoid morphism. Then, there exists a unique group morphism $\hat{f}\colon \operatorname{Q}(M)\to H$ such that $\hat{f}\iota=f$.
\end{pro}

The Grothendieck operator allows us to move from commutative wires $(W,+,\circ)$ to commutative skew braces.

\begin{pro}\label{pro:fractions}
 Let $(W,+,\circ)$ be a commutative wire, and let $(\operatorname{Q}(W),\circ)$ denote the Grothendieck group of the monoid $(W,\circ)$. On $\operatorname{Q}(W)$, define the operation
 \[\frac{u}{v}+\frac{u_1}{v_1}= \frac{u\circ v_1 -(v\circ v_1)+ u_1\circ v}{v\circ v_1}.\]
  Then, $(\operatorname{Q}(W),+,\circ)$ is a commutative skew brace. We call this skew brace the \emph{skew brace of fractions} of $W$ and denote it by~$\operatorname{Q}(W)$.
\end{pro}
The universal property of the Grothendieck group extends to the skew braces of fractions.

\begin{pro}\label{pro:universalskbfrac}
Let $W$ be a commutative wire and $B$ a skew brace. Let\linebreak \hbox{$f\colon W\to B$} be a wire morphism. Then, there exists a unique skew brace morphism $\hat{f}\colon \operatorname{Q}(W)\to B$ such that $\hat{f}\iota=f$. This morphism is given by $\hat{f}(\frac{u}{v})=f(u)\circ f(v)'$.
\end{pro}

Many examples of commutative wires can be built from idempotent endomorphisms of commutative rings. The following result is a direct consequence of Proposition 2.17 of~\cite{LetFreeComm}.

\begin{pro}
\label{pro:wiresfromalg}
    Let $(A,+,\cdot)$ be a commutative ring. Let $p,\pi\colon A\to A$ be idempotent ring endomorphisms. Then, for all subgroups $G$ and $H$ of $A^\times$,
    \begin{enumerate}
        \item[\textnormal{(1)}] $(p^{-1}(G) \cap \pi^{-1}(H),+_{_1},\circ)$ is a commutative wire with 
        \[
u +_{_1} v := \pi(v)\cdot u - p(u)\cdot\pi(v) + p(u)\cdot v\quad\textit{and}\quad u\circ v=u\cdot v.
\]
    \item[\textnormal{(2)}] $(p^{-1}(G),+_{_2},\circ)$ is a commutative wire with 
    \[u +_{_2} v := u +p(u)\cdot(v-1)\quad\textit{and}\quad u\circ v=u\cdot v.\]
    \end{enumerate}
\end{pro}

The following theorem shows that the free commutative wire over $X$ belongs to the family described in Proposition~\ref{pro:wiresfromalg}. Let $\Z[\fab(X)]$ be the integral group ring  of  the free abelian group $\fab(X)$ on $X$.
Then, consider the evaluation maps 
\[
\eva_{1}\colon \Z[\fab(X)][(t_x)_{x \in X}] \to \Z[\fab(X)], \quad t_x \mapsto 1
\]
and
\[
\eva_{X}\colon \Z[\fab(X)][(t_x)_{x \in X}] \to \Z[\fab(X)], \quad t_x \mapsto x.
\]
By Proposition~\ref{pro:wiresfromalg}, the set $\eva_{X}^{-1}(\fab(X))\cap \eva_{1}^{-1}(1)$ can be given a structure of commutative wire with
        \[
        f+_{_1} g= f+\eva_X(f)(g-1)\quad \text{and}\quad f\circ g=f\cdot g
        \]
        for all $f,g\in \eva_X^{-1}(\fab(X))\cap \eva_1^{-1}(1)$. Let us denote this wire by $$(\mathcal{W}(X),+_{_1},\circ).$$ Note that if there is no ambiguity, we simply write $+$ instead of $+_{_1}$.

\begin{thm}[see \cite{LetFreeComm}, Theorem 4.21]
    Let $X$ be a non-empty set. The commutative wire $(\mathcal{W}(X),+_{_1},\circ)$ is the free commutative wire on $X$.
\end{thm}

\begin{thm}[see \cite{LetFreeComm}, Theorem 4.23]
\label{thm:freecomskb}
    Let $X$ be a non-empty set. The skew brace $\operatorname{FSB}_{\mathcal{C}, X}$ is $\operatorname{Q}(\mathcal{W}(X))$, the skew brace of fractions of $\mathcal{W}(X)$.
\end{thm}

\begin{cor}
    The free two-sided skew brace on $\{x\}$ is $\operatorname{FSB}_{\mathcal{T}, \{x\}}=\operatorname{Q}(\mathcal{W}(\{x\}))$.
\end{cor}
\begin{proof}
    This is a direct consequence of Proposition~\ref{two sided brace abelian generated is abelian } and Theorem~\ref{thm:freecomskb}.
\end{proof}

\begin{pro}\label{freeabeliangroup}
    Let $X$ be a non-empty set. The multiplicative group of $\operatorname{FSB}_{\mathcal{C}, X}$ is free abelian.
\end{pro}
\begin{proof}
    It is well known that $R=\Z[\fab(X)][(t_x)_{x\in X}]$ is a unique factorisation domain. We claim that the monoid $M=(\eva_1^{-1}(1),\cdot)$ is free commutative. Since the units of $R$ are precisely the elements of
$\pm\fab(X)$, and $\eva_1$ restricts to the identity on $\Z[\fab(X)]$, we have $M^\times=\{1\}$.

An element of $M$ is irreducible in $M$ if and only if it is
irreducible in $R$. Indeed, if $m\in M$ is irreducible in $M$ and $m=ab$ in $R$, then both $\eva_1(a)$ and $\eva_1(b)$ are units. Normalising $a$ and $b$ by their evaluations yields a factorisation of $m$ in $M$; hence one of $a,b$ is a unit of $R$. The converse follows immediately from $R^\times\cap M=\{1\}$.

Now let
\[
m=u m_1\cdots m_k
\]
be a factorisation of $m\in M$ into irreducibles in $R$. Since $\eva_1(m)=1$, each $\eva_1(m_i)$ is a unit, and
\[
n_i=\eva_1(m_i)^{-1}m_i
\]
is an irreducible element of $M$. Moreover,
\[
m=n_1\cdots n_k.
\]
Uniqueness follows from uniqueness in $R$, since two associates
belonging to $M$ must be equal. Thus $M$ is a reduced unique
factorisation monoid, and hence a free commutative monoid.

Consequently, $\operatorname{Q}(M)$ is a free abelian group. The
embedding
\[
(\mathcal{W}(X),\circ)\hookrightarrow M
\]
induces an embedding
\[
(\operatorname{Q}(\mathcal{W}(X)),\circ)
\hookrightarrow \operatorname{Q}(M),
\]
since $M$ is cancellative. Therefore
$(\operatorname{Q}(\mathcal{W}(X)),\circ)$, being a subgroup of a free
abelian group, is itself free abelian. So the result now follows from Theorem~\ref{thm:freecomskb}.
\end{proof}

\begin{lem}
\label{lem:torsionfreewire}
    Let $W$ be a commutative wire with torsion-free additive group and cancellative multiplicative monoid. Then, the group $(\operatorname{Q}(W),+)$ is also torsion-free.
\end{lem}
\begin{proof}
    One can show by induction that \[n\frac{u}{v}=\frac{(n-1)(u-v)+u}{v}\]
    for all integer $n\geq 1$ and $u,v\in W$. Thus, $n\frac{u}{v}=0$ if and only if $n(u-v)=0$, that is $u=v$ since $W$ is torsion-free.
\end{proof}

\begin{pro}\label{additivegrouptorsionfree}
    Let $X$ be a non-empty set. The additive group of $\operatorname{FSB}_{\mathcal{C}, X}$ is torsion-free.
\end{pro}
\begin{proof}
    By Lemma~\ref{lem:torsionfreewire}, it is enough to see that $(\mathcal{W}(X),+_1)$ is torsion-free. If $f\in (\mathcal{W}(X),+_1)$, then $$\underbrace{f+_1\ldots+_1f}_{\textnormal{$n$ times}}= \left(\sum_{k=0}^{n-1}\eva_X(f)^k\right)f-\sum_{k=1}^{n-1}\eva_X(f)^k.$$ Since the sum $\sum_{k=0}^{n-1}\eva_X(f)^k$ is non-zero, we can conclude as long as the degree of $f$ is $>0$. But, the only polynomial in $\mathcal{W}(X)$ of degree $0$ is $1$ which concludes the proof.
\end{proof}
Our goal for the rest of this section is to show that $\operatorname{FSB}_{\mathcal{C}, X}$ is residually finite. For this purpose we are going to embed it in a skew brace that consists of power series.
\begin{exa}
    \label{ex:series}
    Let $R$ be a commutative ring, $Y$ a set and $\epsilon\colon R[[Y]]\to R$ the map that associates to every series its constant coefficient. By Proposition~\ref{pro:wiresfromalg}, for $G\subset R^\times$ a subgroup, the set $\epsilon^{-1}(G)$ has a commutative skew brace structure with
    \[f+_{_2} g= f+\epsilon(f)(g-1)\quad\text{and}\quad f\circ g=f\cdot g.
        \]
\end{exa}
We are interested in the skew brace $U$ obtained in Example~\ref{ex:series} by setting $R=\Z[\fab(X)]$, $Y=\{t_x\mid x\in X\}$ and $G=\fab(X)$. At a first glance, the additive operations of $\operatorname{Q}(\mathcal{W}(X))$ and $U$ don't seem compatible.
However, operating a change in variable, one can have a different description of the free commutative wire on $X$. This time, considering the evaluations \[
\eva_0\colon \Z[\fab(X)][(t_x)_{x \in X}] \to \mathbb{Z}[\fab(X)], \quad t_x \mapsto 0
\]
and
\[
\eva_{1-X}\colon \mathbb{Z}[\fab(X)][(t_x)_{x \in X}] \to \mathbb{Z}[\fab(X)], \quad t_x \mapsto 1-x,
\]
$\eva_{0}^{-1}(\fab(X))\cap \eva_{1-X}^{-1}(1)$ has a structure of commutative wire with
        \[
        f+_{_1} g= f+\eva_{0}(f)(g-1)\quad \text{and}\quad f\circ g=f\cdot g
        \]
        for all $f,g\in \eva_0^{-1}(\fab(X))\cap \eva_{1-X}^{-1}(1)$.
Let us denote this wire by $\mathcal{V}(X)$.       
The wire~$\mathcal{W}(X)$ is isomorphic to $\mathcal{V}(X)$: apply the change of variable $t_x\mapsto t_x+x$. Since the wire $\mathcal{V}(X)$ is a subwire of $B$, by Proposition~\ref{pro:universalskbfrac}, there is a unique embedding $\operatorname{Q}(\mathcal{V}(X))\to U$ mapping $\frac{f}{g}$ to $fg^{-1}$.

\begin{rem}
\label{rem:triviialint}
    $Q(\mathcal{V}(X))\cap\fab(X)=\{1\}$. Indeed, let $a\in Q(\mathcal{V}(X))\cap\fab(X)$, there exist $u,v\in \mathcal{V}(X)$ such that $u=av$. Hence $1=\eva_{1-X}(u)=a\eva_{1-X}(v)=a$.
\end{rem}

\begin{pro}\label{resfincomm}
    The skew brace $\operatorname{FSB}_{\mathcal{C}, X}$ is residually finite for every non-empty set $X$.
\end{pro}
\begin{proof}
   Fix a well-order on $X$ and order the monomials in $X$ following the graded lexicographic order. Let $q$ be a power of a prime number, consider the ring $\F_q[Y]$ with $Y\subset X$ and take the quotient by the ideal generated by all monomials greater than some fixed monomial $m$. Denote by $\eva_0$ the evaluation of polynomials $x\mapsto 0$ for all $x\in Y$. The evaluation map lifts to the quotient. Then define a wire structure on the subset $\eva_0^{-1}(\F_q^\times)$  of the quotient as
    \[   f+_{_2} g= f+\eva_0(f)(g-1)\quad \text{and}\quad f\circ g=f\cdot g.
        \]
    We denote this wire by $\F_{q,m}(Y)$. Notice that every element of $\F_{q,m}(Y)$ is a sum of a unit and a nilpotent element; thus it is a commutative skew brace. 
    
    Let $f\in \operatorname{Q}(\mathcal{V}(X))\setminus\{1\}$ considered as an element of $U$ through the embedding described above. Our goal is to show that there exist a finite skew brace $B$ and a skew brace morphism $\phi\colon \operatorname{Q}(\mathcal{V}(X))\to B$ such that $\phi(f)\neq 0$.
    
    We know that $f\not\in\fab(X)$ because of Remark~\ref{rem:triviialint}. Hence there are two non-zero elements $\alpha\in \fab(X)$, $\beta\in \Z[\fab(X)]$ and a monomial $m$ such that  \[f=\alpha +\beta m+\text{other greater terms}.\] We claim that there exist a power of a prime number $q$ and a morphism $\Z[\fab(X)]\to \F_q$ such that the images of $\alpha$ and $\beta$ are not zero. Indeed, $\alpha$ and $\beta$ are contained in the Laurent polynomial ring in a finite set of variables $\{t_1,\dots, t_n\}\subset X$. That is, 
    \[\alpha =t_1^{z_1}\dots t_n^{z_n}\quad \text{and}\quad \beta=t_1^{-m_1}\dots t_n^{-m_n}g(t_1,\dots ,t_n)\] for $z_1,\dots ,z_n,b\in\Z$, $m_1,\dots ,m_n\in \N$ and some polynomial $g\in \Z[t_1,\dots ,t_n]$. Therefore, it is enough to show that there exist a power of a prime number $q$ and $\alpha_1,\dots ,\alpha_n\in \F_q^\times$ such that $g(\alpha_1,\dots,\alpha_n)\neq 0$. However, for any prime $p$ sufficiently large so that the residue of the coefficients of $g$ are not all zero, we know that $g$ cannot vanish entirely on $\overline{\F_p}^n$ where $\overline{\F_p}$ is the algebraic closure of $\F_p$. Moreover, the set $\left(\overline{\F_p}^\times\right)^n$ 
    is complementary to   the zero locus of $t_1\dots t_n$. Therefore, it is Zariski dense in $\overline{\F_p}^n$. Hence, $g$ does not vanish entirely on $\left(\overline{\F_p}^\times\right)^n$. Since $\overline{\F_p}=\cup_{k\in \N}\F_{p^k}$, this proves the claim. 
    
    Let $Y\subset X$ be the finite set of variables that appear in $m$. Since $Y$ is finite, there are only finitely many monomials smaller than or equal to $m$ in the variables~$Y$. Hence, there is an induced skew brace morphism $\phi\colon \operatorname{Q}(\mathcal{V}(X))\to \F_{q,m}(Y)$ that maps~$f$ to the sum of an element of $\F_q^\times$ and a non trivial nilpotent element. Therefore~$\phi(f)$ belongs to  $\F_{q,m}(Y)\setminus \{1\}$.
\end{proof}

\section{Commutative radical rings}\label{radicalring}

In the particular case of a two-sided skew brace of abelian type with commutative operation $\circ$, which corresponds to \textit{commutative radical rings}, a very explicit description of the free objects is known. For a subset $X$ of a ring $R$ we denote by $\langle X\rangle_{id}$ the ideal of $R$ generated by $X$.

\begin{thm}[Kepka--N{\v{e}}mec, \cite{KN1}, Section 11]  
\label{free radical commutative}
Let $X$ be a non-empty set. Then the set of rational functions
\[
\left\{\frac{f}{1+g}\ \mid\ f, g \in \langle X\rangle_{id} \right\}\subset \Q(X)
\] with the usual addition of functions, the multiplication given by
\begin{align}
\frac{f_1}{1+g_1} \circ \frac{f_2}{1+g_2} &:= \frac{f_1}{1+g_1} + \frac{f_1}{1+g_1} \frac{f_2}{1+g_2} + \frac{f_2}{1+g_2}\notag\\
&= \left(\frac{f_1}{1+g_1} + 1\right) \left(\frac{f_2}{1+g_2}+1\right) -1.\label{E:ProductFractions}
\end{align}
is the free commutative radical ring generated by $X$, i.e. it is $\FB_{\mathcal{C},X}$.
\end{thm}

\begin{rem}\label{inversecomm}
{\rm The multiplicative inverses in description of the free radical ring given in Theorem \ref{free radical commutative} are as follows: \[\overline{\left(\frac{f}{1+g}\right)}=\frac{-f}{1+f+g}.\]}

Note also that the star operation is the usual multiplication of functions.
\end{rem}

\begin{thm}\label{resfiniradicalring}
Let $X$ be a non-empty set. The radical ring $\FB_{\mathcal{C},X}$ is residually finite. 
\end{thm}
\begin{proof}
For any prime $p$, define \[
F_{p,X} := \left\{\frac{f}{1+g}\ \mid\ f, g \in \langle X\rangle_{id} \right\}\subseteq \mathbb{Z}_p(X).
\] Notice that $F_{p,X} =\FB_{\mathcal{C},X}/I$, where $I$ consists of the elements  $\frac{pf}{1+g}$ with $f, g \in \langle X\rangle_{id}$. 
Also, for every positive integer $n$ and for any subset $Y$ of $X$, we let $J_{p,n,Y}$ be the ideal of $F_{p,X}$ generated by all elements of the form $\frac{f}{1+g}$, where $f\in \langle x^n,y\,:x\in X,\, y\in Y\rangle_{id}$.

Now, let $0\neq w\in \FB_{\mathcal{C},X}$. Then for a large enough prime $p$, a large enough positive integer $n$, and a subset $Y$ of $X$ with finite complement, the image of~$w$ in~$F_{p,X}/J_{p,n,Y}$ is non-zero. Hence it follows that $\FB_{\mathcal{C},X}$ is residually finite.
\end{proof}

We give an elementary and self-contained proof of Theorem \ref{free radical commutative}, different from the one in \cite{KN1}.

\begin{lem}\label{free radical commutative mult}
Let $X$ be a non-empty set. Consider the set of rational functions
\[
U_X := \left\{\frac{1+f}{1+g}\ \mid\ f, g \in \langle X\rangle_{id} \right\}\subset \Q(X)
\]
with the usual multiplication. Then one has a group isomorphism
\begin{align*}
\xi \colon (\FB_{\mathcal{C},X},\circ) &\overset{\sim}{\to} (U_X, \cdot)\\
r &\mapsto r+1
\end{align*}
\end{lem}
\begin{proof}
The map $\xi$ is clearly injective. Writing 
\begin{equation}\label{E:xi}
\xi \left( \frac{f}{1+g} \right) = \frac{1+f+g}{1+g},
\end{equation}
one sees that its image is $U_X$. Finally, relation \eqref{E:ProductFractions} can be translated as $\xi(r_1 \circ r_2) = \xi(r_1)\xi(r_2)$, making $\xi$ a group isomorphism.
\end{proof}

\begin{cor}\label{corradicalring}
Let $X$ be a non-empty set. The group $(\FB_{\mathcal{C},X}, \circ)$ is free abelian. 
\end{cor}
\begin{proof}
Since $\Z[X]$ is a unique factorization domain, a polynomial $1+f$ with $f \in \langle X\rangle_{id} $ uniquely decomposes as a product of irreducible factors, all of which can be chosen to be of the form $1+f_i$. Thus, $(U_X,\cdot)$ is the free abelian group on the set of all irreducible polynomials from $1+\langle X\rangle_{id} \subset \Z[X]$. By \cref{free radical commutative mult}, $(\FB_{\mathcal{C},X},\circ)$ is then free abelian as well.   
\end{proof}

The additive group $(\FB_{\mathcal{C},X}, +)$ turns out to be much subtler to deal with.

\begin{question}
 Is $(\FB_{\mathcal{C},X},+)$ a free abelian group for any non-empty set $X$?
\end{question} 

\begin{proof}[Proof of \cref{free radical commutative}]
The three operations (i.e., the operation $\circ$, and the usual product and sum) clearly take values in $\FB_{\mathcal{C},X}$, and are thus well-defined. Also, $(\FB_{\mathcal{C},X}, +)$ is a subgroup of the abelian group $\Q(X)$. \cref{free radical commutative mult} gives a monomorphism 
\begin{align*}
\xi \colon (\FB_{\mathcal{C},X},\circ) &\longrightarrow (\Q^\times, \cdot)\\
r &\longmapsto r+1,
\end{align*} so in particular also $(\FB_{\mathcal{C},X}, \circ)$ is an abelian group. The skew left distributivity follows from\[
\begin{array}{c}
\xi(a\circ(b+c))=\xi(a)\xi(b+c) = \xi(a)(b+c+1)=\xi(a)(\xi(b)-1+\xi(c))\\[0.2cm]
=\xi(a)\xi(b)-\xi(a)+\xi(a)\xi(c)
=\xi(a\circ b)-\xi(a)+\xi(a\circ c)
\end{array}\]
and the injectivity of $\xi$. Summarising, one sees that $(\FB_{\mathcal{C},X}, +, \circ)$ is a commutative radical ring. The formula for the $\circ$-inverse (see Remark \ref{inversecomm}) follows from \eqref{E:xi} and
\[\xi \left( \frac{-f}{1+f+g} \right) = \frac{1+g}{1+f+g},\]
which is the inverse of $\xi \left( \frac{f}{1+g} \right)$ in $(\Q, \cdot)$.

To see that the skew brace $\FB_{\mathcal{C},X}$ is generated by $X$, first note that, under the operations $+$ and $\ast$, the set $X$ generates all polynomials from $\langle X\rangle$. Since, for $f,g \in \langle X\rangle$, one has
\begin{equation}\label{E:FractionsToBraces}
f \circ \overline{g} - \overline{g} = f+f \ast \overline{g} = f+f\cdot \frac{-g}{1+g} = \frac{f}{1+g},
\end{equation}
so the set $X$ generates the whole skew brace $\FB_{\mathcal{C},X}$.

We are left with the freeness statement. Take a commutative radical ring $B$ with chosen elements $a_x,\ x \in X$. One needs to check that there is a skew brace homomorphism $\phi \colon \FB_{\mathcal{C},X} \to B$ mapping each $x$ to $a_x$. Its uniqueness will follow since $X$ generates~$\FB_{\mathcal{C},X}$. First, since the operation $\ast$ is associative in radical rings, the map $\phi$ can be extended from $X$ to $\langle X\rangle_{id}$ by sending the polynomial operations $+$ and $\cdot$ to the skew brace operations $+$ and $*$ in $B$. To extend it further to fractions, by \eqref{E:FractionsToBraces}, we need to put
\[\phi \left( \frac{f}{1+g} \right) = \phi(f) \circ \overline{\phi(g)} - \overline{\phi(g)} = \lambda_{\phi(g)}^{-1}(\phi(f)).\]
To check that this yields a well-defined map on $\FB_{\mathcal{C},X}$, we need to look at two things. First, if a polynomial $f$ is interpreted as a fraction $\frac{f}{1+0}$, our formula for fractions yields
\[\phi \left( \frac{f}{1+0} \right) = \phi(f) \circ \overline{\phi(0)} - \overline{\phi(0)} = \phi(f) \circ 0 - 0 = \phi(f),\]
which is coherent. Second, a presentation of a fraction in the form $\frac{f}{1+g}$ is not unique. To check that the evaluation of $\phi$ is independent of this presentation, for all $f,g,h \in \langle X\rangle_{id}$ we compute 
\begin{align*}
\phi \left( \frac{f(1+h)}{(1+g)(1+h)} \right) &= \phi \left( \frac{f+fh}{1+g+h+gh} \right) = \phi \left( \frac{f \circ h - h}{1+g \circ h} \right) \\
&= \phi(f \circ h - h) \circ \overline{\phi(g\circ h)} - \overline{\phi(g\circ h)} \\
&= \left( \phi(f) \circ \phi(h)-\phi(h) \right) \circ \overline{\phi(h)}\circ \overline{\phi(g)} - \overline{\phi(h)}\circ \overline{\phi(g)}\\
&= \left( \left( \phi(f) \circ \phi(h)-\phi(h) \right) \circ \overline{\phi(h)} - \overline{\phi(h)} \right)\circ \overline{\phi(g)}\\
&= \left( \phi(f) \circ \phi(h)\circ \overline{\phi(h)}- \phi(h)  \circ \overline{\phi(h)} \right)\circ \overline{\phi(g)}\\
&= \phi(f) \circ \overline{\phi(g)} - \overline{\phi(g)}
=\phi \left( \frac{f}{1+g} \right).
\end{align*}
It remains to show that the map $\phi$ thus defined is a skew brace morphism. We have 
\begin{align*}
\phi \left( \frac{f_1}{1+g_1} + \frac{f_2}{1+g_2} \right) 
&= \phi \left( \frac{f_1+f_2+f_1g_2+f_2g_1}{1+g_1 + g_2 +g_1 g_2}\right)\\
&= \phi \left( \frac{f_1\circ g_2+f_2\circ g_1 - g_1-g_2}{1+g_1 \circ g_2}\right)\\
&= \phi(f_1\circ g_2+f_2\circ g_1 - g_1-g_2) \circ \overline{\phi(g_1 \circ g_2)} - \overline{\phi(g_1 \circ g_2)}\\
&= \phi(f_1\circ g_2)\circ \overline{\phi(g_1 \circ g_2)}+\phi(f_2\circ g_1)\circ \overline{\phi(g_1 \circ g_2)} \\
&\qquad - \phi(g_1)\circ \overline{\phi(g_1 \circ g_2)}-\phi(g_2)\circ \overline{\phi(g_1 \circ g_2)}\\
&=\phi(f_1)\circ \overline{\phi(g_1)} +  \phi(f_2)\circ \overline{\phi(g_2)} -\overline{\phi(g_2)} -\overline{\phi(g_1)}\\
&= \phi \left( \frac{f_1}{1+g_1} \right) + \phi \left( \frac{f_2}{1+g_2} \right)
\end{align*}

\medskip

\noindent and

\medskip

\begin{align*}
\phi \left( \frac{f_1}{1+g_1} \ast \frac{f_2}{1+g_2} \right) 
&= \phi \left( \frac{f_1 \ast f_2}{1+g_1 \circ g_2}\right) 
= \lambda^{-1}_{\phi(g_1 \circ g_2)}\left( \phi(f_1 \ast f_2) \right)\\
&= \lambda^{-1}_{\phi(g_1) \circ \phi(g_2)}\left( \phi(f_1) \ast \phi(f_2) \right)
= \lambda^{-1}_{\phi(g_2)}\lambda^{-1}_{\phi(g_1)}\left( \phi(f_1) \ast \phi(f_2) \right)\\
&=  \lambda^{-1}_{\phi(g_1)}(\phi(f_1)) \ast \lambda^{-1}_{\phi(g_2)}(\phi(f_2))= \phi \left( \frac{f_1}{1+g_1}\right) \ast \phi \left(\frac{f_2}{1+g_2} \right), 
\end{align*}
where in the last lines, we used \cref{L:lamda_vs_*} and the commutativity of $\ast$.
\end{proof}

\begin{cor}
    The free two-sided skew brace of abelian type generated by one element $x$, i.e. $\FB_{\mathcal{T},\{ x \}}$, is the skew brace~$\FB_{\mathcal{C},\{x\}}$ from {\rm \cref{free radical commutative}}.
\end{cor}
\begin{proof}
    By \cref{two-sided and abelianity}, every one-generated two-sided skew brace of abelian type is a commutative radical ring. Thus, the free one-generated two-sided skew brace of abelian type is the same thing as the free one-generated commutative radical ring, the latter being explicitly described in \cref{free radical commutative}.    
\end{proof}

\begin{rem}
The subring generated by $X=\{x\}$ in $\FB_{\mathcal{C},\{x\}}$ is $\langle x \rangle_{id}$, which does not coincide with the sub-brace generated by $X$ because it does not contain $\overline{x}=\frac{-x}{1+x}$. 
\end{rem}

A natural question on free objects in algebra is whether their subobjects are free as well.

\begin{pro}\label{nielsen}
   The sub-brace $I$ of $\FB_{\mathcal{C},\{x\}}$ generated by $\{x+x, x\circ x\}$ is not a free commutative radical ring.
\end{pro}
\begin{proof}
Assume that $I$ is the commutative radical ring freely generated by some $r_1, \ldots, r_k$ from $I$. By Theorem \ref{free radical commutative}, this means that:
\begin{enumerate}
 \item\label{I:i1} any $r\in I$ can be written as $r=\frac{f(r_1,\ldots, r_k)}{1+g(r_1,\ldots, r_k)}$ for some polynomials $f$ and $g$ in $k$ variables $x_1,\ldots, x_k$, with $f,g \in \langle x_1,\ldots, x_k \rangle$;
 \item\label{I:i2} for any other such presentation $r=\frac{\overline{f}(r_1,\ldots, r_k)}{1+\overline{g}(r_1,\ldots, r_k)}$ of the same element, $\frac{f(x_1,\ldots, x_k)}{1+g(x_1,\ldots, x_k)}$ and $\frac{\overline{f}(x_1,\ldots, x_k)}{1+\overline{g}(x_1,\ldots, x_k)}$ coincide as rational functions.
\end{enumerate} 
Applying (\ref{I:i1}) to $x+x=2x$ and $x^2=-x+x\circ x-x=-2x+x\circ x \in I$, one gets 
\[2x=\frac{f(r_1,\ldots, r_k)}{1+g(r_1,\ldots, r_k)}, \qquad x^2=\frac{u(r_1,\ldots, r_k)}{1+v(r_1,\ldots, r_k)}.\]
The fractions $\frac{f(x_1,\ldots, x_k)}{1+g(x_1,\ldots, x_k)}$ and $\frac{u(x_1,\ldots, x_k)}{1+v(x_1,\ldots, x_k)}$ may be taken reduced. Then, the element $(2x)^2=4x^2$ has two presentations:
\[\frac{(f(r_1,\ldots, r_k))^2}{(1+g(r_1,\ldots, r_k))^2} = (2x)^2=4x^2 = \frac{4u(r_1,\ldots, r_k)}{1+v(r_1,\ldots, r_k)}.\]
By (\ref{I:i2}), one then has an equality of rational functions
\[\frac{(f(x_1,\ldots, x_k))^2}{(1+g(x_1,\ldots, x_k))^2} =  \frac{4u(x_1,\ldots, x_k)}{1+v(x_1,\ldots, x_k)},\]
where both fractions are reduced. This implies the polynomial equality $f^2=4u$ (the sign indeterminacy disappears because the constant terms of the denominators are of the same sign). But then $f=2h$ for some $h \in \langle x_1,\ldots, x_k \rangle_{id}$. As a result, $I$ contains the element 
\[x=\frac{h(r_1,\ldots, r_k)}{1+g(r_1,\ldots, r_k)}.\]
Since $I$ is a commutative radical ring generated by two elements $2x$ and $x^2$, Theorem~\ref{free radical commutative} allows us to write $x\in I$ as
\[x=\frac{p(2x, x^2)}{1+q(2x, x^2)}\]
for some $p,q \in \langle x_1, x_2\rangle_{id}$. This implies the equality 
\[x(1+q(2x, x^2))=p(2x, x^2) \in \Z [x].\]
The monomial $x$ in the polynomial on the left comes with coefficient $1$, whereas for the polynomial on the right this coefficient is a multiple of $2$. Contradiction.
\end{proof}

Using our explicit description of free commutative radical rings $\FB_{\mathcal{C},X}$, it becomes easy to compute $\FB_{\mathcal{C},X}\ast \FB_{\mathcal{C},X}$ and  $\FB_{\mathcal{C},X}/(\FB_{\mathcal{C},X}\ast \FB_{\mathcal{C},X})$, whose importance was unveiled above. 

\begin{thm}\label{previousproof}
Let $X$ be a non-empty  set and $n\geq 1$ an integer. Then
\begin{equation}\label{E:F_X * F_X}
 \FB_{\mathcal{C},X}^{(n)}= \left\{\frac{f}{1+g}\ \mid\ f \in \langle X\rangle_{id}^n, g \in \langle X\rangle_{id} \right\}.   
\end{equation}
Moreover, the quotient $\FB_{\mathcal{C},X}^{(n)}/\FB_{\mathcal{C},X}^{(n+1)}$ is isomorphic to $\Triv(\fab(M_n))$, where $M_n$ denotes the set of monomials of total degree $n$. In addition, the quotient map can be explicitly written as follows \textnormal:
\begin{align*}
\xi \colon \FB_{\mathcal{C},X}^{(n)} &\to \fab(M_n),\\
\frac{f}{1+g} &\mapsto \bar{f},
\end{align*}
where $f \mapsto \bar{f}$ is the quotient map $\langle X\rangle_{id}^n \twoheadrightarrow \langle X\rangle_{id}^n / \langle X\rangle_{id}^{n+1} \simeq \fab(M_n)$.
\end{thm}
\begin{proof} 
To prove~\eqref{E:F_X * F_X} we proceed by induction on $n$. The case $n=1$ is true by definition. Recall that the operation $*$ on $\FB_{\mathcal{C},X}$ is the usual multiplication of functions. Thus, the $*$-product of two elements of $\FB_{\mathcal{C},X}$ writes as follows:
\[\frac{f_1}{1+g_1} \ast \frac{f_2}{1+g_2} = \frac{f_1f_2}{1+g_1 \circ g_2}.\] If $f_i\in \langle X\rangle_{id}^{n-1}$ and $g_i\in \langle X\rangle_{id}$, then we have $f_1f_2 \in \langle X\rangle_{id}^n$ and $g_1 \circ g_2 \in \langle X\rangle_{id}$. In the opposite direction, an element from the set on the right side of \eqref{E:F_X * F_X} is a linear combination of functions $\frac{f_1f_2}{1+g} = \frac{f_1}{1+g}*f_2 \in \FB_{\mathcal{C},X}^{(n-1)} \ast \FB_{\mathcal{C},X}$ by induction hypothesis.

Finally, the map $\xi$ sends a generator $m \in M_n$ of $\FB_{\mathcal{C},X}^{(n)}$ to the corresponding generator of $\fab(M_n)$, $\ast$-products to $0$ (recall that in the trivial skew brace $\Triv(\fab(M_n))$ all $\ast$-products vanish), and sums to sums, since
\[\frac{f_1}{1+g_1} + \frac{f_2}{1+g_2} = \frac{f_1+f_2+f_1g_2+g_1f_2}{1+g_1 \circ g_2},\]
and $\overline{f_1+f_2+f_1g_2+g_1f_2} = \overline{f_1+f_2}$ as $f_1g_2+g_1f_2 \in \langle X\rangle_{id}^{n+1}$. Thus $\xi$ describes the skew brace quotient map $\FB_{\mathcal{C},X}^{(n)} \twoheadrightarrow \FB_{\mathcal{C},X}^{(n)}/\FB_{\mathcal{C},X}^{(n+1)}$.
\end{proof}

\begin{rem}
{\rm It follows from Theorem \ref{previousproof} that also $\FB_{\mathcal{C},X}^{(n)}/\FB_{\mathcal{C},X}^{(n)}\ast \FB_{\mathcal{C},X}^{(n)}$ is a free abelian group. In fact, the associativity of the $\ast$-product and this result give that $\FB_{\mathcal{C},X}^{(n)}/\FB_{\mathcal{C},X}^{(n)}\ast \FB_{\mathcal{C},X}^{(n)}$ is an abelian group which is extension of free abelian groups, so it is free abelian itself.}
\end{rem}

\section{Free centrally nilpotent skew braces of class \texorpdfstring{$2$}{2}}\label{centrallynilponte}

\noindent The aim of this section is to deal with free objects in the class $\mathcal{CN}_2$ of centrally nilpotent skew braces of class at most $2$. 

\smallskip

Let $I$ be an ordered set, and let $X=\{x_i\,:\, i\in I\}$ be a set indexed by a non-empty set $I$. For each $(u,v)\in I\times I$ define an element~$y_{u,v}$, and let $Y$ be the set of all these elements.

Let $(\Fg_+(Y),+)$ be the free abelian group on $Y$, and let $(H,+)$ be the free nilpotent group of class $2$ on $X$ (if $X$ is a singleton, take $(H,+)$ to be free abelian, and replace class $2$ with class $1$ everywhere in the following). Note that $[H,H]$ is free abelian on the set $\{[u,v]\,:\, (u,v)\in X(2)\}$, where $X(2)$ is a subset of $X\times X$ that contains a unique pair $(u,v)$ for each choice of two distinct elements~$u$ and $v$ of~$X$. 
Consider the direct product 
 $$(\FSB_{\mathcal{CN}_2,X},+)=(\Fg_+(Y),+)\times (H,+).$$  
In order to define a skew brace structure on $\FSB_{\mathcal{CN}_2,X}$, we need to define the~\hbox{$\lambda$-func}\-tion. Thus, if $w\in \FSB_{\mathcal{CN}_2,X}$, then it may be written uniquely in the form
$$w=\varepsilon_1 x_{i_1}+\ldots+\varepsilon_nx_{i_n}+c+v,$$ where $c\in [H,H]$, $v\in \Fg_+(Y)$, $\varepsilon_i \in \mathbb{Z}$ for every $i\in\{1,\ldots,n\}$, and $i_1<\ldots<i_n$.

We put $$\lambda_w(y)=y \text{ for every } y\in Y$$
and $$\lambda_w(x_j)=\lambda_{x_{i_1}}^{\varepsilon_1}\ldots \lambda_{x_{i_n}}^{\varepsilon_n}(x_j)=x_j+\varepsilon_1y_{i_{1},j}+\ldots+\varepsilon_ny_{i_{n},j}$$ for every $x_j\in X$. So we obtain two mappings $$\lambda_w \colon Y \to \Fg_+(Y)\quad\textnormal{ and  }\quad\lambda_w \colon X\to H\times \Fg_+(Y).$$
The former obviously extends to a homomorphism  $\Fg_+(Y)\to   \Fg_+(Y) \times H$. The latter extends  to a homomorphism $H\to \Fg_+(Y) \times H$  because 
$H$ is free nilpotent on $X$ of class $2$ and $H\times \Fg_+(Y)$ is nilpotent of class $2$.
Since the  images of these maps commute, these maps naturally yield a homomorphism, which we also denote by~$\lambda_w$:
  $$\lambda_w \colon \FSB_{\mathcal{CN}_2,X} \to \FSB_{\mathcal{CN}_2,X}.$$
Note that $\lambda_w (v)\!=\!v$ for any  $v\!\in\! \Fg_+(Y)\cup[H,H]$. It easily yields  $\lambda_w\!\in\! \Aut (\FSB_{\mathcal{CN}_2,X})$.
If $h\in H$ and $v\in \Fg_+(Y)$ then $\lambda_{h+v}=\lambda_h$. Hence, we obtain, for $a,b\in \FSB_{\mathcal{CN}_2,X}$, $$\lambda_{a+\lambda_a(b)}=\lambda_{a+b}=\lambda_a\lambda_b.$$ Thus, $\FSB_{\mathcal{CN}_2,X}$ has a skew brace structure $(\FSB_{\mathcal{CN}_2,X},+,\circ)$ with $a\circ b = a+\lambda_a(b)$.

\begin{rem}
{\rm Clearly, $\FSB_{\mathcal{CN}_2,X}\ast \FSB_{\mathcal{CN}_2,X}$ and $[\FSB_{\mathcal{CN}_2,X},\FSB_{\mathcal{CN}_2,X}]_+$ are free abelian groups, while $\FSB_{\mathcal{CN}_2,X}/\FSB_{\mathcal{CN}_2,X}\ast \FSB_{\mathcal{CN}_2,X}$ is the free nilpotent group of class $2$.} 
\end{rem}

\begin{thm}\label{theofreecentralynilp}
Let $X=\{x_i\,:\, i\in I\}$ be a non-empty set. Then the skew brace $(\FSB_{\mathcal{CN}_2,X},+,\circ)$ is the free centrally nilpotent skew brace of class $2$ on $X$.
\end{thm}
\begin{proof}
Let $(B,+,\circ)$ be a centrally nilpotent skew brace of class at most $2$, and let $\pi:X\rightarrow B$ be any map. Write $(\FSB_{\mathcal{CN}_2,X},+)=(H,+)\times (\Fg_+(Y),+)$, where $H$ and~$Y$ are defined as explained above. Since $H$ is the free nilpotent group of class~$2$ on~$X$, we can extend $\pi$ to a homomorphism $\tau$ from $(H,+)$ to $(B,+)$ (this maps any commutator in the generators $X$ of $(H,+)$ to the corresponding one in $(B,+)$). Moreover, since~$\Fg_+(Y)$ is free abelian on $Y$ and $B*B$ is a trivial brace, we may define a homomorphism 
$$\sigma \colon  \Fg_+(Y) \to B$$ that maps every $y_{i,j}$, $i,j\in I$, to $\pi(x_i)\ast\pi(x_j)$. Finally, since $\sigma(\Fg_+(Y))$ lies in the centre of $B$, we may define a homomorphism 
 $$\theta \colon (\FSB_{\mathcal{CN}_2,X},+) \to (B,+)$$ extending $\tau$ and $\sigma$ (and so also $\pi$).

We need to show that $\theta$ is a skew brace homomorphism and therefore that   $\theta$ preserves the product. To see this, first notice that if $x_i,x_j\in X$, then $$
\begin{array}{c}
\theta(\lambda_{x_i}(x_j))=\theta(x_j+y_{i,j})=\theta(x_j)+\theta(y_{i,j})\\[0.2cm]
=\theta(x_j)+\theta(x_i)\ast\theta(x_j)=\lambda_{\theta(x_i)}(\theta(x_j)).
\end{array}
$$ This easily implies by the definitions,  for every $u,v\in \FSB_{\mathcal{CN}_2,X}$, we have $$
\begin{array}{c}
\theta(\lambda_u(v))=\lambda_{\theta(u)}(\theta(v)).
\end{array}
$$ 
As $u\circ v=u+\lambda_u(v)$, it follows that $\theta (u\circ v) =\theta (u) \circ \theta (v)$, as desired. Hence, the mapping $\pi \colon X\to B$ can be extended to a skew brace homomorphism $\theta \colon \FSB_{\mathcal{CN}_2,X}\to B$ and the result thus follows.
\end{proof}

\begin{cor} \label{FreeBrace}
Let $X$ be a non-empty set. Then $$(\FSB_{\mathcal{CN}_2,X}/[\FSB_{\mathcal{CN}_2,X},\FSB_{\mathcal{CN}_2,X}]_+,+,\circ)$$ is the free centrally nilpotent skew brace of abelian type of class $2$ on $X$.
\end{cor}
\begin{proof}
Note that since $[\FSB_{\mathcal{CN}_2,X},\FSB_{\mathcal{CN}_2,X}]_+ \subseteq \ker(\lambda)$, it is an ideal. Hence, $(\FSB_{\mathcal{CN}_2,X}/[\FSB_{\mathcal{CN}_2,X},\FSB_{\mathcal{CN}_2,X}]_+,+,\circ)$ is a centrally nilpotent skew brace of abelian type of class $2$. If $(B,+,\circ)$ is any centrally nilpotent skew brace of abelian type of class $2$, and $\varepsilon:X\rightarrow B$ is any map, then there exists a unique homomorphism $\pi:\FSB_{\mathcal{CN}_2,X}\rightarrow B$ extending $X$ by~The\-o\-rem~\ref{theofreecentralynilp}. Clearly, $[\FSB_{\mathcal{CN}_2,X},\FSB_{\mathcal{CN}_2,X}]_+$ lies in the kernel of such homomorphism and so $\pi$ extends to a homomorphism of $\FSB_{\mathcal{CN}_2,X}/[\FSB_{\mathcal{CN}_2,X},\FSB_{\mathcal{CN}_2,X}]_+$ to $B$. Finally, since every homomorphism of the latter type can be lifted to a homomorphism of $\FSB_{\mathcal{CN}_2,X}$ to $B$, we obtain the uniqueness requirement.
\end{proof}

\begin{rem}
{\rm In order to construct the free centrally nilpotent {\it skew brace of abelian type} of class $2$, we could have just followed the same arguments as above, replacing $H$ by a free abelian group on $X$.}
\end{rem}

\begin{thm}\label{resfinfinall}
Let $X$ be any non-empty set. Then both skew braces  $\FSB_{\mathcal{CN}_2,X}$ and $\FSB_{\mathcal{CN}_2,X}/[\FSB_{\mathcal{CN}_2,X},\FSB_{\mathcal{CN}_2,X}]_+$ are residually finite. 
\end{thm}
\begin{proof}
Write 
\[
(\FSB_{\mathcal{CN}_2,X},+)=(H,+)\times (\Fg_+(Y),+),
\]
where $(H,+)$ and $(\Fg_+(Y),+)$ are as in the construction at the beginning of the section. Let $w\in \FSB_{\mathcal{CN}_2,X}$. Then $w$ has a unique expression as $x+d+c$, where $x\in X$, $d\in [H,H]_+$, and \hbox{$c\in \Fg_+(Y)$.} Let $W$ be the finite subset of $X$ made up of all the elements necessary to define $x$, $d$, and~$c$. By the universal property, $\FSB_{\mathcal{CN}_2,X}$ has an epimorphic image which is isomorphic to~$\FSB_{\mathcal{CN}_2,X}$ and in which the image of $w$ is non-zero. Thus, in order to prove that~$\FSB_{\mathcal{CN}_2,X}$ is residually finite, we only need to prove that~$\FSB_{\mathcal{CN}_2,X}$ is such. However, being~$F_W$ finitely generated and centrally nilpotent, this follows from~\cite[Corollary~3.25]{Wordproblem}.

The proof in the case of $\FSB_{\mathcal{CN}_2,X}/[\FSB_{\mathcal{CN}_2,X},\FSB_{\mathcal{CN}_2,X}]_+$ is essentially the same. 
\end{proof}

Let us examine the particular case in which $X=\{x\}$. Here, for clarity sake, we simply write $F_x$ instead of $\FSB_{\mathcal{CN}_2,X}$. First, we observe that in this case both $(F_x,+)$ and~$(F_x,\circ)$ are always abelian. Moreover,  from the construction it follows that $F_x\ast F_x=\langle x\ast x\rangle_+$. Actually, this can be seen directly as follows: clearly, \hbox{$F_x\ast F_x\leq Z(F_x)$,} and thus  $\langle x\ast x\rangle_+$ is an ideal of $F_x$; therefore, Lemma 2.1 of \cite{Dedekind} shows that $F_x/\langle x\ast x\rangle_+$ is a trivial skew brace, and hence   $F_x\ast F_x=\langle x\ast x\rangle_+$. In particular,~$F_x$ is an extension of an infinite cyclic group by an infinite cyclic group.  It is worth remarking that not every non-trivial sub-skew brace of $F_x$ is free of rank~$1$. In fact, for example, the sub-skew brace $C$ of $F_x$ generated by $x^2$ and $x\ast x$ is such that $\langle x\ast x\rangle=Z(C)>C\ast C=\langle x^2\ast x^2\rangle=\langle 4(x\ast x)\rangle$. (Note that since $F_x$ is centrally nilpotent of class $2$ we can write $x^2=2x+c$ for some $c\in Z(F_x)$, and hence $x^2\ast x^2=2(x\ast x^2)=2(x\ast (2x+c))=4(x\ast x)$.) But in our construction of the free centrally nilpotent skew brace $F_x$ of class $2$ on one generator, the centre coincides with $F_x\ast F_x$, so the previous sub-skew brace cannot be free of rank $1$. Nevertheless, we have the following result.

\begin{thm}
Let $F_x=\FSB_{\mathcal{CN}_2,\{x\}}$.
Then every non-trivial one-generated sub-skew brace of $F_x$ is isomorphic to $F_x$.
\end{thm}
\begin{proof}
Let $C=\langle u\rangle$ be a non-trivial one-generator sub-skew brace of $F_x$. In particular, $C\ast C\neq\{0\}$.  Then $\langle u\ast u\rangle_+=\langle u\ast u\rangle$ is a non-zero ideal of $C$. So, $(\langle u\rangle,+)=\langle u\rangle_+\oplus\langle u\ast u\rangle_+$ because $F_x$ is centrally nilpotent of class $2$, and both $u$ and $u\ast u$ are elements of infinite additive order. Thus, the map defined by mapping~$x$ to $u$ is an isomorphism of $F_x$ and $\langle u\rangle$ because every non-zero element of $F_x$ has a unique form as a sum of an element of $\langle x\rangle_+$ and an element of~$\langle x\ast x\rangle_+$.
\end{proof}

Let $(Q_1,+)$ and $(Q_2,+)$ be copies of the additive group of the rational numbers, and let $(Q,+)$ be their direct product. Let $\varphi_i:(Q_i,+)\rightarrow (\mathbb Q,+)$ be isomorphisms for $i\in\{1,2\}$. In order to define a skew brace structure on $Q$ we need to define the $\lambda$-map. Let $x=x_1+x_2,\; u=u_1+u_2\in Q$,  where $x_1,u_1\in Q_1$ and $x_2,u_2\in Q_2$. Also, write $\varphi_1(u_1)=m_1/n_1$ and $\varphi_1(x_1)=m_1'/n_1'$, 
where $n_1,n_2>0$, and $(m_1,n_1)=1=(m_1',n_1')$ if $m_1,m_1'\neq0$. Define $$\lambda_{x}(u)=\lambda_{x_1}(u)=u_1+\varphi_2^{-1}(m_1m_1'/n_1n_1')+u_2.$$  If $a,b\in Q$, then we can  easily see that $\lambda_{a+\lambda_a(b)}(w)=\lambda_{a+b}(w)=\lambda_a(\lambda_b(w))$, so $Q$ can be given a skew brace structure $(Q,+,\circ)$.

\begin{thm}\label{risultatofinale}
The skew brace $(Q,+,\circ)$ is centrally nilpotent of class $2$, locally free of rank $1$, but it is not finitely generated.
\end{thm}
\begin{proof}
By construction, $Q_2=Z(Q)$ and $Z(Q/Z(Q))=Q/Z(Q)$, so $Q$ is centrally nilpotent of class $2$. Thus, $Q$ is not finitely generated because $Q/Z(Q)$ is a non-finitely generated group.

Let $E$ be any finite subset of $Q$. Recall that $E$ is made by elements of the form~\hbox{$a+b$,} where $a\in Q_1$ and $b\in Q_2$; let $F_1$ be the subgroup of $Q_1$ generated by all such $a$'s, and let $F_2$ be the subgroup of $Q_2$ generated by all $b$'s. Then $F_1$ and~$F_2$ are finitely generated and so they are (infinite) cyclic groups. It is now possible to find $E_1=\langle u\rangle_+\leq (Q_1,+)$ and $E_2=\langle v\rangle_+\leq (Q_2,+)$ such that $F_1\leq E_1$, $F_2\leq E_2$, and $\lambda_u(u)=u+v$ (it is enough to choose them in such a way that their images in $\mathbb Q$ under the $\varphi_1$ and $\varphi_2$ coincide). Therefore, the additive group of the sub-skew brace $\langle u\rangle$ is $\langle u\rangle_+\times\langle v\rangle_+$ and the $\lambda$-function is uniquely defined by $\lambda_u(u)=u+v$. This structure coincides with that of the free centrally nilpotent skew of class $2$ on one generator, so $\langle u\rangle$ is free centrally nilpotent skew of class $2$ on one generator.
\end{proof}

\begin{thm}
The multiplicative group of the free centrally nilpotent skew brace of class $2$ of abelian type \textnormal(resp. the free centrally nilpotent skew brace of class 2 \textnormal) on a non-empty set $X$, namely $\FB_{\mathcal{CN}_2,X}$ \textnormal(resp. $\FSB_{\mathcal{CN}_2,X}$\textnormal)  is isomorphic to the additive group.
\end{thm}
\begin{proof} By Corollary~\ref{FreeBrace}, $$(\FB_{\mathcal{CN}_2,X},+,\circ)=(\FSB_{\mathcal{CN}_2,X}/[\FSB_{\mathcal{CN}_2,X},\FSB_{\mathcal{CN}_2,X}]_+,+,\circ).$$
In particular, $(\FSB_{\mathcal{CN}_2,X}/[\FSB_{\mathcal{CN}_2,X},\FSB_{\mathcal{CN}_2,X}]_+,+)$ is the free abelian group on~\hbox{$Y\cup X$.}
If $a,b$ are elements of $\FB_{\mathcal{CN}_2,X}$, then 
\begin{align*}
b^{-1}\circ a^{-1}&\circ b\circ a
=(a\circ b)^{-1}\circ b\circ a
=(a\circ b)^{-1}\circ (b+a+a\ast b)\\
&=(a\circ b)^{-1}\circ b-(a\circ b)^{-1}+(a\circ b)^{-1}\circ(a+a\ast b)\\
&=(a\circ b)^{-1}\circ b-(a\circ b)^{-1}+(a\circ b)^{-1}\circ a-(a\circ b)^{-1}+(a\circ b)^{-1}\circ (a\ast b)\\
&=(a\circ b)^{-1}\circ b-(a\circ b)^{-1}+(a\circ b)^{-1}\circ a-(a\circ b)^{-1}+(a\circ b)^{-1}+ a\ast b\\
&=(a\circ b)^{-1}\circ b-(a\circ b)^{-1}+(a\circ b)^{-1}\circ a+ a\ast b\\
&=b^{-1}\circ a^{-1}\circ b-(a\circ b)^{-1}+b^{-1}+ a\ast b\\
&=b^{-1}\circ (a^{-1}+b+a^{-1}\ast b)-(a\circ b)^{-1}+b^{-1}+ a\ast b\\
&=b^{-1}\circ (a^{-1}+b-a\ast b)-(a\circ b)^{-1}+b^{-1}+ a\ast b\\
&=b^{-1}\circ a^{-1}-b^{-1}+b^{-1}\circ(b-a\ast b)
-(a\circ b)^{-1}+b^{-1}+a\ast b\\
&=b^{-1}\circ a^{-1}+b^{-1}\circ(b-a\ast b)
-(a\circ b)^{-1}+a\ast b\\
&=b^{-1}\circ a^{-1}-b^{-1}+b^{-1}-a\ast b-(a\circ b)^{-1}+a\ast b\\
&=-b^{-1}+b^{-1}-a\ast b+a\ast b\\
&=0
\end{align*}

\noindent(Here we have used that $(a\circ b)^{-1}\circ (a\ast b)=(a\circ b)^{-1}+ a\ast b$ for all $a,b\in B$ because $a\ast b\in Z(\FB_{\mathcal{CN}_2,X})$.)

This shows that $(\FB_{\mathcal{CN}_2,X},\circ)$ is abelian. The fact that it is free abelian on $X\cup Y$ follows from the fact that both $(\Fg_+(Y),\circ )$ and $(\FB_{\mathcal{CN}_2,X}/\Fg_+(Y),\circ)$ are free abelian groups on $Y$ and (the image of) $X$, respectively. This implies that the multiplicative group of~$\FB_{\mathcal{CN}_2,X}$ is isomorphic with the additive one.

\medskip

Let us  now consider the free centrally nilpotent skew brace $\FSB_{\mathcal{CN}_2,X}$ of class~$2$. By the above, the multiplicative group of $$\FSB_{\mathcal{CN}_2,X}/[\FSB_{\mathcal{CN}_2,X},\FSB_{\mathcal{CN}_2,X}]_+$$ can be written as the direct product of the image of $\Fg_+(Y)$ and a free abelian group $U/[\FSB_{\mathcal{CN}_2,X},\FSB_{\mathcal{CN}_2,X}]_+$ generated by the image of $X$ multiplicatively; of course we may assume that $[\FSB_{\mathcal{CN}_2,X},\FSB_{\mathcal{CN}_2,X}]_+\subseteq U$. As $[\FSB_{\mathcal{CN}_2,X},\FSB_{\mathcal{CN}_2,X}]_+\leq Z(B)$, we obtain that $U$ is an ideal of the skew brace $\FSB_{\mathcal{CN}_2,X}$. Since also \hbox{$U\cap \Fg_+(Y)=\{ 0\}$,} it follows that $(\FSB_{\mathcal{CN}_2,X},\circ)=(U,\circ)\times (\Fg_+(Y),\circ)$. 
Since $$(\FSB_{\mathcal{CN}_2,X}/\Fg_+(Y),+)=(\FSB_{\mathcal{CN}_2,X}/\Fg_+(Y),\circ)\simeq (U,\circ)$$ is a free nilpotent group of class~$2$ on $X$, and $(\Fg_+(Y),+)=(\Fg_+(Y),\circ)$, the statement is proved.
\end{proof}

\section*{Acknowledgments}
Properzi, Trombetti, and Van Antwerpen are members of the non-profit association \href{www.advgrouptheory.com}{AGTA --- Advances in Group Theory and Applications}.  Properzi is supported by Fonds Wetenschappelijk Onderzoek - Vlaanderen, via a PhD Fellowship for fundamental research, grant 11PIO24N. 
Trombetti is supported by GNSAGA (INdAM).
Letourmy is supported by FNRS via an ASP grant.

\smallskip

This work originated from the question “Is there an analogue of the Nielsen--Schreier theorem for skew braces?” posed during the Banff workshop ‘‘Skew Braces, Braids and the Yang–Baxter Equation’’ (24w5201, May 5--10, 2024). 
It was subsequently carried out through a focused research group at the Banach Center in B\c edlewo (January 26--31, 2026) under the same title, and we gratefully acknowledge the financial support and conducive working environment of the Banach Center and Banff International Research Station. We also gratefully acknowledge Be'eri Greenfeld and Victoria Lebed for many useful discussions and insights that contributed to this work.

\bibliography{refs}
\bibliographystyle{abbrv}
\end{document}